\documentclass[twoside,11pt]{article}

\newif\ifarxivversion
\arxivversiontrue
\ifarxivversion
  \usepackage[preprint]{jmlr2e}
\else
  \usepackage{jmlr2e}
\fi
\usepackage{placeins}
\usepackage{needspace}
\usepackage{xcolor}

\usepackage{amsmath}
\usepackage{mathtools}
\usepackage{bm}
\usepackage{booktabs}
\usepackage{multirow}
\usepackage{algorithm}
\usepackage{algorithmic}
\usepackage{lastpage}
\usepackage[capitalise,nameinlink]{cleveref}
\hypersetup{
  hidelinks,
  hypertexnames=false,
  pdftitle={Retraction-Free Nonsmooth Optimization on the Stiefel Manifold: Subgradient and Proximal Landing Methods},
  pdfauthor={Zilong Cui, Tongxin Hu, Bin Gao, and Ran Gu},
  pdfsubject={Retraction-free methods for nonsmooth composite optimization on the Stiefel manifold},
  pdfkeywords={nonsmooth optimization, Stiefel manifold, retraction-free methods, subgradient and proximal methods, Landing methods}
}

\graphicspath{{figures/}}
\providecommand{\R}{}
\renewcommand{\R}{\mathbb{R}}

\providecommand{\St}{}
\renewcommand{\St}{\operatorname{St}}
\providecommand{\sym}{}
\renewcommand{\sym}{\operatorname{sym}}
\providecommand{\skew}{}
\renewcommand{\skew}{\operatorname{skew}}
\providecommand{\prox}{}
\renewcommand{\prox}{\operatorname{prox}}
\providecommand{\dist}{}
\renewcommand{\dist}{\operatorname{dist}}
\providecommand{\rank}{}
\renewcommand{\rank}{\operatorname{rank}}

\providecommand{\argmin}{}
\renewcommand{\argmin}{\operatorname*{arg\,min}}

\newcommand{\inner}[2]{\left\langle #1,#2\right\rangle}
\newcommand{\proofstep}[1]{%
  \par\medskip{\setlength{\parindent}{0pt}\emph{#1}\par}\smallskip}
 
\newtheorem{assumption}{Assumption}
\crefname{assumption}{Assumption}{Assumptions}
\crefname{algorithm}{Algorithm}{Algorithms}
\crefname{theorem}{Theorem}{Theorems}
\crefname{lemma}{Lemma}{Lemmas}
\crefname{proposition}{Proposition}{Propositions}
\crefname{corollary}{Corollary}{Corollaries}
\crefname{definition}{Definition}{Definitions}
\crefname{remark}{Remark}{Remarks}
\crefname{appendix}{Appendix}{Appendices}
\Crefname{appendix}{Appendix}{Appendices}

\usepackage{etoolbox}
\pretocmd{\footnoterule}{\ifnum\value{page}=1\relax\vfill\fi}{}%
  {\PackageError{titlefootnote}{Cannot position title-page footnote}{Check footnote formatting.}}
\makeatletter
\newcommand{\prooflinktarget}{}
\newcommand{\prooflink}[1]{\gdef\prooflinktarget{#1}}
\newcommand{\prooflinkedheading}[1]{%
  \ifdefempty{\prooflinktarget}{#1}{%
    \hyperref[\prooflinktarget]{#1}%
    \global\let\prooflinktarget\@empty}}
\patchcmd{\@begintheorem}{#1\ #2}{\prooflinkedheading{#1\ #2}}{}%
  {\PackageError{prooflinks}{Cannot patch theorem heading}{Check theorem formatting.}}
\makeatother

\jmlrheading
  {}
  {2026}
  {1--\pageref{LastPage}}
  {}
  {}
  {}
  {Zilong Cui, Tongxin Hu, Bin Gao, and Ran Gu}

\ifarxivversion\else
\makeatletter
\apptocmd{\ps@jmlrtps}{%
  \def\@oddhead{\scriptsize Journal of Machine Learning Research\hfill Submission for review}%
  \def\@oddfoot{\parbox[t]{\textwidth}{\raggedright\scriptsize
    \copyright 2026 Zilong Cui, Tongxin Hu, Bin Gao, and Ran Gu.\\[5pt]
    License: CC-BY 4.0, see \url{https://creativecommons.org/licenses/by/4.0/}.\hfill}}%
}{}{\PackageError{submissionheader}{Cannot set submission header}{Check JMLR page style.}}
\makeatother
\fi

\ShortHeadings
  {Nonsmooth Landing on the Stiefel Manifold}
  {Cui, Hu, Gao, and Gu}

\firstpageno{1}

\begin{document}
\raggedbottom

\title{Retraction-Free Nonsmooth Optimization on the Stiefel Manifold:
Subgradient and Proximal Landing Methods}

\author{%
\name Zilong Cui \email zl\_c@mail.nankai.edu.cn \\
\addr NITFID, School of Statistics and Data Science\\
Nankai University\\
Tianjin 300071, China
\AND
\name Tongxin Hu \email hutongxin26@mails.ucas.ac.cn \\
\addr 
University of Chinese Academy of Sciences\\
State Key Laboratory of Mathematical Sciences,
    Academy of Mathematics and Systems Science\\ 
    Chinese Academy of Sciences\\
    Beijing, 100190, China
\AND
\name Bin Gao \email gaobin@lsec.cc.ac.cn \\
\addr State Key Laboratory of Mathematical Sciences,
    Academy of Mathematics and Systems Science\\ 
    Chinese Academy of Sciences\\
Beijing, 100190, China
\AND
\name Ran Gu\thanks{Corresponding author.} \email rgu@nankai.edu.cn \\
\addr NITFID, School of Statistics and Data Science, LPMC, KLMDASR, LEBPS and AAIS \\Nankai University\\Tianjin, 300071, China
}

\editor{Under Review}

\begin{NoHyper}
\maketitle
\end{NoHyper}

\begin{abstract}%
We develop three retraction-free Landing methods for nonsmooth composite optimization over the Stiefel manifold, combining an (approximately) tangential direction with an infeasibility-reducing normal direction.
The subgradient Landing method, with normal stepsizes parameterized by $\alpha\in[0,1]$, has best-iterate complexity bounds $\mathcal O(\epsilon^{-4})$ for surrogate stationarity and $\mathcal O(\epsilon^{-2/(2-\alpha)})$ for feasibility.
For $\alpha=0$, it converges locally $R$-linearly under local weak sharpness and mild assumptions.
The proximal gradient Landing method computes a tangential direction by solving a proximal subproblem and combines it with the same normal direction.
Moreover, the inexact proximal gradient Landing method allows the proximal subproblem to be solved approximately and applies an explicit correction to control violations of the tangency condition.
Both proximal methods have outer-iteration complexity bounds $\mathcal O(\epsilon^{-2})$ for stationarity and $\mathcal O(\epsilon^{-1})$ for feasibility, and every accumulation point is first-order stationary.
Under an additional asymptotic Bregman-error condition and the Kurdyka--{\L}ojasiewicz property of a suitable auxiliary function, the primal sequence of the inexact method converges to a first-order stationary point.
Numerical experiments illustrate the effectiveness of all three methods. On the tested sparse principal component analysis instances, the inexact method has favorable running times, and both proximal methods maintain small feasibility residuals.
\end{abstract}


\section{Introduction}
\label{sec:introduction}

We consider nonsmooth composite optimization over the Stiefel manifold:
\begin{equation}\label{problem}
\min_{X\in\St(n,p)} f(X):=h(X)+r(X),
\quad
\St(n,p):=\{X\in\R^{n\times p}:X^\top X=I_p\},
\end{equation}
where \(n\ge p\), \(h:\R^{n\times p}\to\R\) is continuously differentiable
and possibly nonconvex, and \(r:\R^{n\times p}\to\R\) is a convex,
possibly nonsmooth function. Applications of this model in data analysis
and scientific computing include sparse principal component analysis
\citep{jolliffe2003modified}, compressed-mode computation
\citep{ozolicnvs2013compressed}, orthogonality-constrained unsupervised
feature selection \citep{yang2011}, orthogonal dictionary learning
\citep{li2021weakly}, and least-absolute-deviation formulations of robust
subspace recovery \citep{lerman2018overview,li2021weakly}.

\subsection{Related Work}
\label{sec:related-work}

We review methods for problem~\eqref{problem} that maintain feasibility,
use variable splitting, or avoid orthogonalization at every iteration.

Riemannian subgradient and proximal gradient methods
maintain feasibility by combining tangent-space search directions with
retractions.
 \citet{li2021weakly} establish global iteration complexity bounds for the Riemannian subgradient method (RieSub) under weak convexity, as well as local \(R\)-linear convergence under a local weak sharpness condition.
The manifold proximal gradient
 method (ManPG) of \citet{chen2020proximal} exploits the composite structure by solving a convex proximal subproblem over the tangent space of the Stiefel manifold.
 \citet{huang2022riemannian} develop a proximal gradient framework for general Riemannian manifolds.  Its subsequent inexact extension \citep{huang2023inexact} allows approximate solutions of the proximal subproblems.

Variable-splitting methods separate the treatment of the nonsmooth term
from that of the orthogonality constraints by introducing auxiliary
variables linked through equality constraints. The splitting orthogonality constraints (SOC) method
\citep{lai2014splitting} and the proximal alternating minimized augmented
Lagrangian (PAMAL) method \citep{chen2016augmented}  split the problem by introducing an auxiliary variable and enforce orthogonality through projection.
In the manifold alternating direction method of multipliers
(MADMM) \citep{kovnatsky2016madmm} and the Riemannian alternating direction
method of multipliers (RADMM) \citep{li2025riemannian}, auxiliary variables
 are used to handle the nonsmooth term, while the orthogonal variable is updated using manifold optimization steps.
Beyond these splitting schemes, manifold augmented Lagrangian methods
\citep{zhou2021semismooth,deng2023manifold,deng2025oracle,xu2025oracle}
 are developed with convergence and oracle complexity guarantees.

The above methods use retractions or projections
at every iteration to enforce orthogonality on the original variable or
an auxiliary variable. These operations can incur substantial computational
cost, particularly when objective and gradient evaluations are inexpensive.
This motivates methods that allow infeasible
iterates without requiring orthogonalization at every iteration.
For smooth problems, \citet{gao2019parallelizable} develop PLAM and PCAL, which allow infeasible iterates and avoid orthonormalization during the main iterations. \citet{xiao2021exact} develop an
exact-penalty approach for \(\ell_{2,1}\)-regularized optimization over the
Stiefel manifold, while \citet{hu2024constraint} propose a
constraint-dissolving reformulation for nonsmooth optimization on this
manifold. The sequential linearized
proximal gradient method (SLPG)~\citep{liu2024penaltyfree}
instead alternates between tangential proximal steps and normal steps to improve optimality and feasibility, respectively, and allows inexact computation of the tangential steps.
The linearized smoothing augmented
 Lagrangian method (LSALM)~\citep{zhu2026primal} combines a linearized proximal update with
updates of a smoothing center and the constraint multiplier, giving a
single-loop scheme when the required proximal mapping is available.
For a class of embedded manifolds defined by nonlinear equality constraints,
the feasibility-safeguarded inexact proximal linearized method (FSIPL)
\citep{xie2026fsipl} combines inexact proximal linearization with a normal
correction and a feasibility-safeguarded nonmonotone line search.
Its main iterations require neither projection nor retraction onto the manifold.

\subsection{Motivation and Contributions}
\label{sec:motivation-contributions}

For smooth problems, Landing methods
\citep{ablin2022fast,gao2022optimization,ablin2024infeasible}
 avoid per-iteration orthogonalization by combining a tangential direction with an infeasibility-reducing normal direction.
For general smooth nonlinear equality-constrained problems, \citet{schechtman2023orthogonal}
propose an orthogonal-directions method that recovers the Landing algorithm under a suitable choice of projection metric.
We extend the Landing framework to the nonsmooth composite problem~\eqref{problem}, which is not straightforward and requires exquisite constructions.

Following the geometric
interpretation of Landing flows in~\citet{gao2022optimization},
for a full-column-rank matrix \(X\in\R^{n\times p}\),  we consider the layered manifold~\citep{goyens2024layered}
\begin{equation}\label{intro:layer}
\St_{X^\top X}=\{Y\in\R^{n\times p}:Y^\top Y=X^\top X\}.
\end{equation}
 When \(X\) is feasible, i.e., \(X\in\St(n,p)\), this layered manifold reduces to the Stiefel manifold.
We measure orthogonality violation by
\begin{equation}\label{intro:normal-potential}
\mathcal N(X)=\frac14\|X^\top X-I_p\|_{\mathrm F}^{2}.
\end{equation}
Its negative gradient defines the normal direction
\begin{equation}\label{intro:normal-direction}
d_{\mathrm N}(X)=-\nabla\mathcal N(X)=-X(X^\top X-I_p).
\end{equation}
 For any \(d_{\mathrm T}(X)\in\mathrm T_X\St_{X^\top X}\), the directions \(d_{\mathrm T}(X)\) and \(d_{\mathrm N}(X)\) are orthogonal with respect to the Frobenius inner product. This orthogonality provides the geometric basis for an additive update that combines a tangential direction to optimize the objective with a normal direction to reduce infeasibility.
 \par In this paper, we construct an (approximately) tangential direction \(d_{\mathrm T,k}\) at each iterate \(X_k\) and retain the normal direction \(d_{\mathrm N,k}=d_{\mathrm N}(X_k)\), yielding the retraction-free update
\begin{equation}\label{intro:landing-template}
X_{k+1}=X_k+\eta_k d_{\mathrm T,k}+\zeta_k d_{\mathrm N,k},
\end{equation}
where \(\eta_k,\zeta_k>0\) are the tangential and normal stepsizes,
respectively. We construct the tangential direction using subgradient information or by solving a proximal subproblem. For the proximal approach, we first
establish convergence guarantees when the subproblems are solved exactly. In practice, each proximal subproblem is solved by an iterative method, which can be costly when high accuracy is required at every outer iteration. This motivates an inexact variant. An inexact solution, however, may yield a direction that is not tangent to the layered manifold and therefore need not be orthogonal to the normal direction. We address this issue through an explicit correction that controls violations of the tangency condition, together with acceptance conditions that allow early termination of the inner iterations while ensuring convergence.
Our main contributions are as follows.

\begin{enumerate}
\item \textbf{Subgradient Landing with parameterized normal stepsizes.}
We develop a family of subgradient Landing methods (SubLanding) with
\(\zeta_k=\lambda\eta_k^\alpha\), where \(\lambda>0\) and
\(\alpha\in[0,1]\).
Given a subgradient \(\xi_k\in\partial r(X_k)\), we use the Riemannian
metric of~\citet{gao2022optimization} on the layered manifold to construct
the tangential direction as
\begin{equation}\label{intro:subgradient-direction}
d_{\mathrm T,k}=-2\skew\bigl((\nabla h(X_k)+\xi_k)X_k^\top\bigr)X_k,
\end{equation}
where \(\skew(A)=(A-A^\top)/2\). A direction formed from an arbitrarily selected
subgradient need not vanish at a stationary point, so its norm alone
is not a suitable stationarity measure. To analyze SubLanding, we
introduce a Moreau-type
surrogate stationarity measure \(\Theta_\mu\), motivated by the
construction of \citet{li2021weakly}. The surrogate and the feasibility
residual vanish simultaneously if and only if \(X\) is a first-order
stationary point. We establish
best-iterate complexity bounds of \(\mathcal O(\epsilon^{-4})\)
for \(\Theta_\mu(X_k)\leq\epsilon\) and
\(\mathcal O(\epsilon^{-2/(2-\alpha)})\) for
\(\|X_k^\top X_k^{}-I_p\|_{\mathrm F}\leq\epsilon\)
(\cref{thm:alpha_sublanding_complexity}). For \(\alpha=0\), we also prove local
\(R\)-linear convergence under local weak sharpness, sufficiently close
initialization, suitable parameters, and geometrically decreasing
tangential stepsizes (\cref{thm:sublanding_local_linear_convergence}).

\item \textbf{Proximal gradient Landing with exact tangential directions.}
We develop the proximal gradient Landing method (PGLanding) by combining
the exact solution of a proximal subproblem on the tangent space of the layered manifold with
the normal direction.  For a fixed proximal parameter \(t>0\), define
\begin{equation}\label{intro:proximal-model}
V(X):=\underset{Z\in\mathrm T_X\St_{X^\top X}}{\operatorname*{arg\,min}}\left\{\inner{\nabla h(X)}{Z}+\frac{1}{2t}\|Z\|_{\mathrm F}^{2}+r(X+Z)\right\}.
\end{equation}
The tangential direction is \(d_{\mathrm T,k}=V(X_k)\).
We choose \(\zeta_k=\lambda\eta_k\), with \(\lambda>0\), and use a
merit-function-based line search to select \(\eta_k\). We establish best-iterate
outer-iteration complexity bounds of \(\mathcal O(\epsilon^{-2})\) for
\(\|V(X_k)\|_{\mathrm F}/t\leq\epsilon\) and
\(\mathcal O(\epsilon^{-1})\) for
\(\|X_k^\top X_k^{}-I_p\|_{\mathrm F}\leq\epsilon\). We also prove that every
accumulation point is first-order stationary (\cref{thm:prox_complexity}).

\item \textbf{Inexact proximal gradient Landing with tangency correction.}
The inexact proximal gradient Landing method (IPGLanding) uses directions
computed from intermediate multipliers and applies an explicit correction
to control violations of the tangency condition.
Let \(V(X,\Gamma)\) denote the proximal direction associated with a
symmetric multiplier \(\Gamma\) for the tangency constraint.
When \(\Gamma\) solves the multiplier equation, \(V(X,\Gamma)=V(X)\)
(\cref{prop:prox_wellposedness}).
Let \(\Gamma_k\) denote the multiplier accepted at outer iteration \(k\).
The corrected direction is
\begin{equation}\label{intro:correction}
d_{\mathrm T,k}=V(X_k,\Gamma_k)
-X_k\operatorname{sym}\!\bigl(X_k^\top V(X_k,\Gamma_k)\bigr),
\end{equation}
where \(\operatorname{sym}(A)=(A+A^\top)/2\).
We choose \(\zeta_k=\lambda\eta_k\), with \(\lambda>0\), and use a
nonmonotone merit-function-based line search to select \(\eta_k\).
We establish best-iterate
outer-iteration complexity bounds of
\(\mathcal O(\epsilon^{-2})\) for
\(\|V(X_k,\Gamma_k)\|_{\mathrm F}/t\leq\epsilon\) and
\(\mathcal O(\epsilon^{-1})\) for
\(\|X_k^\top X_k^{}-I_p\|_{\mathrm F}\leq\epsilon\). We also prove that every
accumulation point is first-order stationary (\cref{thm:iprox_complexity}).
Under an additional asymptotic
Bregman-error condition and the Kurdyka--\L ojasiewicz property of a
suitable auxiliary function, we further prove that the primal sequence
converges to a first-order stationary point (\cref{thm:iprox_kl_convergence}).

\item \textbf{Numerical experiments.}
For SubLanding, we
investigate the effect of the
normal stepsize scaling on infeasibility reduction and examine convergence
with geometrically decreasing stepsizes. For the proximal methods, we
compare objective values, feasibility residuals, iteration counts, and
running times with ManPG-Ada, RADMM, SOC, LSALM, and SLPG on random,
structured-synthetic, and real-data sparse principal component analysis
instances. These experiments demonstrate the practical effectiveness of
the proposed methods on the tested instances, with favorable running
times for IPGLanding.
\end{enumerate}

\cref{tab:method-guarantees} compares the structural features and theoretical
guarantees of the proposed methods with existing methods.

\enlargethispage{\baselineskip}
\begin{table}[!htbp]
\begingroup
\centering
\small
\setlength{\tabcolsep}{2.2pt}
\renewcommand{\arraystretch}{1.08}
\begin{tabular*}{\textwidth}{@{\extracolsep{\fill}}lcccc@{}}
\toprule
& \raisebox{6pt}{\mbox{Retraction-free}}
& \shortstack{Stationarity\\complexity}
& \shortstack{Feasibility\\complexity}
& \shortstack{Full-sequence\\convergence} \\
\midrule
SOC~\citep{lai2014splitting} & -- & -- & -- & -- \\
ManPG~\citep{chen2020proximal} & -- & \(\mathcal O(\epsilon^{-2})\) & -- & -- \\
RieSub~\citep{li2021weakly} & -- & \(\mathcal O(\epsilon^{-4})\) & -- & -- \\
RPG~\citep{huang2022riemannian} & -- & \(\mathcal O(\epsilon^{-2})\) & -- & \(\checkmark\) \\
SLPG~\citep{liu2024penaltyfree} & \(\checkmark\) & \(\mathcal O(\epsilon^{-2})\) & \(\mathcal O(\epsilon^{-1})\) & -- \\
RADMM~\citep{li2025riemannian} & -- & \(\mathcal O(\epsilon^{-4})\) & -- & -- \\
LSALM~\citep{zhu2026primal} & \(\checkmark\) & \(\mathcal O(\epsilon^{-3})\) & \(\mathcal O(\epsilon^{-3})\) & \(\checkmark\) \\
FSIPL~\citep{xie2026fsipl} & \(\checkmark\) & \(\mathcal O(\epsilon^{-2})\) & \(\mathcal O(\epsilon^{-1})\) & \(\checkmark\) \\
\midrule
SubLanding (this paper) & \(\checkmark\) & \(\mathcal O(\epsilon^{-4})\) & \(\mathcal O(\epsilon^{-1})\) & -- \\
PGLanding (this paper) & \(\checkmark\) & \(\mathcal O(\epsilon^{-2})\) & \(\mathcal O(\epsilon^{-1})\) & -- \\
IPGLanding (this paper)   & \(\checkmark\) & \(\mathcal O(\epsilon^{-2})\) & \(\mathcal O(\epsilon^{-1})\) & \(\checkmark\) \\

\bottomrule
\end{tabular*}
\endgroup
\caption{Comparison of first-order methods for nonsmooth optimization over the Stiefel manifold. Retraction-free refers to the main iterations. Best-iterate outer-iteration complexity bounds use method-specific measures. SubLanding bounds correspond to \(\alpha=0\). A dash means not applicable or not established.}
\label{tab:method-guarantees}
\end{table}
\FloatBarrier

\subsection{Organization}
\label{sec:organization}
\cref{sec:preliminaries} introduces notation and preliminary results.
\cref{sec:sublanding,sec:proxlanding,sec:iproxlanding} present SubLanding,
PGLanding, and IPGLanding and their convergence analyses, respectively.
\cref{sec:experiments} reports numerical experiments, and
\cref{sec:conclusion} discusses future work.
The main proofs are collected in~\cref{app:technical-proofs}.

\FloatBarrier

\section{Preliminaries}
\label{sec:preliminaries}
This section fixes the notation and standing assumptions used throughout the
paper and records the constraint geometry and first-order stationarity
conditions needed in the subsequent analysis.

\subsection{Notation}

Let \(\R_*^{n\times p}=\{X\in\R^{n\times p}:\rank(X)=p\}\),
\(\mathbb S^p=\{S\in\R^{p\times p}:S^\top=S\}\), and
\(\mathrm O(n)=\{Q\in\R^{n\times n}:Q^\top Q=I_n\}\).
The Frobenius inner product is
\(\langle U,V\rangle=\operatorname{tr}(U^\top V)\), and
\(\|U\|_{\mathrm{F}}\) and \(\|U\|_2\) denote the Frobenius and spectral norms.
For a matrix \(U\), \(\|U\|_1=\sum_{i,j}|U_{ij}|\) is the entrywise
\(\ell_1\)-norm. For a square
matrix \(A\), let \(\sym(A)=(A+A^\top)/2\) and
\(\skew(A)=(A-A^\top)/2\).

For a smooth embedded manifold \(\mathcal M\),
\(\mathrm T_X\mathcal M\) denotes its tangent space at \(X\).
For a nonempty closed convex set \(\mathcal C\),
\(\operatorname{Proj}_{\mathcal C}\) denotes the Frobenius orthogonal
projection onto \(\mathcal C\). For a nonempty set \(\mathcal S\), define
\(\dist(U,\mathcal S)=\inf_{V\in\mathcal S}\|U-V\|_{\mathrm{F}}\).

For a proper lower semicontinuous function \(\varphi\), \(\partial\varphi\)
denotes its limiting (Mordukhovich) subdifferential
\citep{rockafellar1998variational} in the ambient Euclidean space.
When \(\varphi\) is convex, it coincides with the convex subdifferential.
For \(t>0\), the proximal mapping of
\(r\) is defined by
\[
\prox_{t r}(Z):=\argmin_Y\left\{r(Y)+\frac{1}{2t}\|Y-Z\|_{\mathrm F}^{2}\right\}.
\]
Throughout, \(\varepsilon\in(0,1)\) denotes the fixed safe-region radius,
whereas \(\epsilon>0\) denotes a target stationarity or feasibility accuracy.

\subsection{Assumptions, Geometry, and Stationarity}

\begin{assumption}\label{ass:objective_functions}
The following conditions hold.
\begin{enumerate}
\item The function \(h:\R^{n\times p}\to\R\) is continuously
differentiable, and its gradient is \(L_h\)-Lipschitz continuous for some
\(L_h>0\), that is,
\(\|\nabla h(X)-\nabla h(Y)\|_{\mathrm{F}}\leq L_h\|X-Y\|_{\mathrm{F}}\)
for all \(X,Y\in\R^{n\times p}\).
\item The function \(r:\R^{n\times p}\to\R\) is convex, its proximal
mapping is efficiently computable, and it is \(L_r\)-Lipschitz continuous
for some \(L_r>0\), that is,
\(\lvert r(X)-r(Y)\rvert\leq L_r\|X-Y\|_{\mathrm{F}}\)
for all \(X,Y\in\R^{n\times p}\).
\end{enumerate}
\end{assumption}
Under~\cref{ass:objective_functions}, the composite function is
subdifferentially regular~\citep{rockafellar1998variational}. Consequently,
\begin{equation}\label{composite_subdifferential}
\partial f(X)=\nabla h(X)+\partial r(X).
\end{equation}
Since \(r\) is convex and \(L_r\)-Lipschitz continuous, every \(\xi\in\partial r(X)\) satisfies
\begin{equation}\label{nonsmooth_subgradient_bound}
\|\xi\|_{\mathrm{F}}\leq L_r.
\end{equation}
Moreover, Lipschitz continuity of \(\nabla h\) and convexity of \(r\) imply
that \(f\) is \(L_h\)-weakly convex. Thus, for all \(X,Y\in\R^{n\times p}\),
\[
f(Y)\geq f(X)+\langle F,Y-X\rangle
-\frac{L_h}{2}\|Y-X\|_{\mathrm{F}}^2
\quad\text{for every }F\in\partial f(X).
\]
Since the methods allow infeasible iterates, we introduce the safe
region~\citep{ablin2024infeasible} for \(\varepsilon\in(0,1)\),
\[
\St(n,p)^\varepsilon
=\{X\in\R^{n\times p}:\|X^\top X-I_p\|_{\mathrm{F}}\leq\varepsilon\}\subset\R_*^{n\times p}.
\]
On the safe region, let
\mbox{\(L_f^\varepsilon=\|\nabla h(0)\|_{\mathrm{F}}+L_h\sqrt{p(1+\varepsilon)}+L_r\)}. Then
\begin{equation}\label{safe_region_objective_bounds}
|f(X)-f(Y)|\leq L_f^\varepsilon\|X-Y\|_{\mathrm{F}},
\quad
\|F\|_{\mathrm{F}}\leq L_f^\varepsilon,
\end{equation}
for all \(X,Y\in\St(n,p)^\varepsilon\) and
\(F\in\partial f(X)\).
For \(X\in\St(n,p)^\varepsilon\), the tangent space of the layered manifold defined in~\eqref{intro:layer} is
\begin{equation*}
\mathrm T_X\St_{X^\top X}
=
\{Z:X^\top Z+Z^\top X=0\}
=
\{WX:W\in\R^{n\times n},W^\top=-W\}.
\end{equation*}

\begin{definition}
\label[definition]{def:first_order_stationarity}
A matrix \(X\in\R_*^{n\times p}\) is called a first-order stationary point of
 problem~\eqref{problem} if there are \(\xi\in\partial r(X)\) and
\(\Gamma\in\mathbb S^p\) such that
\begin{equation}\label{explicit_nonsmooth_KKT}
X^\top X-I_p=0,
\quad
\nabla h(X)+\xi-2X\Gamma=0.
\end{equation}
We refer to~\cref{explicit_nonsmooth_KKT} as the
Karush--Kuhn--Tucker (KKT) system for~problem~\eqref{problem}.
\end{definition}

\section{The Subgradient Landing Method}
\label{sec:sublanding}
This section introduces the SubLanding iteration and its surrogate
stationarity measure, establishes safe-region invariance and global
complexity, and then derives a  local \(R\)-linear convergence result.

\subsection{Full-Rank Metric and the SubLanding Iteration}

To construct a tangential subgradient direction on the current  layered manifold, we use the metric of \citet{gao2022optimization}. For
\(X\in\R_*^{n\times p}\), this metric is given by
\begin{equation*}
g_X(U,V)
=
\left\langle
U,
\left(I_n-\frac12X(X^\top X)^{-1}X^\top\right)
V(X^\top X)^{-1}
\right\rangle,
\end{equation*}
for \(U,V\in\mathrm T_X\St_{X^\top X}\).
The Riemannian subdifferential of \(f\) on \(\St_{X^\top X}\) with respect to this metric
has the explicit form
\begin{equation}\label{canonical_riemannian_subdifferential}
\partial_{\St_{X^\top X}}^{g}f(X)
=
\left\{
2\skew\left((\nabla h(X)+\xi)X^\top\right)X:
\xi\in\partial r(X)
\right\}.
\end{equation}
Accordingly, at \(X\in\St(n,p)^\varepsilon\), select
\(\xi\in\partial r(X)\), set \(F=\nabla h(X)+\xi\), and define
\(d_{\mathrm{T}}(X;\xi)=-2\skew(FX^\top)X\in\mathrm T_X\St_{X^\top X}\).
By~\cref{canonical_riemannian_subdifferential}, \(d_{\mathrm{T}}(X;\xi)\) is a
negative Riemannian subgradient on \(\St_{X^\top X}\).   
We combine this tangential direction with the normal direction \(d_{\mathrm N}(X)\),
using tangential stepsizes \(\eta_k>0\) and normal stepsizes
\(\zeta_k=\lambda\eta_k^\alpha\), where \(\alpha\in[0,1]\) and \(\lambda>0\).
At each iterate \(X_k\), choose \(\xi_k\in\partial r(X_k)\) and set
\begin{equation}\label{sublanding_quantities}
F_k=\nabla h(X_k)+\xi_k,
\quad
W_k=2\skew(F_k^{}X_k^\top).
\end{equation}
Then the SubLanding update is
\begin{equation}\label{sublanding_iteration}
X_{k+1}
=X_k-\eta_kW_kX_k-\zeta_kX_k(X_k^\top X_k^{}-I_p).
\end{equation}
The iteration is summarized in~\cref{alg:SubLanding}.
\begin{algorithm}[H]
\caption{The Subgradient Landing Method (SubLanding)}
\label{alg:SubLanding}
\begin{algorithmic}[1]
\REQUIRE \(X_0\in\St(n,p)^\varepsilon\), \(\alpha\in[0,1]\),
\(\lambda>0\), and a sequence of stepsizes \(\eta_k>0\).
\FOR{\(k=0,1,\ldots\)}
    \STATE Choose \(\xi_k\in\partial r(X_k)\) and set
    \(F_k=\nabla h(X_k)+\xi_k\).
    \STATE Set \(W_k=2\skew(F_k^{}X_k^\top)\) and
    \(\zeta_k=\lambda\eta_k^\alpha\).
    \STATE Set
    \(X_{k+1}=X_k-\eta_kW_kX_k-\zeta_kX_k(X_k^\top X_k^{}-I_p)\).
\ENDFOR
\end{algorithmic}
\end{algorithm}
\subsection{A Moreau-Type Surrogate Stationarity Measure}
\label{subsec:sublanding-stationarity-measure}

For nonsmooth objectives, the
residual \(\dist(0,\partial_{\St_{X^\top X}}^g f(X))\) need not approach
zero even along a sequence converging to a first-order stationary point.
To address this difficulty in the global convergence analysis, we follow
the Moreau-envelope approach of~\citet{li2021weakly} and construct a surrogate
using the orthogonal-orbit geometry underlying the SubLanding direction.

For
\(W=2\skew(FX^\top)\), one has \(W^\top=-W\), and the curve
\(\tau\mapsto\exp(-\tau W)X\) has initial velocity \(-WX\). Thus a step
along the tangential direction is a first-order approximation of a left
orthogonal rotation. For every \(X\in\R_*^{n\times p}\),
\cref{prop:layer_orbit_identity} shows that the layered manifold is its left orthogonal orbit.
\begin{equation}\label{eq:layer_orbit_identity}
\St_{X^\top X}=\{QX:Q\in\mathrm O(n)\}.
\end{equation}
Optimization over this orbit preserves the feasibility residual. We
therefore regularize the objective over orthogonal transformations,
penalizing their deviation from the identity. For \(\mu>0\), define
\begin{equation}\label{orthogonal_envelope}
f_\mu(X)
=
\min_{Q\in\mathrm O(n)}
\left\{f(QX)+\frac{1}{2\mu}\|Q-I_n\|_{\mathrm{F}}^2\right\}.
\end{equation}
Since \(\mathrm O(n)\) is compact and the minimized function is continuous,
the minimum is attained. Choose any minimizer
\(\widehat Q_X\) in~\cref{orthogonal_envelope} and set
\begin{equation}\label{Theta_definition}
\widehat X=\widehat Q_XX,
\quad
\Theta_\mu(X)
=
\frac{2\sqrt{1+\varepsilon}}{\mu}
\|\widehat Q_X-I_n\|_{\mathrm{F}}.
\end{equation}
The minimizer \(\widehat Q_X\) need not be unique, but the following result
holds for every such choice. It shows that \(\Theta_\mu(X)\) controls both
the distance from \(X\) to \(\widehat X\) and the stationarity residual
at \(\widehat X\) on the  layered manifold. Together with feasibility, its
vanishing characterizes first-order stationarity of problem~\eqref{problem}.

\prooflink{app:orthogonal_stationarity_proof}
\begin{proposition}
\label[proposition]{prop:orthogonal_stationarity}
For every \(X\in\St(n,p)^\varepsilon\) and
\(0<\mu<1/(L_h+L_f^\varepsilon)\),
\begin{equation*}
\|\widehat X-X\|_{\mathrm{F}}\leq\frac\mu2\Theta_\mu(X),\quad
\dist\bigl(0,\partial_{\St_{\widehat X^\top \widehat X}}^g f(\widehat X)\bigr)
\leq\Theta_\mu(X).
\end{equation*}
Moreover,
\begin{equation*}
\Theta_\mu(X)=0,\; X^\top X-I_p=0
\;\Longleftrightarrow\;
X\text{ is a first-order stationary point of problem~\eqref{problem}}.
\end{equation*}
\end{proposition}
We use \(\Theta_\mu\) as an analytical surrogate rather than a numerical
stopping criterion.

\subsection{Safe-Region Invariance and Global Complexity}
\label{sec:unified_sublanding_convergence}

To establish safe-region invariance for SubLanding, we first examine the
normal-only update
\(Y_\zeta(X)=X+\zeta d_{\mathrm{N}}(X)=X-\zeta X(X^\top X-I_p)\).
To quantify its effect uniformly over the safe region, define
\[q_\varepsilon(\zeta)=\max_{|s|\leq\varepsilon}
|1-\zeta(1+s)(2-\zeta s)|.\]
The following proposition bounds the residual after this update and
identifies the stepsize range for which \(q_\varepsilon(\zeta)<1\).
\prooflink{app:normal_contraction_proof}
\begin{proposition}
\label[proposition]{prop:sublanding_exact_lambda_range}
For every \(X\in\St(n,p)^\varepsilon\) and every \(\zeta>0\),
\begin{equation}\label{exact_normal_contraction}
\|Y_\zeta(X)^\top Y_\zeta(X)-I_p\|_{\mathrm{F}}
\leq
q_\varepsilon(\zeta)\|X^\top X-I_p\|_{\mathrm{F}}.
\end{equation}
For a constant normal stepsize \(\zeta=\lambda>0\), \(q_\varepsilon(\lambda)<1\) if and only if
\begin{equation}\label{eq:sublanding_lambda_range}
0<\lambda<\lambda_{\mathrm{crit}}(\varepsilon),
\quad
\lambda_{\mathrm{crit}}(\varepsilon)=
\begin{cases}
\dfrac{2}{1+\varepsilon+\sqrt{1-\varepsilon^2}},
&0<\varepsilon\leq1/\sqrt2,\\[1.5ex]
2(\sqrt2-1),
&1/\sqrt2<\varepsilon<1.
\end{cases}
\end{equation}
\end{proposition}
For the full SubLanding update, the tangential step contributes a
second-order term to the  feasibility residual. Combining a bound on this
term with the normal contraction yields safe-region invariance under
suitable stepsize conditions. Together with the surrogate stationarity
measure introduced above, this leads to the following global complexity
result.

\prooflink{app:sublanding_master_descent_proof}
\begin{theorem}
\label{thm:alpha_sublanding_complexity}
Fix \(\alpha\in[0,1]\) and \(\lambda>0\), and impose
\cref{eq:sublanding_lambda_range} when \(\alpha=0\). There exist constants
\(c_{\mathrm{N}}>0\), \(\eta^*>0\), \(\mu^*>0\), and
\(\nu^*>0\) such that the following holds. Let
\(X_0\in\St(n,p)^\varepsilon\), and let \(\{X_k\}\) be generated by
\cref{sublanding_iteration} with \(0<\eta_k\leq\eta^*\) for every
\(k\). Suppose that \(0<\mu\leq\mu^*\) and
\(\nu\geq\nu^*\). Define
\(\mathcal L_{\mu,\nu}(X)=f_\mu(X)+\nu\|X^\top X-I_p\|_{\mathrm{F}}\) and set
\(\Delta_{\mathrm{S},0}
=\mathcal L_{\mu,\nu}(X_0)
-\inf_{X\in\St(n,p)^\varepsilon}\mathcal L_{\mu,\nu}(X)\).
Then
\(X_k\in\St(n,p)^\varepsilon\) for every \(k\), and there exists
\(C_{\mathrm{S}}>0\) such that, for every \(K\geq1\),
\begin{equation}\label{alpha_general_best_iterate}
\min_{0\leq k<K}
\left\{
\frac{\eta_k}{8(1+\varepsilon)}\Theta_\mu^2(X_k)
+\frac{\nu c_{\mathrm{N}}\lambda}{4}\eta_k^\alpha
\|X_k^\top X_k^{}-I_p\|_{\mathrm{F}}
\right\}
\leq
\frac{\Delta_{\mathrm{S},0}
+C_{\mathrm{S}}\sum_{k=0}^{K-1}\eta_k^2}{K}.
\end{equation}
Let \(j_K\) attain the minimum in \cref{alpha_general_best_iterate}.
Fix \(0<\eta_0\leq\eta^*\) independently of \(K\) and use one of the following
stepsize schedules:
\begin{enumerate}
\renewcommand{\labelenumi}{(\alph{enumi})}
\item \textup{\textbf{Horizon-dependent stepsize.}}
For a prescribed horizon \(K\), set \(\eta_k=\eta_0/\sqrt K\) for
\(k=0,\ldots,K-1\).

\item \textup{\textbf{Diminishing stepsize.}}
Set \(\eta_k=\eta_0/\sqrt{k+1}\) for \(k=0,1,\ldots\).
\end{enumerate}
In either case,
\begin{equation}\label{alpha_stepsize_rates}
\Theta_\mu(X_{j_K})=\mathcal O(K^{-1/4}),
\quad
\|X_{j_K}^\top X_{j_K}^{}-I_p\|_{\mathrm{F}}
=\mathcal O(K^{-(1-\alpha/2)}).
\end{equation}
Consequently, the best-iterate complexities for
\(\Theta_\mu\) and the feasibility residual are
\(\mathcal O(\epsilon^{-4})\) and
\(\mathcal O(\epsilon^{-2/(2-\alpha)})\), respectively. The feasibility
complexity specializes to \(\mathcal O(\epsilon^{-1})\) for \(\alpha=0\)
and \(\mathcal O(\epsilon^{-2})\) for \(\alpha=1\).
\end{theorem}

\subsection{Local \texorpdfstring{$R$}{R}-Linear Convergence Under Weak Sharpness}
\label{sec:sublanding_local_linear_convergence}

We next establish a  local \(R\)-linear convergence result for the
constant-normal-stepsize regime \(\alpha=0\), using geometrically decreasing
tangential stepsizes. The resulting iteration is
\begin{equation}\label{eq:local_sublanding_iteration}
X_{k+1}
=X_k-\eta_kW_kX_k-\lambda X_k(X_k^\top X_k^{}-I_p),
\quad
\eta_k=\eta_0\vartheta^k,
\quad
0<\vartheta<1.
\end{equation}

We use the local weak-sharp-minima condition of
\citet[Definition~1]{li2021weakly} on the Stiefel manifold, with the set of minima taken to be
nonempty and closed. For earlier work on weak sharp minima in mathematical
programming and on Riemannian manifolds, see \citet{burke1993weak},
\citet{li2011weak}, and \citet{karkhaneei2019nonconvex}.

\begin{definition}
\label[definition]{def:sublanding_weak_sharpness}
A nonempty closed set \(\mathcal X\subseteq\St(n,p)\) is a set of
local weak sharp minima with parameter \(\sigma_{\mathrm{sh}}>0\) if there
is \(\rho_{\mathrm{sh}}>0\) such that
\begin{equation}\label{eq:sublanding_weak_sharpness}
f(U)-f(Y)
\geq
\sigma_{\mathrm{sh}}\dist(U,\mathcal X),
\end{equation}
for every \(Y\in\mathcal X\) and every \(U\in\St(n,p)\) with
\(\dist(U,\mathcal X)<\rho_{\mathrm{sh}}\).
\end{definition}
\citet[Proposition~4 and Section~7.1]{li2021weakly} establish local weak
sharpness under suitable data conditions for dual
principal component pursuit and least-absolute-deviations formulations of
robust subspace recovery.
\Cref{prop:spca_local_weak_sharpness} gives sufficient conditions for this
property in SPCA.

Let \(\mathcal X\) be a set of local weak sharp minima. To apply weak
sharpness to possibly infeasible iterates, we consider their feasible polar
factors
\(\overline X_k=\mathcal P(X_k)=X_k(X_k^\top X_k^{})^{-1/2}\).
For \(X_k\in\St(n,p)^\varepsilon\), the deviation from the polar factor
is bounded by the feasibility residual:
\begin{equation}\label{eq:sublanding_polar_distance}
\|X_k-\overline X_k\|_{\mathrm{F}}
\leq
\frac{1}{1+\sqrt{1-\varepsilon}}\,\|X_k^\top X_k^{}-I_p\|_{\mathrm{F}}.
\end{equation}
We measure the orthogonal displacement needed to map the feasible factor into \(\mathcal X\).
For \(U\in\St(n,p)\), let
\(\mathcal Q(U)=\{Q\in\mathrm O(n):QU\in\mathcal X\}\).
The orthogonal-orbit identity~\cref{eq:layer_orbit_identity} and the
nonemptiness of \(\mathcal X\) ensure that \(\mathcal Q(U)\) is nonempty;
it is compact because \(\mathcal X\) is closed and \(\mathrm O(n)\) is
compact. We therefore define the minimum orthogonal displacement
\begin{equation}\label{eq:sublanding_orthogonal_distance}
\omega(U)=\min_{Q\in\mathcal Q(U)}\|Q-I_n\|_{\mathrm{F}},
\quad
\omega_k=\omega(\overline X_k).
\end{equation}

\prooflink{app:sublanding_local_linear_proof}
\begin{theorem}
\label{thm:sublanding_local_linear_convergence}
Assume that \(\mathcal X\) is a set of local weak sharp
minima in the sense of~\cref{def:sublanding_weak_sharpness}, and
\(0<\lambda<\lambda_{\mathrm{crit}}(\varepsilon)\), where
\(\lambda_{\mathrm{crit}}\) is defined in
\cref{eq:sublanding_lambda_range}. There is
\(\underline\vartheta\in(0,1)\) such that, for every
\(\vartheta\in(\underline\vartheta,1)\), there are positive constants
\(\overline\eta,C_E,C_\omega\) with the following
property. If \(0<\eta_0\leq\overline\eta\),
\(\|X_0^\top X_0^{}-I_p\|_{\mathrm{F}}\leq C_E\eta_0^2\), and
\(\omega_0\leq C_\omega\eta_0\), then the iterates generated
by~\cref{eq:local_sublanding_iteration} satisfy
\(X_k\in\St(n,p)^\varepsilon\) for every \(k\) and converge to a
first-order stationary point
\(X_\infty\in\mathcal X\). Moreover,
\[
\|X_k^\top X_k^{}-I_p\|_{\mathrm{F}}=\mathcal O(\vartheta^{2k}),
\quad
\|X_k-X_\infty\|_{\mathrm{F}}=\mathcal O(\vartheta^k).
\]
Hence the sequence converges \(R\)-linearly to \(X_\infty\).
\end{theorem}
  
\section{The Proximal Gradient Landing Method}
\label{sec:proxlanding}
PGLanding uses a tangential proximal subproblem to handle the nonsmooth
term \(r\) and combines its solution with the normal direction
\(d_{\mathrm N}(X)\).

\subsection{The PGLanding Algorithm}
\label{sec:prox_direction_update}
Whereas SubLanding uses a selected subgradient of \(r\), PGLanding retains
\(r\) in a tangential proximal subproblem and linearizes only the smooth term
\(h\). We restrict the
direction to the tangent space of the layered manifold \(\St_{X^\top X}\).
For a fixed proximal parameter \(t>0\), define
\begin{equation}
\label{prox_subproblem}
V(X):=\underset{Z\in\mathrm T_X\St_{X^\top X}}{\argmin}\left\{\inner{\nabla h(X)}{Z}+\frac{1}{2t}\|Z\|_{\mathrm F}^{2}+r(X+Z)\right\}.
\end{equation}
At a feasible point, this reduces to the tangential proximal subproblem of
ManPG~\citep{chen2020proximal}.
The tangency constraint is linear in \(Z\), and the quadratic term makes
the subproblem strongly convex, so its solution is unique.

A finite step along \(V(X)\) need not remain on the current layered manifold.
To control the  feasibility residual, we combine this tangential direction
with the normal direction \(d_{\mathrm N}(X)\). For \(\lambda>0\), set
\begin{equation}\label{prox_total_direction}
\Lambda(X):=V(X)+\lambda d_{\mathrm N}(X)=V(X)-\lambda X(X^\top X-I_p).
\end{equation}
The resulting retraction-free update is
\begin{equation}
\label{prox_primal_update}
X_{k+1}=X_k+\eta_k\Lambda(X_k)
=X_k+\eta_k\bigl[V(X_k)-\lambda X_k(X_k^\top X_k^{}-I_p)\bigr].
\end{equation}

\label{sec:prox_algorithm}
The normal direction may increase the objective while reducing
infeasibility. We therefore select the stepsize using the merit function
\begin{equation}\label{landing_merit_function}
\Phi_\sigma(X)=f(X)+\sigma\|X^\top X-I_p\|_{\mathrm{F}},\quad \sigma>0.
\end{equation}
For PGLanding, we choose \(\sigma=\sigma_{\mathrm P}\)
large enough that the weighted feasibility decrease along the normal
direction outweighs the possible objective increase. Given an Armijo
parameter \(\rho\in(0,1)\), set
\begin{equation}
\label{prox_sigma_condition}
\sigma_{\mathrm P}=
\vartheta_{\mathrm{P}}
\frac{\rho+\lambda L_f^\varepsilon\sqrt{1+\varepsilon}}
{2\lambda(1-\varepsilon)},\quad
\vartheta_{\mathrm{P}}>1.
\end{equation}

Merit decrease alone does not guarantee that a trial remains in the safe
region, so the line search also checks its constraint violation.
At iteration \(k\), we hold \(V(X_k)\) fixed.
A trial \(X_{k,q}^+=X_k+\eta_{k,q}\Lambda(X_k)\), where
\(\eta_{k,q}=\beta_{\mathrm{ls}}^q\eta_0\), \(\eta_0>0\),
\(\beta_{\mathrm{ls}}\in(0,1)\) and \(q=0,1,\ldots\), is accepted if it satisfies the
sufficient-decrease condition
\begin{equation}\label{prox_acceptance_test}
\Phi_{\sigma_{\mathrm P}}(X_{k,q}^+)
\leq\Phi_{\sigma_{\mathrm P}}(X_k)-\eta_{k,q}\rho
\left(\frac{\|V(X_k)\|_{\mathrm{F}}^2}{t}+\|X_k^\top X_k^{}-I_p\|_{\mathrm{F}}\right)
\end{equation}
and the safe-region condition
\begin{equation}\label{prox_safe_region_test}
\|(X_{k,q}^+)^\top X_{k,q}^+-I_p\|_{\mathrm{F}}\leq\varepsilon.
\end{equation}
If either test fails, the
stepsize is reduced by \(\beta_{\mathrm{ls}}\) and the trial is recomputed.
These steps give~\cref{alg:ProxLanding}.
\begin{algorithm}[H]
\caption{The Proximal Gradient Landing Method (PGLanding)}
\label{alg:ProxLanding}
\begin{algorithmic}[1]
\REQUIRE \(X_0\in\St(n,p)^\varepsilon\), \(t>0\), \(\lambda>0\),
\(\sigma_{\mathrm P}\) satisfying~\cref{prox_sigma_condition},
\(\rho\in(0,1)\),
\(\beta_{\mathrm{ls}}\in(0,1)\), and \(\eta_0>0\).
\FOR{\(k=0,1,\ldots\)}
    \STATE Compute the tangential direction \(V(X_k)\) by
    solving~subproblem~\eqref{prox_subproblem}.
    \STATE Set \(\eta=\eta_0\) and
    \(X^+=X_k+\eta\Lambda(X_k)\).
    \WHILE{either condition in~\cref{prox_acceptance_test,prox_safe_region_test}
    fails for the current trial \((\eta,X^+)\)}
        \STATE Set \(\eta=\beta_{\mathrm{ls}}\eta\) and
        \(X^+=X_k+\eta\Lambda(X_k)\).
    \ENDWHILE
    \STATE Set \(X_{k+1}=X^+\) and \(\eta_k=\eta\).
\ENDFOR
\end{algorithmic}
\end{algorithm}

\subsection{Proximal Subproblem and Multiplier Equation}
\label{sec:prox_subproblem_solution}
It remains to compute the direction \(V(X)\) used in the update.
We introduce a symmetric Lagrange multiplier
\(\Gamma\in\mathbb S^p\) for the linear constraint
\(X^\top V+V^\top X=0\). With the multiplier term
\(-\inner{\Gamma}{X^\top V+V^\top X}\), the KKT conditions
for~subproblem~\eqref{prox_subproblem} are
\[
0\in\nabla h(X)+t^{-1}V+\partial r(X+V)-2X\Gamma,\quad
0=X^\top V+V^\top X.
\]
The first condition gives the multiplier-dependent
proximal direction
\begin{equation}\label{prox_multiplier_direction}
V(X,\Gamma)=\prox_{t r}
\bigl(X-t\nabla h(X)+2tX\Gamma\bigr)-X.
\end{equation}
Substituting this expression into the tangency constraint yields the
multiplier equation \(R(X,\Gamma)=0\), where the tangency residual is
\begin{equation}\label{prox_multiplier_residual}
R(X,\Gamma)=X^\top V(X,\Gamma)+V(X,\Gamma)^\top X.
\end{equation}

\prooflink{app:prox_wellposedness_proof}
\begin{proposition}
\label[proposition]{prop:prox_wellposedness}
For every \(X\in\R_*^{n\times p}\),  subproblem~\eqref{prox_subproblem} has a unique solution. Moreover, there is
\(\Gamma_X\in\mathbb S^p\) such that
\(R(X,\Gamma_X)=0\) and \(V(X,\Gamma_X)=V(X)\).
For every \(\Gamma\in\mathbb S^p\), \(R(X,\Gamma)=0\) if and only if
\(V(X,\Gamma)=V(X)\).
Finally,
\begin{equation*}
V(X)=0,\; X^\top X-I_p=0
\;\Longleftrightarrow\;
X\text{ is a first-order stationary point of problem~\eqref{problem}}.
\end{equation*}
\end{proposition}
Thus, \(\|V(X)\|_{\mathrm{F}}/t\) and \(\|X^\top X-I_p\|_{\mathrm{F}}\) provide computable
stationarity and feasibility residuals, respectively.

The multiplier equation \(R(X,\Gamma)=0\) can be solved by
a regularized semismooth Newton (SSN) method
\citep{qi1993nonsmooth,xiao2018regularized} when \(\prox_{t r}\) is
semismooth, following \citet[Section~4.2]{chen2020proximal}.

\subsection{Convergence and Complexity Analysis}
\label{sec:prox_convergence}
For the convergence analysis, we assume that the tangential subproblem is
solved exactly. The subproblem optimality conditions, the normal direction,
and the choice of \(\sigma_{\mathrm P}\) ensure that sufficiently small trial
steps satisfy both acceptance conditions. A uniform stepsize lower bound
and the merit decrease then yield the following convergence and
outer-iteration complexity guarantees.
\prooflink{app:proof-prox-complexity}
\begin{theorem}
\label{thm:prox_complexity}
Assume~\cref{ass:objective_functions} holds. Apply
\cref{alg:ProxLanding} with exact tangential subproblem solutions and with
\(\sigma_{\mathrm P}\) satisfying~\cref{prox_sigma_condition}. Let \(\{X_k\}\) be the
resulting sequence, \(f_*=\min_{X\in\St(n,p)^\varepsilon}f(X)\), and
\(\Delta_{\mathrm{P},0}=\Phi_{\sigma_{\mathrm P}}(X_0)-f_*\).
Then every line search terminates after finitely many backtracking steps,
\(X_k\in\St(n,p)^\varepsilon\) for every \(k\), and there is a constant
\(\underline\eta>0\) such that \(\eta_k\geq\underline\eta\) for every \(k\).
For every \(K\geq1\), let \(i_K\) minimize
\(\|V(X_k)\|_{\mathrm{F}}^2/t+\|X_k^\top X_k^{}-I_p\|_{\mathrm{F}}\) over
\(0\leq k<K\). Then
\begin{equation}
\label{prox_separate_complexity_bounds}
\frac{\|V(X_{i_K})\|_{\mathrm{F}}}{t}\leq
\sqrt{\frac{\Delta_{\mathrm{P},0}}
{\rho tK\underline\eta}},\quad
\|X_{i_K}^\top X_{i_K}^{}-I_p\|_{\mathrm{F}}\leq
\frac{\Delta_{\mathrm{P},0}}
{\rho K\underline\eta}.
\end{equation}
The corresponding stationarity and feasibility outer-iteration complexities
are \(\mathcal O(\epsilon^{-2})\) and
\(\mathcal O(\epsilon^{-1})\), respectively. Moreover,
\(\|V(X_k)\|_{\mathrm{F}}/t\to0\) and
\(\|X_k^\top X_k^{}-I_p\|_{\mathrm{F}}\to0\),
and every accumulation point is first-order stationary for
 problem~\eqref{problem}.
\end{theorem}
 
\section{The Inexact Proximal Gradient Landing Method}
\label{sec:iproxlanding}
Solving the tangential proximal subproblem accurately at every outer
iteration may be costly. IPGLanding allows approximate solutions and
applies an explicit correction to control violations of the tangency
condition.

\subsection{The IPGLanding Algorithm}
\label{sec:iprox_direction_update}
As shown in~\cref{sec:prox_subproblem_solution}, the exact direction
is obtained by solving \(R(X,\Gamma)=0\). With an approximate multiplier,
\(R(X,\Gamma)\) may be nonzero, so \(V(X,\Gamma)\) need not be tangent
to the current layered manifold. We apply an explicit correction to
control this violation of the tangency condition.

At feasible points, the Euclidean orthogonal projection onto the tangent space is
\[
\operatorname{Proj}_{\mathrm T_X\St(n,p)}(V)
=V-X\sym(X^\top V).
\]
We retain this explicit formula at infeasible  points to avoid solving the matrix equation associated with an exact tangent projection, and define
\begin{equation}\label{iprox_corrected_direction}
D(X,\Gamma)=V(X,\Gamma)-X\sym\bigl(X^\top V(X,\Gamma)\bigr)
=V(X,\Gamma)-\frac12XR(X,\Gamma).
\end{equation}
This correction can also be obtained by replacing \((X^\top X)^{-1}\)
with \(I_p\) in an exact oblique projection; the exact orthogonal and
oblique projections are discussed in~\cref{app:iprox_tangency_correction}.
Although the corrected direction need not be exactly tangent at infeasible
points, its  violation of the tangency condition satisfies
\begin{equation}\label{iprox_tangency_defect_bound}
\|X^\top D(X,\Gamma)+D(X,\Gamma)^\top X\|_{\mathrm F}
\leq\|X^\top X-I_p\|_{\mathrm F}\|R(X,\Gamma)\|_{\mathrm F}.
\end{equation}
In particular, the corrected direction is exactly tangent when
\(X^\top X=I_p\) or \(R(X,\Gamma)=0\).

We replace the exact direction \(V(X)\) in PGLanding with \(D(X,\Gamma)\) and retain the same normal direction. For a fixed \(\lambda>0\), set
\begin{equation}\label{iprox_total_direction}
\Lambda(X,\Gamma):=D(X,\Gamma)+\lambda d_{\mathrm N}(X)=D(X,\Gamma)-\lambda X(X^\top X-I_p).
\end{equation}
At outer iteration \(k\), the point \(X_k\) remains fixed while
\(\Gamma_{k,j}\) is updated. For a trial stepsize \(\eta>0\), the trial point is
\(X^+=X_k+\eta\Lambda(X_k,\Gamma_{k,j})\).
If this trial is accepted, set \(J_k=j\), \(\Gamma_k=\Gamma_{k,J_k}\), and
\(\eta_k=\eta\). The resulting outer update is
\begin{equation}\label{iprox_primal_update}
X_{k+1}=X_k+\eta_k\Lambda(X_k,\Gamma_k)
=X_k+\eta_k\bigl[D(X_k,\Gamma_k)
-\lambda X_k(X_k^\top X_k^{}-I_p)\bigr].
\end{equation}

\label{sec:iprox_algorithm}
We compute approximate multipliers by a fixed-point iteration that uses
the tangency residual \(R(X,\Gamma)\) to update \(\Gamma\).
With \(X\in\St(n,p)^\varepsilon\) held fixed, choose
\begin{equation}
\label{iprox_safe_dual_step}
0<\theta_{\mathrm d}<2,\quad
\tau_{\mathrm d}=\frac{\theta_{\mathrm d}}{4t(1+\varepsilon)}.
\end{equation}
Set \(R_{\mathrm a}=L_f^\varepsilon/\sqrt{1-\varepsilon}\) and
\(\mathbb B_{\mathrm a}=\{\Gamma\in\mathbb S^p:
\|\Gamma\|_{\mathrm F}\leq R_{\mathrm a}\}\).
This ball contains a solution multiplier for every point in the safe region.
Starting from \(\Gamma_0\in\mathbb B_{\mathrm a}\), update
\begin{equation}
\label{iprox_dual_update}
\Gamma_{j+1}=\Gamma_j-\tau_{\mathrm d}R(X,\Gamma_j),\quad j\geq0.
\end{equation}
With this stepsize, \(R(X,\Gamma_j)\to0\) when the iteration is continued
at fixed \(X\), as proved in~\cref{app:iprox_complexity_proof}.

The multiplier iteration need not solve \(R(X,\Gamma)=0\) before an
outer update. Instead, we test the trial point generated by each candidate
direction. Acceptance requires merit decrease relative to a reference
value and membership in the safe region.
We use the merit function
in~\cref{landing_merit_function} with \(\sigma=\sigma_{\mathrm{IP}}\).
For \(0<\rho<1\), choose
\begin{equation}
\label{iprox_penalty_condition}
\sigma_{\mathrm{IP}}=
\vartheta_{\mathrm{IP}}
\frac{\rho+\lambda L_f^\varepsilon\sqrt{1+\varepsilon}}
{2\lambda(1-\varepsilon)},\quad
\vartheta_{\mathrm{IP}}>1.
\end{equation}
The target decrease uses the current proximal direction and feasibility residual:
\begin{equation}\label{iprox_target_decrease}
M(X,\Gamma)
=\rho\bigl(\|V(X,\Gamma)\|_{\mathrm{F}}^2/t
+\|X^\top X-I_p\|_{\mathrm{F}}\bigr).
\end{equation}

Unlike PGLanding, we compare each trial with a weighted reference value rather than the current merit value. We use the
Zhang--Hager reference sequence~\citep{zhang2004nonmonotone}: choose
\(\omega_{\mathrm{mem}}\in[0,1)\), set \(a_0^{\mathrm{nm}}=1\) and
\(\overline\Phi_0=\Phi_{\sigma_{\mathrm{IP}}}(X_0)\), and, after a step produces
\(X_{k+1}\), update
\begin{equation}
\label{iprox_nonmonotone_reference}
a_{k+1}^{\mathrm{nm}}=\omega_{\mathrm{mem}}a_k^{\mathrm{nm}}+1,\quad
\overline\Phi_{k+1}=
\frac{\omega_{\mathrm{mem}}a_k^{\mathrm{nm}}\overline\Phi_k
+\Phi_{\sigma_{\mathrm{IP}}}(X_{k+1})}{a_{k+1}^{\mathrm{nm}}}.
\end{equation}

After \(j\) multiplier updates at \(X_k\), write
\(V_{k,j}=V(X_k,\Gamma_{k,j})\),
\(R_{k,j}=R(X_k,\Gamma_{k,j})\). Also set
\(\Lambda_{k,j}=\Lambda(X_k,\Gamma_{k,j})\),
\(M_{k,j}=M(X_k,\Gamma_{k,j})\).
Fix \(c\in(0,1)\), an initial trial stepsize \(\eta_0>0\), and a
backtracking factor \(\beta_{\mathrm{ls}}\in(0,1)\). For each multiplier
\(\Gamma_{k,j}\), first test the trial point
\(X_k+\eta_0\Lambda_{k,j}\). More generally, a trial with stepsize
\(\eta>0\) is accepted if it satisfies the nonmonotone decrease condition
\begin{equation}
\label{iprox_acceptance_test}
\Phi_{\sigma_{\mathrm{IP}}}(X_k+\eta\Lambda_{k,j})
\leq\overline\Phi_k-c\eta M_{k,j}
\end{equation}
and the safe-region condition
\begin{equation}\label{iprox_safe_region_test}
\|(X_k+\eta\Lambda_{k,j})^\top(X_k+\eta\Lambda_{k,j})-I_p\|_{\mathrm{F}}
\leq\varepsilon.
\end{equation}

If the initial trial fails these tests, we must decide whether to reduce
the stepsize or refine the direction. Controlling the  violation of the tangency condition does
not by itself ensure descent, so backtracking on the current direction
may fail. We therefore assess its quality using the following upper bound on the
first-order merit change, with \(\Lambda=\Lambda(X,\Gamma)\):
\begin{equation}\label{iprox_merit_change_bound}
\ell(X,\Gamma)
=\inner{\nabla h(X)}{\Lambda}+r(X+\Lambda)-r(X)
-\sigma_{\mathrm{IP}}\|X^\top X-I_p\|_{\mathrm{F}}
\bigl(2\lambda(1-\varepsilon)-\|R(X,\Gamma)\|_{\mathrm{F}}\bigr).
\end{equation}
Writing \(\ell_{k,j}=\ell(X_k,\Gamma_{k,j})\), we require both a sufficient
descent bound and a bound on the tangency residual before backtracking:
\begin{equation}
\label{iprox_primal_conditions}
\ell_{k,j}\leq-M_{k,j},
\quad
\|R_{k,j}\|_{\mathrm{F}}\leq\lambda(1-\varepsilon).
\end{equation}
We first check the tangency-residual bound and compute \(\ell_{k,j}\)
only if this bound holds. If both conditions hold, keep the multiplier fixed and backtrack over
\(\eta=\beta_{\mathrm{ls}}^q\eta_0\), \(q=1,2,\ldots\), accepting the first
trial that satisfies~\cref{iprox_acceptance_test,iprox_safe_region_test}. Otherwise, update the multiplier by
\cref{iprox_dual_update}, recompute the direction, and test a new trial point
with the same stepsize \(\eta_0\).

After accepting the outer step, we take one additional multiplier step
using \(R_{k,J_k}\) and project the result onto \(\mathbb B_{\mathrm a}\)
to initialize the next outer iteration. This ensures that
\mbox{\(\Gamma_{k+1,0}\in\mathbb B_{\mathrm a}\)}.
The complete procedure is given in~\cref{alg:iproxlanding}.

\begin{algorithm}[H]
\caption{The Inexact Proximal Gradient Landing Method (IPGLanding)}
\label{alg:iproxlanding}
\begin{algorithmic}[1]
\REQUIRE \(X_0\in\St(n,p)^\varepsilon\),
\(\Gamma_{0,0}\in\mathbb B_{\mathrm{a}}\),
\(t,\lambda,\eta_0>0\), \(\rho,c,\beta_{\mathrm{ls}}\in(0,1)\),
\(\omega_{\mathrm{mem}}\in[0,1)\), and
\cref{iprox_safe_dual_step,iprox_penalty_condition}.
\STATE Set \(a_0^{\mathrm{nm}}=1\) and
\(\overline\Phi_0=\Phi_{\sigma_{\mathrm{IP}}}(X_0)\).
\FOR{\(k=0,1,\ldots\)}
    \STATE Set \(j=0\) and compute
    \(V_{k,j}\), \(R_{k,j}\), \(\Lambda_{k,j}\), and \(M_{k,j}\).
    \REPEAT
        \STATE Test \(X_k+\eta_0\Lambda_{k,j}\) against
        \cref{iprox_acceptance_test,iprox_safe_region_test}.
        \IF{the \(\eta_0\)-trial is accepted}
            \STATE Set \(\eta_k=\eta_0\).
        \ELSIF{\(\|R_{k,j}\|_{\mathrm F}\leq\lambda(1-\varepsilon)\)}
            \STATE Compute \(\ell_{k,j}\) by~\cref{iprox_merit_change_bound}.
            \IF{\(\ell_{k,j}\leq-M_{k,j}\)}
                \STATE Keep \(\Gamma_{k,j}\) fixed.
                \STATE Backtrack over \(q=1,2,\ldots\) and set
                \(\eta_k=\beta_{\mathrm{ls}}^q\eta_0\) at the first trial
                satisfying~\cref{iprox_acceptance_test,iprox_safe_region_test}.
            \ENDIF
        \ENDIF
        \IF{no primal stepsize has been accepted for the current candidate}
            \STATE Set
            \(\Gamma_{k,j+1}=\Gamma_{k,j}-\tau_{\mathrm{d}}R_{k,j}\)
            and \(j=j+1\).
            \STATE Recompute \(V_{k,j}\), \(R_{k,j}\),
            \(\Lambda_{k,j}\), and \(M_{k,j}\).
        \ENDIF
    \UNTIL{a primal stepsize \(\eta_k\) has been accepted}
    \STATE Set \(J_k=j\),
    \(\Gamma_k=\Gamma_{k,J_k}\), and
    \(X_{k+1}=X_k+\eta_k\Lambda_{k,J_k}\).
    \STATE Set
    \(\Gamma_{k+1,0}=\operatorname{Proj}_{\mathbb B_{\mathrm{a}}}
    (\Gamma_k-\tau_{\mathrm{d}}R_{k,J_k})\).
    \STATE Update \(a_{k+1}^{\mathrm{nm}}\) and
    \(\overline\Phi_{k+1}\) by~\cref{iprox_nonmonotone_reference}.
\ENDFOR
\end{algorithmic}
\end{algorithm}

\subsection{Convergence and Complexity Analysis}
\label{sec:iprox_convergence}
The following characterization relates \(V(X,\Gamma)\) to
first-order stationarity.
\prooflink{app:iprox_residual_exactness_proof}
\begin{proposition}
\label[proposition]{prop:iprox_residual_exactness}
For every \(X\in\St(n,p)^\varepsilon\),
\begin{equation*}
\begin{gathered}
\exists\,\Gamma\in\mathbb S^p:
V(X,\Gamma)=0,\quad X^\top X-I_p=0
\\ \Longleftrightarrow\quad
X\text{ is a first-order stationary point of problem~\eqref{problem}}.
\end{gathered}
\end{equation*}
\end{proposition}
The next theorem shows that backtracking terminates whenever it is invoked
and that accepted stepsizes are uniformly bounded away from zero.
The decrease in the reference sequence then yields the same outer-iteration
complexity orders as PGLanding.
\prooflink{app:iprox_complexity_proof}
\begin{theorem}
\label{thm:iprox_complexity}
Assume~\cref{ass:objective_functions} holds. Apply
\cref{alg:iproxlanding} with \(X_0\in\St(n,p)^\varepsilon\),
\(\Gamma_{0,0}\in\mathbb B_{\mathrm{a}}\), \(t>0\),
\(\lambda>0,\eta_0>0\), \(\rho,c,\beta_{\mathrm{ls}}\in(0,1)\),
\(\omega_{\mathrm{mem}}\in[0,1)\), and parameters satisfying
\cref{iprox_safe_dual_step,iprox_penalty_condition}.

Whenever backtracking is invoked, it
terminates after finitely many trials, and every accepted primal stepsize
satisfies \(\eta_k\geq\underline\eta_{\mathrm{IP}}>0\) for a constant independent
of \(k\). At every outer iterate \(X_k\) that is not first-order stationary
for~problem~\eqref{problem}, the inner loop terminates after finitely many
multiplier updates.
For the outer-iteration complexity, define
\[
\Delta_{\mathrm{IP}}
=
\frac{\overline\Phi_0-
\min_{X\in\St(n,p)^\varepsilon}\Phi_{\sigma_{\mathrm{IP}}}(X)}
{c(1-\omega_{\mathrm{mem}})\underline\eta_{\mathrm{IP}}}.
\]
For every \(K\geq1\) for which \(X_K\) is generated, let \(i_K\) minimize
\(\rho\|V(X_k,\Gamma_k)\|_{\mathrm{F}}^2/t+
\rho\|X_k^\top X_k^{}-I_p\|_{\mathrm{F}}\) over \(0\leq k<K\). Then
\begin{equation}
\label{iprox_primal_best_iterate_rates}
\frac{\|V(X_{i_K},\Gamma_{i_K})\|_{\mathrm{F}}}{t}
\leq
\sqrt{\frac{\Delta_{\mathrm{IP}}}{\rho tK}},\quad
\|X_{i_K}^\top X_{i_K}^{}-I_p\|_{\mathrm{F}}
\leq
\frac{\Delta_{\mathrm{IP}}}{\rho K}.
\end{equation}
Consequently, the stationarity and feasibility outer-iteration complexities
are \(\mathcal O(\epsilon^{-2})\) and
\(\mathcal O(\epsilon^{-1})\), respectively. 
If infinitely many outer iterates are generated, every accumulation point of
\(\{X_k\}\) is first-order stationary for~problem~\eqref{problem}.
\end{theorem}

\subsection{Sequence Convergence under the KL Property}
\label{sec:iprox_kl_convergence}
For full-sequence convergence, we use the KL property and an additional
error condition on the mismatch between the current iterate and the point
at which proximal optimality supplies a subgradient.
At every completed outer iteration set
\(M_k=M_{k,J_k}=M(X_k,\Gamma_k)\) and
\(Z_k=X_k+V(X_k,\Gamma_k)\).
We use the proximal optimality condition to define
\begin{equation}
\label{iprox_selected_subgradient}
\xi_k=-\nabla h(X_k)-\frac{1}{t}V(X_k,\Gamma_k)+2X_k\Gamma_k
\in\partial r(Z_k).
\end{equation}
The gap between \(r(X_k)\) and its subgradient model at \(Z_k\) is the
Bregman error
\begin{equation}
\label{iprox_bregman_error}
B_k=r(X_k)-r(Z_k)-\inner{\xi_k}{X_k-Z_k}\geq0.
\end{equation}
We require this gap to be controlled asymptotically by the decrease in the
nonmonotone reference sequence.
\begin{assumption}
\label{ass:iprox_bregman}
There exist a constant \(C_B>0\) and an integer \(k_B\geq1\) such that
\begin{equation*}
B_k\leq C_B\bigl(\delta_{k-1}^{\mathrm{nm}}+\delta_k^{\mathrm{nm}}\bigr),
\quad
\delta_k^{\mathrm{nm}}=\overline\Phi_k-\overline\Phi_{k+1},
\quad k\geq k_B.
\end{equation*}
\end{assumption}
A sufficient condition for \cref{ass:iprox_bregman} is
\(B_k\leq L_B\|V(X_k,\Gamma_k)\|_{\mathrm{F}}^2\)
for some \(L_B>0\) and all sufficiently large \(k\). This quadratic bound
holds when \(r\) is continuously differentiable with a locally Lipschitz
gradient. For
\(r=\gamma\|\cdot\|_1\), it also holds with \(B_k=0\) once \(X_k\) and \(Z_k\)
have the same support and signs. Soft thresholding may set entries of \(Z_k\)
exactly to zero while the corresponding entries of \(X_k\) remain nonzero.
Thus this identification is not guaranteed by the proximal step, and
\cref{ass:iprox_bregman} remains an explicit assumption.

Under \cref{ass:iprox_bregman}, we introduce an auxiliary function
to account for both the nonmonotone reference sequence and the
mismatch between \(X_k\) and \(Z_k\) in the KL analysis. We adopt the standard Kurdyka--{\L}ojasiewicz property
\citep[Definition~2.4]{attouch2013convergence}.
Following the finite-length framework of \citet{attouch2013convergence}
and the uniformization argument of \citet{bolte2014proximal}, we adapt the
construction in \citet[Section~4]{xie2026fsipl} to the Zhang--Hager reference
sequence. Define the lifted function on
\((\R^{n\times p})^2\times\mathbb S^p\times\R^{n\times p}\times\R\),
equipped with the standard Euclidean product norm, by
\begin{equation}
\label{iprox_kl_lift}
\mathcal H_{\mathrm{IP}}(X,Z,\Gamma,\xi,u)
=h(X)+r(Z)+\inner{\xi}{X-Z}
-\inner{\Gamma}{X^\top X-I_p}+u\lvert u\rvert.
\end{equation}
The term \(u\lvert u\rvert\) provides a continuously differentiable correction
of either sign, allowing the lifted energy to match a decreasing average of
the reference values. If \(h\) and \(r\) are semialgebraic, then
\(\mathcal H_{\mathrm{IP}}\) is semialgebraic and hence a KL function
\citep[Theorem~3]{bolte2014proximal}.

\prooflink{app:iprox_kl_convergence_proof}
\begin{theorem}
\label{thm:iprox_kl_convergence}
Assume the hypotheses of~\cref{thm:iprox_complexity} hold,
\cref{alg:iproxlanding} generates infinitely many outer iterates,
\cref{ass:iprox_bregman} holds, and the function in
\cref{iprox_kl_lift} is a KL function. Then
\begin{equation}
\label{iprox_kl_finite_length}
\sum_{k=0}^{\infty}\|X_{k+1}-X_k\|_{\mathrm{F}}<\infty.
\end{equation}
Consequently, the primal sequence \(\{X_k\}\) converges. Its limit \(X_\infty\) is
a first-order stationary point of problem~\eqref{problem}.
\end{theorem}

\section{Numerical Experiments}
\label{sec:experiments}
We evaluate the proposed methods on sparse principal component analysis
(SPCA) problems. The SubLanding experiments examine the effects of the
normal stepsize parameters and compare its convergence behavior with
RieSub. We also compare PGLanding and IPGLanding with
ManPG-Ada\footnote{The authors' MATLAB implementation of ManPG-Ada is available at 
 \url{https://github.com/chenshixiang/ManPG/tree/master/manpg_code_share/SPCA}.},
SLPG\footnote{The Python reference implementation of SLPG is provided in PySTOP v0.2.2, available at
 \url{https://pypi.org/project/pystop/0.2.2/}.},
RADMM\footnote{The authors' RADMM implementation is available at
 \url{https://github.com/JasonJiaxiangLi/RADMM}.},
SOC\footnote{A MATLAB implementation of SOC is available at
 \url{https://github.com/chenshixiang/ManPG/tree/master/manpg_code_share/SPCA}.},
and LSALM.
All computations are performed in MATLAB~R2025b on a computer with an
Intel Core Ultra~9 275HX processor and 32~GB of RAM.
The source code for the proposed algorithms will be made publicly available at \url{https://github.com/cui-zilong/Nonsmooth-Stiefel-Landing}.

\subsection{Sparse PCA and Experimental Setup}
\label{sec:spca}

Given a symmetric positive semidefinite matrix $L\in\R^{n\times n}$,
we consider the following orthogonality-constrained SPCA problem:
\begin{equation}
\label{eq:spca}
\min_{X\in\R^{n\times p}}
f(X)=-\frac{1}{2}\operatorname{tr}(X^\top L X)
+\gamma\|X\|_1,
\quad
\text{subject to}\quad X^\top X=I_p,
\end{equation}
where $n$ is the number of features, $p\leq n$ is the number of principal components, and
$\gamma>0$ is the sparsity regularization parameter.
We construct \(L\) from three types of data: randomly generated data,
structured synthetic data, and the Colon gene-expression data set.

\emph{Randomly generated data.}
We generate an \(m\times n\) matrix with independent standard Gaussian
entries, where \(m\) is the number of samples. We then center each column
and normalize it to unit Euclidean norm. Denoting the resulting matrix
by \(A\), we set \(L=2A^\top A\).

\emph{Structured synthetic data.}
Following the five-pattern construction used by
\citet{zhu2026primal}, we use the profiles illustrated in
\citet[Figure~4]{huang2022riemannian}, repeat each profile \(m/5\) times,
and add independent Gaussian noise with mean zero and variance
\(0.25\) to each entry. We then center each column and normalize it
to unit Euclidean norm, obtaining \(A\in\R^{m\times n}\), and set
\(L=2A^\top A\).

\emph{Real-world data.}
 We use the Colon gene-expression data set~\citep{alon1999colon}\footnote{The data are available at \url{https://github.com/alxiang/Colon-Cancer-Classification/blob/master/ColonX.csv}.},
which contains 62 samples and 2000 variables. After removing
the sample identifiers and variables with no observed values,
we replace non-finite entries by the mean of the finite values
in the corresponding variable and discard constant variables.
We then center each retained variable and scale it to unit
sample variance. Denoting the resulting matrix by
\(A\in\R^{m\times n}\), we set \(L=2A^\top A/(m-1)\).
We use \(p=5\), following~\citet{thiao2010dc}, who
report that the first five principal components explain approximately
\(70\%\) of the total variance.

Following~\citet{zhu2026primal}, all methods on each problem instance
start from a common point obtained by projecting a standard Gaussian
matrix onto the Stiefel manifold and running the Riemannian subgradient
method~\citep{li2021weakly} for 250 iterations with stepsizes
$\eta_k=(k+1)^{-3/4}$, $k=0,\ldots,249$.
In the following experiments, sparsity is measured as the proportion of
entries satisfying \(|X_{ij}|\leq10^{-5}\).
CPU times reported in the tables are averaged over 10 runs with fixed
data and initialization. These repetitions measure timing variability,
not variability across random initializations.

\subsection{SubLanding: Effects of the Normal Direction and Stepsize}\label{subsec:sublanding-numerics}

The first comparison evaluates SubLanding against the Riemannian subgradient
method (RieSub) of~\citet{li2021weakly} and examines the role of the normal
direction. Both methods use the same subgradient selection
\(F_k=-LX_k+\gamma\operatorname{sign}(X_k)\in\partial f(X_k)\),
where \(\operatorname{sign}\) is applied entrywise with
\(\operatorname{sign}(0)=0\).
\begin{samepage}
RieSub uses the update
\[
X_{k+1}=\operatorname{Retr}_{X_k}
\left(-\eta_k\operatorname{Proj}_{\mathrm T_{X_k}\St(n,p)}(F_k)\right),
\]
where \(\operatorname{Proj}_{\mathrm T_{X_k}\St(n,p)}(\cdot)\) denotes the
orthogonal projection onto the tangent space at \(X_k\), and
\(\operatorname{Retr}_{X_k}(\cdot)\) is the polar retraction at \(X_k\).
SubLanding uses \(F_k\) in~\cref{sublanding_quantities,sublanding_iteration}.
\end{samepage}

A randomly generated SPCA instance with \(n=m=500\), \(p=20\), and
\(\gamma=0.5\) is used.
To examine the effect of \(\alpha\) and \(\lambda\), each setting is run for
\(K=50000\) iterations using
\(\eta_k=10^{-3}(k+1)^{-1/2}\), \(k=0,1,\ldots,K-1\). The tested values are
\(\lambda\in\{5,50,500\}\) for \(\alpha=1\) and
\(\lambda\in\{0.25,0.5,0.75\}\) for \(\alpha=0\).

\begin{table}[tbp]
\centering
\small
\setlength{\tabcolsep}{3pt}
\begin{tabular}{@{}lcccccr@{}}
\toprule
Method & \(\alpha\) & \(\lambda\) & \(f(X)\) & \(\|X^\top X-I_p\|_{\mathrm{F}}\) & Sparsity & Time (s) \\
\midrule
RieSub & -- & -- & \(-9.995\) & \(2.0\times10^{-14}\) & 0.9980 & 21.206 \\
\midrule
SubLanding & 1 & 5 & \(-9.996\) & \(3.9\times10^{-04}\) & 0.9980 & 10.748 \\
SubLanding & 1 & 50 & \(-9.995\) & \(2.5\times10^{-05}\) & 0.9980 & 10.571 \\
SubLanding & 1 & 500 & \(-9.995\) & \(2.5\times10^{-06}\) & 0.9980 & 10.699 \\
\midrule
SubLanding & 0 & 0.25 & \(-9.995\) & \(2.3\times10^{-08}\) & 0.9980 & 10.542 \\
SubLanding & 0 & 0.50 & \(-9.995\) & \(1.2\times10^{-08}\) & 0.9980 & 10.925 \\
SubLanding & 0 & 0.75 & \(-9.995\) & \(7.9\times10^{-09}\) & 0.9980 & 10.887 \\
\bottomrule
\end{tabular}
\caption{Effects of \(\alpha\) and \(\lambda\) on SubLanding
after \(50000\) iterations, with RieSub as a baseline.}
\label{tab:sublanding-parameter-screening}
\end{table}

\begin{figure}[t]
\centering
\includegraphics[width=\textwidth]{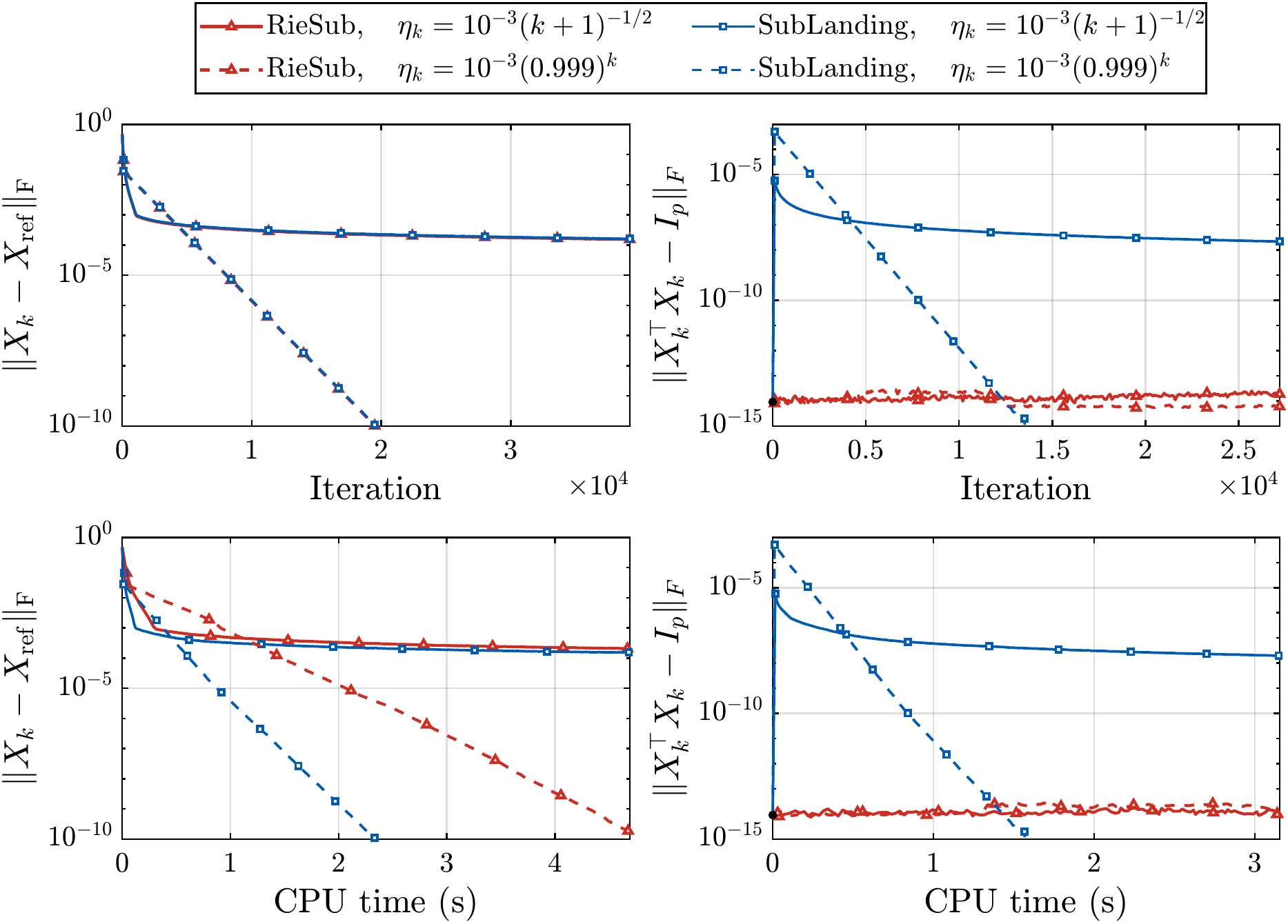}
\caption{Comparison of RieSub and SubLanding in iterate distance
(left) and  feasibility residual (right). Black dots mark the
common initial feasibility residual. The horizontal axes are
truncated for clarity.}
\label{fig:sublanding-riesub-convergence}
\end{figure}

As shown in~\cref{tab:sublanding-parameter-screening}, after \(50000\)
iterations on this instance, the tested choices with
\(\alpha=0\) give substantially smaller feasibility residuals than those
with \(\alpha=1\). This qualitative difference is compatible with the
best-iterate asymptotic bounds in~\cref{thm:alpha_sublanding_complexity}, which give
feasibility complexity \(\mathcal O(\epsilon^{-2})\) for
\(\alpha=1\) and \(\mathcal O(\epsilon^{-1})\) for \(\alpha=0\).

For the convergence comparison in~\cref{fig:sublanding-riesub-convergence},
we use \((\alpha,\lambda)=(0,0.5)\) for SubLanding and test both methods
with the polynomially diminishing stepsize \(\eta_k=10^{-3}(k+1)^{-1/2}\)
and the geometrically decreasing stepsize \(\eta_k=10^{-3}(0.999)^k\).

We use ManPG-Ada~\citep{chen2020proximal} to compute a
high-accuracy reference point \(X_{\mathrm{ref}}\) satisfying
\(\|V_{\mathrm{ref}}/t_{\mathrm{ref}}\|_{\mathrm{F}}\leq10^{-10}\),
where \(V_{\mathrm{ref}}\) is the computed tangential proximal direction
and \(t_{\mathrm{ref}}\) is the corresponding proximal parameter.
\Cref{fig:sublanding-riesub-convergence} reports the distance
\(\|X_k-X_{\mathrm{ref}}\|_{\mathrm{F}}\) and the  feasibility residual
\(\|X_k^\top X_k^{}-I_p\|_{\mathrm{F}}\) against iteration count and CPU time,
with each curve obtained from a single run.
With geometrically decreasing stepsizes, both methods attain smaller
distances to \(X_{\mathrm{ref}}\) over the displayed range than with
polynomially diminishing stepsizes. SubLanding has lower measured CPU time
on this instance. RieSub preserves feasibility up to numerical precision,
whereas the SubLanding feasibility residual decreases over the run. For the
geometric stepsize, the iterate distance and feasibility residual of SubLanding
exhibit approximately linear decay on the semilogarithmic plots over the
displayed decreasing portions.

\FloatBarrier
\subsection{Comparison with Existing Methods}
\label{sec:prox-implementation}
\label{subsec:cross-method-comparison}

We compare PGLanding and IPGLanding with
ManPG-Ada~\citep{chen2020proximal},
SLPG~\citep{liu2024penaltyfree}, RADMM~\citep{li2025riemannian},
SOC~\citep{lai2014splitting}, and LSALM~\citep{zhu2026primal}.

\emph{Test instances.}
We use the three types of data described in~\cref{sec:spca}.
For the random instances, we fix \(m=100\), \(n=1000\), and
\(\gamma=0.5\), and vary \(p\in\{5,10,15,20\}\).
For the structured synthetic instances, we fix \(m=1000\) and
\(\gamma=0.5\), and use
\((n,p)\in\{(200,100),(300,150),(400,200),(500,250)\}\).
For the Colon data set, we use \(m=62\), \(n=2000\), and \(p=5\),
and vary \(\gamma\in\{1,5,10\}\).

\emph{Initialization and reference points.}
All methods use the common initialization in~\cref{sec:spca}.
Starting from this point, ManPG-Ada computes the numerical reference point
\(X_{\mathrm{ref}}\) and value \(f_{\mathrm{ref}}=f(X_{\mathrm{ref}})\).
Let \(V_{\mathrm{ref}}\) be the computed tangential proximal direction
for subproblem~\eqref{prox_subproblem} at \(X_{\mathrm{ref}}\), with
proximal parameter \(t_{\mathrm{ref}}\). The reference computation terminates when 
\(\|V_{\mathrm{ref}}/t_{\mathrm{ref}}\|_{\mathrm{F}}\leq10^{-6}\). 
We tighten the SSN tolerance to \(10^{-12}\) for this computation.
All reference points satisfy this stationarity criterion.

\begin{samepage}
\emph{Stopping and reporting criteria.}
All comparison runs use the common tolerance \(\epsilon=10^{-4}\). 
For \(k\geq1\), let \(X_k\) and \(X_{k-1}\) denote consecutive reported
iterates. A run terminates when
\begin{equation}\label{eq:common-dataset-stopping}
f(X_k)-f_{\mathrm{ref}}\leq\epsilon,
\quad
\|X_k-X_{k-1}\|_{\mathrm{F}}\leq\epsilon,
\quad
\|X_k^\top X_k^{}-I_p\|_{\mathrm{F}}\leq\epsilon.
\end{equation}
\end{samepage}
For SOC, \(U_k\), \(Q_k\), and \(X_k\) denote its three primal split
 variables, with \(X_k\) satisfying the orthogonality constraints.
SOC additionally requires
\[
\frac{\|U_k-Q_k\|_{\mathrm{F}}}
{\max\{1,\|U_k\|_{\mathrm{F}},\|Q_k\|_{\mathrm{F}}\}}
+
\frac{\|U_k-X_k\|_{\mathrm{F}}}
{\max\{1,\|U_k\|_{\mathrm{F}},\|X_k\|_{\mathrm{F}}\}}
\leq\epsilon.
\]
For RADMM, \(X_k\) and \(Z_k\) denote its two primal split variables, and the
additional requirement is
\[
\frac{\|X_k-Z_k\|_{\mathrm{F}}}{\max\{1,\|X_k\|_{\mathrm{F}},\|Z_k\|_{\mathrm{F}}\}}
\leq\epsilon.
\]
After termination, we verify
\(\|X-X_{\mathrm{ref}}\|_{\mathrm{F}}<0.1\) for the returned point \(X\).
All methods are limited to 30000 outer iterations.
``Iter'' denotes the iteration count, and
``Time/Iter'' is the average CPU time per iteration.

\begin{table}[!t]
\centering
\begingroup
\small
\setlength{\tabcolsep}{2.5pt}
\resizebox{\textwidth}{!}{
\begin{tabular}{ccclcccrcc}
\toprule
\(n\) & \(m\) & \(p\) & Method & \(f(X)\) & Sparsity & Iter & Time (s) & Time/Iter (s) & \(\|X^\top X-I_p\|_{\mathrm{F}}\) \\
\midrule
\multirow{7}{*}{1000} & \multirow{7}{*}{100} & \multirow{7}{*}{5} & ManPG-Ada   & \(-28.0296\) & \(0.5016\) & \(452\)  & \(0.230\) & \(5.09\times10^{-04}\)  & \(6.24\times10^{-15}\) \\
& & & PGLanding & \(-28.0296\) & \(0.5020\) & \(1374\) & \(0.710\) & \(5.17\times10^{-04}\)  & \(6.16\times10^{-09}\) \\
& & & RADMM       & \(-28.0295\) & \(0.5016\) & \(2774\) & \(0.590\) & \(2.13\times10^{-04}\)  & \(3.08\times10^{-15}\) \\
& & & SOC         & \(-28.0295\) & \(0.5016\) & \(2282\) & \(1.310\) & \(5.74\times10^{-04}\)  & \(1.51\times10^{-15}\) \\
& & & LSALM       & \(-28.0296\) & \(0.5018\) & \(2454\) & \(0.482\) & \(1.96\times10^{-04}\) & \(1.72\times10^{-05}\) \\
& & & SLPG        & \(-28.0296\) & \(0.5018\) & \(598\) & \(0.194\) & \(3.24\times10^{-04}\) & \(8.74\times10^{-14}\) \\
& & & IPGLanding   & \(-28.0295\) & \(0.5018\) & \(622\) & \(0.169\) & \(2.72\times10^{-04}\) & \(6.06\times10^{-09}\) \\
\midrule
\multirow{7}{*}{1000} & \multirow{7}{*}{100} & \multirow{7}{*}{10} & ManPG-Ada   & \(-52.4990\) & \(0.5193\) & \(1031\) & \(1.340\) & \(1.30\times10^{-03}\)  & \(8.96\times10^{-15}\) \\
& & & PGLanding & \(-52.4990\) & \(0.5192\) & \(3730\) & \(4.150\) & \(1.11\times10^{-03}\)  & \(4.94\times10^{-09}\) \\
& & & RADMM       & \(-52.4989\) & \(0.5194\) & \(7257\) & \(3.330\) & \(4.59\times10^{-04}\)  & \(8.91\times10^{-15}\) \\
& & & SOC         & \(-52.4989\) & \(0.5195\) & \(6030\) & \(3.270\) & \(5.42\times10^{-04}\)  & \(5.05\times10^{-15}\) \\
& & & LSALM       & \(-52.4989\) & \(0.5191\) & \(6910\) & \(2.250\) & \(3.26\times10^{-04}\) & \(1.19\times10^{-05}\) \\
& & & SLPG        & \(-52.4990\) & \(0.5193\) & \(1518\) & \(1.379\) & \(9.08\times10^{-04}\) & \(6.90\times10^{-14}\) \\
& & & IPGLanding   & \(-52.4989\) & \(0.5193\) & \(1597\) & \(0.846\) & \(5.30\times10^{-04}\) & \(2.36\times10^{-09}\) \\
\midrule
\multirow{7}{*}{1000} & \multirow{7}{*}{100} & \multirow{7}{*}{15} & ManPG-Ada   & \(-74.9531\) & \(0.5301\) & \(1829\)  & \(4.230\) & \(2.31\times10^{-03}\) & \(1.34\times10^{-14}\) \\
& & & PGLanding & \(-74.9531\) & \(0.5301\) & \(7275\)  & \(9.970\) & \(1.37\times10^{-03}\) & \(4.58\times10^{-09}\) \\
& & & RADMM       & \(-74.9531\) & \(0.5300\) & \(15854\) & \(9.330\) & \(5.88\times10^{-04}\) & \(1.28\times10^{-14}\) \\
& & & SOC         & \(-74.9531\) & \(0.5301\) & \(13660\) & \(9.160\) & \(6.71\times10^{-04}\) & \(5.95\times10^{-15}\) \\
& & & LSALM       & \(-74.9531\) & \(0.5302\) & \(13558\) & \(4.281\) & \(3.16\times10^{-04}\) & \(7.04\times10^{-06}\) \\
& & & SLPG        & \(-74.9532\) & \(0.5297\) & \(3162\) & \(3.020\) & \(9.55\times10^{-04}\) & \(1.48\times10^{-13}\) \\
& & & IPGLanding   & \(-74.9531\) & \(0.5297\) & \(3316\) & \(2.065\) & \(6.23\times10^{-04}\) & \(2.86\times10^{-09}\) \\
\midrule
\multirow{7}{*}{1000} & \multirow{7}{*}{100} & \multirow{7}{*}{20} & ManPG-Ada   & \(-94.2129\) & \(0.5509\) & \(659\)  & \(1.820\) & \(2.76\times10^{-03}\)  & \(1.77\times10^{-14}\) \\
& & & PGLanding & \(-94.2129\) & \(0.5512\) & \(2323\) & \(4.870\) & \(2.10\times10^{-03}\)  & \(2.80\times10^{-09}\) \\
& & & RADMM       & \(-94.2129\) & \(0.5502\) & \(6616\) & \(6.010\) & \(9.08\times10^{-04}\)  & \(1.69\times10^{-14}\) \\
& & & SOC         & \(-94.2129\) & \(0.5501\) & \(6392\) & \(5.270\) & \(8.24\times10^{-04}\)  & \(7.29\times10^{-15}\) \\
& & & LSALM       & \(-94.2129\) & \(0.5509\) & \(4160\) & \(1.834\) & \(4.41\times10^{-04}\) & \(1.19\times10^{-05}\) \\
& & & SLPG        & \(-94.2130\) & \(0.5506\) & \(1050\) & \(1.339\) & \(1.28\times10^{-03}\) & \(2.84\times10^{-13}\) \\
& & & IPGLanding   & \(-94.2129\) & \(0.5500\) & \(1343\) & \(0.875\) & \(6.52\times10^{-04}\) & \(1.70\times10^{-10}\) \\
\bottomrule
\end{tabular}}
\endgroup
\caption{Comparison on randomly generated SPCA instances.}
\label{tab:comparison-random-n1000-m100-tol1e-4}
\end{table}

\begin{table}[!t]
\centering
\begingroup
\small
\setlength{\tabcolsep}{2.5pt}
\resizebox{\textwidth}{!}{
\begin{tabular}{ccclcccrcc}
\toprule
\(n\) & \(m\) & \(p\) & Method & \(f(X)\) & Sparsity & Iter & Time (s) & Time/Iter (s) & \(\|X^\top X-I_p\|_{\mathrm{F}}\) \\
\midrule
\multirow{7}{*}{200} & \multirow{7}{*}{1000} & \multirow{7}{*}{100} & ManPG-Ada   & \(-51.9269\) & \(0.9938\) & \(14\)  & \(1.434\) & \(1.02\times10^{-01}\)  & \(1.02\times10^{-14}\) \\
& & & PGLanding & \(-51.9269\) & \(0.9938\) & \(10\)  & \(0.813\) & \(8.13\times10^{-02}\)  & \(7.38\times10^{-07}\) \\
& & & RADMM       & \(-51.9269\) & \(0.9938\) & \(94\)  & \(0.140\) & \(1.49\times10^{-03}\)  & \(2.50\times10^{-14}\) \\
& & & SOC         & \(-51.9269\) & \(0.9938\) & \(71\)  & \(0.204\) & \(2.87\times10^{-03}\)  & \(2.92\times10^{-14}\) \\
& & & LSALM       & \(-51.9274\) & \(0.9938\) & \(51\) & \(0.019\) & \(3.73\times10^{-04}\) & \(6.23\times10^{-05}\) \\
& & & SLPG        & \(-51.9269\) & \(0.9938\) & \(11\) & \(0.015\) & \(1.36\times10^{-03}\) & \(1.48\times10^{-15}\) \\
& & & IPGLanding   & \(-51.9269\) & \(0.9938\) & \(8\) & \(0.011\) & \(1.38\times10^{-03}\) & \(1.33\times10^{-09}\) \\
\midrule
\multirow{7}{*}{300} & \multirow{7}{*}{1000} & \multirow{7}{*}{150} & ManPG-Ada   & \(-78.0380\) & \(0.9956\) & \(67\)  & \(10.660\) & \(1.59\times10^{-01}\)  & \(1.74\times10^{-14}\) \\
& & & PGLanding & \(-78.0380\) & \(0.9956\) & \(125\) & \(9.987\) & \(7.99\times10^{-02}\)  & \(8.01\times10^{-09}\) \\
& & & RADMM       & \(-78.0379\) & \(0.9956\) & \(202\) & \(0.496\) & \(2.46\times10^{-03}\)  & \(3.52\times10^{-14}\) \\
& & & SOC         & \(-78.0379\) & \(0.9956\) & \(188\) & \(0.640\) & \(3.40\times10^{-03}\)  & \(4.24\times10^{-14}\) \\
& & & LSALM       & \(-78.0380\) & \(0.9956\) & \(262\) & \(0.159\) & \(6.07\times10^{-04}\) & \(6.11\times10^{-06}\) \\
& & & SLPG        & \(-78.0380\) & \(0.9956\) & \(40\) & \(0.090\) & \(2.25\times10^{-03}\) & \(8.38\times10^{-16}\) \\
& & & IPGLanding   & \(-78.0380\) & \(0.9956\) & \(39\) & \(0.042\) & \(1.08\times10^{-03}\) & \(2.90\times10^{-09}\) \\
\midrule
\multirow{7}{*}{400} & \multirow{7}{*}{1000} & \multirow{7}{*}{200} & ManPG-Ada   & \(-108.5487\) & \(0.9951\) & \(109\) & \(27.940\) & \(2.56\times10^{-01}\)  & \(1.85\times10^{-14}\) \\
& & & PGLanding & \(-108.5487\) & \(0.9951\) & \(780\) & \(59.180\) & \(7.59\times10^{-02}\)  & \(4.37\times10^{-08}\) \\
& & & RADMM       & \(-108.5487\) & \(0.9951\) & \(465\) & \(1.891\) & \(4.07\times10^{-03}\)  & \(5.11\times10^{-14}\) \\
& & & SOC         & \(-108.5487\) & \(0.9951\) & \(308\) & \(1.939\) & \(6.30\times10^{-03}\)  & \(5.32\times10^{-14}\) \\
& & & LSALM       & \(-108.5488\) & \(0.9951\) & \(416\) & \(0.358\) & \(8.61\times10^{-04}\) & \(1.60\times10^{-05}\) \\
& & & SLPG        & \(-108.5488\) & \(0.9951\) & \(99\) & \(0.316\) & \(3.19\times10^{-03}\) & \(1.95\times10^{-11}\) \\
& & & IPGLanding   & \(-108.5487\) & \(0.9951\) & \(82\) & \(0.123\) & \(1.50\times10^{-03}\) & \(1.13\times10^{-10}\) \\
\midrule
\multirow{7}{*}{500} & \multirow{7}{*}{1000} & \multirow{7}{*}{250} & ManPG-Ada   & \(-131.6638\) & \(0.9968\) & \(108\) & \(55.320\) & \(5.12\times10^{-01}\)  & \(3.36\times10^{-14}\) \\
& & & PGLanding & \(-131.6638\) & \(0.9968\) & \(356\) & \(76.720\) & \(2.16\times10^{-01}\) & \(1.25\times10^{-08}\) \\
& & & RADMM       & \(-131.6637\) & \(0.9968\) & \(391\) & \(3.170\) & \(8.11\times10^{-03}\)  & \(5.48\times10^{-14}\) \\
& & & SOC         & \(-131.6637\) & \(0.9968\) & \(386\) & \(4.482\) & \(1.16\times10^{-02}\)  & \(6.39\times10^{-14}\) \\
& & & LSALM       & \(-131.6638\) & \(0.9968\) & \(426\) & \(0.501\) & \(1.18\times10^{-03}\) & \(1.60\times10^{-06}\) \\
& & & SLPG        & \(-131.6638\) & \(0.9968\) & \(81\) & \(0.439\) & \(5.42\times10^{-03}\) & \(2.22\times10^{-16}\) \\
& & & IPGLanding   & \(-131.6638\) & \(0.9968\) & \(81\) & \(0.167\) & \(2.06\times10^{-03}\) & \(2.05\times10^{-09}\) \\
\bottomrule
\end{tabular}}
\endgroup
\caption{Comparison on structured synthetic SPCA instances.}
\label{tab:comparison-synthetic-tol1e4}
\end{table}

\begin{table}[!t]
\centering
\begingroup
\small
\setlength{\tabcolsep}{2.5pt}
\resizebox{\textwidth}{!}{
\begin{tabular}{lcccclcccrcc}
\toprule
Data & \(n\) & \(m\) & \(p\) & \(\gamma\) & Method & \(f(X)\) & Sparsity & Iter & Time (s) & Time/Iter (s) & \(\|X^\top X-I_p\|_{\mathrm{F}}\) \\
\midrule
\multirow{7}{*}{Colon} & \multirow{7}{*}{2000} & \multirow{7}{*}{62} & \multirow{7}{*}{5} & \multirow{7}{*}{1} & ManPG-Ada   & \(-1256.1860\) & \(0.2144\) & \(577\)  & \(0.700\) & \(1.21\times10^{-03}\)  & \(2.40\times10^{-15}\) \\
& & & & & PGLanding & \(-1256.1860\) & \(0.2144\) & \(1608\) & \(2.680\) & \(1.67\times10^{-03}\)  & \(2.86\times10^{-10}\) \\
& & & & & RADMM       & \(-1256.1860\) & \(0.2142\) & \(2554\) & \(1.710\) & \(6.70\times10^{-04}\)  & \(2.90\times10^{-15}\) \\
& & & & & SOC         & \(-1256.1860\) & \(0.2140\) & \(2622\) & \(3.960\) & \(1.51\times10^{-03}\)  & \(2.32\times10^{-15}\) \\
& & & & & LSALM       & \(-1256.1860\) & \(0.2143\) & \(1274\) & \(0.622\) & \(4.88\times10^{-04}\) & \(2.89\times10^{-06}\) \\
& & & & & SLPG        & \(-1256.1860\) & \(0.2144\) & \(591\) & \(0.517\) & \(8.75\times10^{-04}\) & \(3.23\times10^{-16}\) \\
& & & & & IPGLanding   & \(-1256.1860\) & \(0.2141\) & \(575\) & \(0.274\) & \(4.77\times10^{-04}\) & \(1.21\times10^{-09}\) \\
\midrule
\multirow{7}{*}{Colon} & \multirow{7}{*}{2000} & \multirow{7}{*}{62} & \multirow{7}{*}{5} & \multirow{7}{*}{5} & ManPG-Ada   & \(-779.5406\) & \(0.7104\) & \(1545\) & \(1.990\) & \(1.29\times10^{-03}\)  & \(2.40\times10^{-15}\) \\
& & & & & PGLanding & \(-779.5406\) & \(0.7104\) & \(4103\) & \(6.260\) & \(1.53\times10^{-03}\)  & \(4.93\times10^{-10}\) \\
& & & & & RADMM       & \(-779.5406\) & \(0.7104\) & \(7974\) & \(5.500\) & \(6.90\times10^{-04}\)  & \(1.25\times10^{-15}\) \\
& & & & & SOC         & \(-779.5406\) & \(0.7104\) & \(8977\) & \(13.300\) & \(1.48\times10^{-03}\)  & \(1.97\times10^{-15}\) \\
& & & & & LSALM       & \(-779.5406\) & \(0.7106\) & \(4813\) & \(2.860\) & \(5.94\times10^{-04}\) & \(5.42\times10^{-06}\) \\
& & & & & SLPG        & \(-779.5406\) & \(0.7104\) & \(1717\) & \(1.615\) & \(9.41\times10^{-04}\) & \(1.54\times10^{-14}\) \\
& & & & & IPGLanding   & \(-779.5406\) & \(0.7104\) & \(1677\) & \(0.793\) & \(4.73\times10^{-04}\) & \(2.15\times10^{-08}\) \\
\midrule
\multirow{7}{*}{Colon} & \multirow{7}{*}{2000} & \multirow{7}{*}{62} & \multirow{7}{*}{5} & \multirow{7}{*}{10} & ManPG-Ada   & \(-424.4652\) & \(0.8004\) & \(174\) & \(0.280\) & \(1.61\times10^{-03}\)  & \(2.85\times10^{-16}\) \\
& & & & & PGLanding & \(-424.4652\) & \(0.8004\) & \(356\) & \(0.880\) & \(2.47\times10^{-03}\)  & \(3.88\times10^{-08}\) \\
& & & & & RADMM       & \(-424.4651\) & \(0.8004\) & \(772\) & \(0.590\) & \(7.64\times10^{-04}\)  & \(6.43\times10^{-15}\) \\
& & & & & SOC         & \(-424.4652\) & \(0.8004\) & \(678\) & \(1.110\) & \(1.64\times10^{-03}\)  & \(2.37\times10^{-15}\) \\
& & & & & LSALM       & \(-424.5024\) & \(0.8004\) & \(565\) & \(0.272\) & \(4.81\times10^{-04}\) & \(6.59\times10^{-05}\) \\
& & & & & SLPG        & \(-424.4652\) & \(0.8004\) & \(165\) & \(0.157\) & \(9.52\times10^{-04}\) & \(8.39\times10^{-15}\) \\
& & & & & IPGLanding   & \(-424.4652\) & \(0.8004\) & \(180\) & \(0.096\) & \(5.33\times10^{-04}\) & \(6.76\times10^{-10}\) \\
\bottomrule
\end{tabular}}
\endgroup
\caption{Comparison on the Colon data set.}
\label{tab:comparison-realworld-tol1e-4}
\end{table}

\emph{Algorithmic parameters.}
We implement SLPG in MATLAB based on PySTOP v0.2.2
and LSALM following~\citet{zhu2026primal}.
We set \(t_0=1/\|L\|_2\) for these comparisons.

For PGLanding, we set
\((t,\varepsilon,\lambda)=(t_0,0.5,0.5)\),
\(\rho=10^{-3}\), \(\eta_0=1\), \(\beta_{\mathrm{ls}}=0.5\),
and \(\vartheta_{\mathrm{P}}=1.1\).
In the numerical implementation, each tangential proximal subproblem
is solved by a regularized semismooth Newton method with an adaptive
residual tolerance between \(10^{-12}\) and \(10^{-7}\) and at most
20 inner iterations. This approximates the exact subproblem solution
used in the PGLanding convergence analysis.
For IPGLanding, we use the multiplier update
in~\cref{iprox_dual_update}, initialize
\(\Gamma_{0,0}=0_{p\times p}\), and set
\(t=2t_0\), \(\varepsilon=0.5\), \(\lambda=5/12\),
\(\eta_0=1.20\), \(\theta_{\mathrm{d}}=1.95\),
\(\omega_{\mathrm{mem}}=0.95\), \(\rho=c=10^{-3}\),
\(\beta_{\mathrm{ls}}=0.4\), and \(\vartheta_{\mathrm{IP}}=1.1\).

For the comparison runs, ManPG-Ada~\citep{chen2020proximal} uses the
parameter settings supplied in the authors' official implementation.
For SLPG, we use the default parameter settings of the reference
implementation, except that the Barzilai--Borwein stepsize
\citep{barzilai1988two} is capped at \(2.5t_0\) to improve numerical
stability. No final polar-decomposition postprocessing is applied.
For SOC, we set the penalty parameter to \(\|L\|_2\).
Following~\citet{zhu2026primal}, RADMM uses
\(\gamma_{\mathrm{R}}=10^{-12}\), \(\rho_{\mathrm{R}}=\|L\|_2\),
and \(\eta_{\mathrm{R}}=(2\|L\|_2)^{-1}\).
For LSALM, we adopt the parameter settings in~\citet{zhu2026primal} and set
\(\rho_{\mathrm{LS}}=10\), \(r_{\mathrm{LS}}=15\),
\(\alpha_{\mathrm{LS}}=\operatorname{round}(0.07\sqrt{np})\),
\(\beta_{\mathrm{LS}}=0.5\), \(\lambda_{\mathrm{LS}}=t_0\),
\(\epsilon_{\mathrm{d}}=10^{-10}\),
\(R_X^{\operatorname{op}}=10\), and \(R_Y=1000\).
We use these settings on the random instances. To facilitate convergence
on the structured synthetic and Colon instances, we set
\(r_{\mathrm{LS}}=25\) for the former and
\((\rho_{\mathrm{LS}},\alpha_{\mathrm{LS}})=(100,600)\) for the latter,
leaving the other parameters unchanged.

The results in
\cref{tab:comparison-random-n1000-m100-tol1e-4,tab:comparison-synthetic-tol1e4,tab:comparison-realworld-tol1e-4}
show that all seven methods terminate before reaching the iteration limit on all 11 instances
and satisfy both the stopping conditions and
\(\|X-X_{\mathrm{ref}}\|_{\mathrm{F}}<0.1\).
IPGLanding achieves the shortest CPU times across the tested instances.
LSALM also has a low per-iteration computational cost.
The lower running times of IPGLanding relative to PGLanding under the same
stopping criteria support the practical value of the inexact proximal
approach on these instances.

ManPG-Ada, RADMM, and SOC explicitly enforce orthogonality through
retraction or orthogonal projection, yielding feasibility residuals on the
order of \(10^{-16}\) to \(10^{-14}\) in these experiments.
SLPG applies a Newton--Schulz-type normal correction after each tangential
step and attains feasibility residuals on the order of
\(10^{-16}\) to \(10^{-11}\), without final orthogonalization.
The feasibility residuals are on the order of
\(10^{-10}\) to \(10^{-8}\) for IPGLanding and
\(10^{-10}\) to \(10^{-7}\) for PGLanding. LSALM returns larger
feasibility residuals, on the order of \(10^{-6}\) to \(10^{-5}\).

\FloatBarrier
\section{Conclusion}
\label{sec:conclusion}
This work developed three retraction-free methods for nonsmooth composite
optimization over the Stiefel manifold.  All three combine an (approximately) tangential direction to optimize the objective with a normal direction to reduce infeasibility.
 SubLanding uses a tangential subgradient direction, and PGLanding solves a tangential proximal subproblem. IPGLanding solves the proximal subproblem approximately and applies an explicit correction to control violations of the tangency condition.
Under the respective assumptions, we established outer-iteration complexity
bounds for stationarity and feasibility. For SubLanding with \(\alpha=0\),
we further proved local \(R\)-linear convergence under local weak sharpness
and the stated initialization, parameter, and stepsize conditions.
For IPGLanding, the KL property and an additional asymptotic Bregman-error
condition yield finite length and convergence of the primal sequence to a
first-order stationary point.

Numerical experiments on SPCA illustrated the effects of normal-direction
scaling and stepsize choices in SubLanding. On the tested instances,
IPGLanding achieved favorable running times, while both proximal methods
maintained small feasibility residuals.  Future work includes extending Landing methods to nonsmooth composite optimization under general equality constraints. Another direction is to design stochastic Landing methods for nonsmooth problems while controlling both stationarity and feasibility.
 
\begin{samepage}
\acks{BG was supported by the National Key R\&D Program of China
(grant 2023YFA1009300) and the National Natural Science Foundation of China (grant 12288201).
RG was supported by the National Key R\&D Program of China (grant
2022YFA1003800), the National Natural Science Foundation of China (grant
12671364), and the Natural Science Foundation of Tianjin (grant
25JCJQJC00300).
The authors declare no competing interests.
}
\end{samepage}

\appendix
\crefalias{section}{appendix}
\crefalias{subsection}{appendix}
\crefalias{subsubsection}{appendix}

\section{Technical Proofs}
\label{app:technical-proofs}
This appendix collects the proofs and auxiliary results used in the main
text.
Throughout the appendix, we use the residual notation
\[
E(X):=X^\top X-I_p.
\]
For proofs involving an iterate sequence \(\{X_k\}\), we write
\(E_k:=E(X_k)\).

\subsection{Layer-Orbit Identity}
\label{app:layer_orbit_identity}

\begin{proposition}
\label[proposition]{prop:layer_orbit_identity}
Let \(X\in\R_*^{n\times p}\). For every \(Y\in\R^{n\times p}\),
\[
Y^\top Y=X^\top X
\quad\Longleftrightarrow\quad
Y=QX\text{ for some }Q\in\mathrm O(n).
\]
Equivalently, the  layered manifold \(\St_{X^\top X}\) coincides with the
left orthogonal orbit \(\{QX:Q\in\mathrm O(n)\}\).
\end{proposition}

\begin{proof}
To verify the set equality in~\cref{eq:layer_orbit_identity}, first observe
that \((QX)^\top(QX)=X^\top X\) for every \(Q\in\mathrm O(n)\), so
\(QX\in\St_{X^\top X}\). Conversely, let \(Y\in\St_{X^\top X}\) and set
\(G=X^\top X=Y^\top Y\succ0\). Then
\(U=XG^{-1/2}\) and \(V=YG^{-1/2}\) belong to \(\St(n,p)\). Choose
orthogonal completions \([U\ U_\perp],[V\ V_\perp]\in\mathrm O(n)\) and set
\(Q=[V\ V_\perp][U\ U_\perp]^\top\). Since \(Q\in\mathrm O(n)\) and
\(QU=V\), we have \(QX=Y\). Hence every point of
\(\St_{X^\top X}\) is obtained in this way. The same calculation gives
\(E(QX)=X^\top Q^\top QX-I_p=E(X)\).
\end{proof}

\subsection{Proof of the Safe-Region Objective Bounds}
\label{app:safe_region_objective_bounds}

\begin{proof}
For every \(X\in\St(n,p)^\varepsilon\), the eigenvalues of
\(X^\top X\) lie in \([1-\varepsilon,1+\varepsilon]\). Consequently,
\[
\sigma_{\min}(X)\geq\sqrt{1-\varepsilon},\quad
\|X\|_2\leq\sqrt{1+\varepsilon},\quad
\|X\|_{\mathrm{F}}\leq\sqrt{p(1+\varepsilon)}.
\]
Every line segment joining two points of \(\St(n,p)^\varepsilon\) lies in
the ball of radius \(\sqrt{p(1+\varepsilon)}\). Hence every point \(Z\) on
such a segment satisfies
\[
\|\nabla h(Z)\|_{\mathrm{F}}
\leq
\|\nabla h(0)\|_{\mathrm{F}}+L_h\|Z\|_{\mathrm{F}}
\leq
\|\nabla h(0)\|_{\mathrm{F}}+L_h\sqrt{p(1+\varepsilon)}.
\]
Integration along the segment between \(X\) and \(Y\), together with the
\(L_r\)-Lipschitz continuity of \(r\), gives
\[
|f(X)-f(Y)|
\leq
\bigl(\|\nabla h(0)\|_{\mathrm{F}}+L_h\sqrt{p(1+\varepsilon)}+L_r\bigr)
\|X-Y\|_{\mathrm{F}}.
\]
For any \(F\in\partial f(X)\),
\cref{composite_subdifferential} gives
\(F=\nabla h(X)+\xi\) for some \(\xi\in\partial r(X)\). Therefore,
\cref{nonsmooth_subgradient_bound} and the same gradient estimate yield
\[
\|F\|_{\mathrm{F}}
\leq
\|\nabla h(0)\|_{\mathrm{F}}+L_h\sqrt{p(1+\varepsilon)}+L_r
=L_f^\varepsilon.
\]
This proves~\cref{safe_region_objective_bounds}.
\end{proof}

\subsection{Proof of \texorpdfstring{\cref{prop:orthogonal_stationarity}}{Proposition~\ref{prop:orthogonal_stationarity}}}
\label{app:orthogonal_stationarity_proof}

\begin{proof}
We first prove the stationarity equivalence. Let
\(F=\nabla h(X)+\xi\). If~\cref{explicit_nonsmooth_KKT} holds, then
\(F=2X\Gamma\) with \(\Gamma\) symmetric, so
\(\skew(FX^\top)=0\). Conversely, if \(E(X)=0\) and
\(0\in\partial_{\St_{X^\top X}}^g f(X)\), then
\(F-XF^\top X=0\) for some \(\xi\in\partial r(X)\). Multiplication by
\(X^\top\) shows that \(F^\top X\) is symmetric. Thus
\(\Gamma=F^\top X/2\) satisfies~\cref{explicit_nonsmooth_KKT}.

We next establish the claims involving the surrogate.
Let \(\widehat Q=\widehat Q_X\) and
\(\widehat X=\widehat QX\). Since \(\widehat Q\) is orthogonal,
\(\widehat X\in\St_{X^\top X}\) and \(E(\widehat X)=E(X)\). Also,
\(\|X\|_2\leq\sqrt{1+\varepsilon}\) gives
\[
\|\widehat X-X\|_{\mathrm{F}}
\leq\sqrt{1+\varepsilon}\|\widehat Q-I_n\|_{\mathrm{F}}
=\frac\mu2\Theta_\mu(X).
\]
Because \(X\) has full column rank, the
linear map \(A\mapsto AX\) is surjective: for every
\(Y\in\R^{n\times p}\), the choice
\(A=Y(X^\top X)^{-1}X^\top\) gives \(AX=Y\). The exact chain rule
\citep[Chapter~10]{rockafellar1998variational} for the regular function
\(f=h+r\) therefore gives
\(\partial_Q f(QX)=\{FX^\top:F\in\partial f(QX)\}\).
The nonsmooth Fermat rule on \(\mathrm O(n)\) yields
\(F_X\in\partial f(\widehat X)\) such that
\begin{equation}\label{app_orthogonal_projected_fermat}
0
=
\operatorname{Proj}_{\mathrm T_{\widehat Q}\mathrm O(n)}
\left(F_XX^\top+\frac1\mu(\widehat Q-I_n)\right).
\end{equation}
For \(Q\in\mathrm O(n)\),
\(\operatorname{Proj}_{\mathrm T_Q\mathrm O(n)}(Z)=Q\skew(Q^\top Z)\).
Applying this identity to~\cref{app_orthogonal_projected_fermat} and
conjugating the resulting skew-symmetric part by \(\widehat Q\) give
\[
\begin{aligned}
0&=\skew\left(F_X\widehat X^\top+\mu^{-1}(I_n-\widehat Q^\top)\right),\\
2\skew(F_X\widehat X^\top)\widehat X
&=-\mu^{-1}(\widehat Q-\widehat Q^\top)\widehat X.
\end{aligned}
\]
The left-hand side of the second identity belongs to
\(\partial_{\St_{\widehat X^\top \widehat X}}^g f(\widehat X)\) by
\cref{canonical_riemannian_subdifferential}. Therefore,
\[
\dist\left(0,\partial_{\St_{\widehat X^\top \widehat X}}^g f(\widehat X)\right)
\leq\frac1\mu\|(\widehat Q-\widehat Q^\top)\widehat X\|_{\mathrm{F}}
\leq\frac{2\sqrt{1+\varepsilon}}{\mu}\|\widehat Q-I_n\|_{\mathrm{F}}=\Theta_\mu(X).
\]
If \(\Theta_\mu(X)=0\), then \(\widehat Q=I_n\) and
\(\widehat X=X\). The second identity then gives an element equal to zero
in \(\partial_{\St_{X^\top X}}^g f(X)\).
If \(\Theta_\mu(X)=0\) and \(E(X)=0\), this conclusion and
the stationarity equivalence established above give the KKT system in
\cref{def:first_order_stationarity}. Conversely, assume that \(X\) satisfies
this KKT system. Then \(E(X)=0\), and there is \(F\in\partial f(X)\) such
that \(FX^\top\) is symmetric. For every \(Q\in\mathrm O(n)\), set
\(\Delta=Q-I_n\). Since
\(\Delta+\Delta^\top=-\Delta\Delta^\top\),
\(\|FX^\top\|_2\leq L_f^\varepsilon\), and \(\|X\|_2=1\), weak convexity
gives
\[
\begin{aligned}
f(QX)-f(X)
&\geq\langle F,\Delta X\rangle-\frac{L_h}{2}\|\Delta X\|_{\mathrm{F}}^2
\geq-\frac{L_f^\varepsilon+L_h}{2}\|\Delta\|_{\mathrm{F}}^2,\\
f(QX)+\frac{1}{2\mu}\|Q-I_n\|_{\mathrm{F}}^2
&\geq f(X)+\frac12(\mu^{-1}-L_h-L_f^\varepsilon)
\|Q-I_n\|_{\mathrm{F}}^2.
\end{aligned}
\]
Therefore, if \(\mu<1/(L_h+L_f^\varepsilon)\), then \(I_n\) is the unique
minimizer in~\cref{orthogonal_envelope}, and hence \(\Theta_\mu(X)=0\).
Together with \(E(X)=0\), this proves the stated equivalence.
\end{proof}

\subsection{Proof of \texorpdfstring{\cref{prop:sublanding_exact_lambda_range}}{Proposition~\ref{prop:sublanding_exact_lambda_range}}}
\label{app:normal_contraction_proof}

\begin{proof}
Write \(\kappa_\varepsilon(\zeta)=1-q_\varepsilon(\zeta)\).
Write \(E=E(X)\). Direct expansion gives
\(E(Y_\zeta(X))=(1-2\zeta)E+(\zeta^2-2\zeta)E^2+\zeta^2E^3\).
If
\(E=U\operatorname{diag}(s_1,\ldots,s_p)U^\top\), then
\(|s_i|\leq\varepsilon\) and
\[
E(Y_\zeta(X))
=
U\operatorname{diag}\left(
s_i\{1-\zeta(1+s_i)(2-\zeta s_i)\}
\right)_{i=1}^pU^\top.
\]
Taking Frobenius norms proves~\cref{exact_normal_contraction}.
Set \(a_\lambda(s)=\lambda(1+s)(2-\lambda s)\). Because \(0\in[-\varepsilon,\varepsilon]\), uniform contraction implies
\(|1-2\lambda|<1\), hence \(0<\lambda<1\). In this interval
\(a_\lambda(s)>0\), and
\(|1-a_\lambda(s)|<1\) is equivalent to \(a_\lambda(s)<2\).
The function \(a_\lambda(s)=2\lambda+\lambda(2-\lambda)s-\lambda^2s^2\) is concave and has vertex
\(s_{\mathrm{v}}=(2-\lambda)/(2\lambda)\). This vertex belongs to
\([-\varepsilon,\varepsilon]\) exactly when
\(\lambda\geq\lambda_{\mathrm{v}}(\varepsilon)=2/(1+2\varepsilon)\). For
\(0<\lambda<\lambda_{\mathrm{v}}(\varepsilon)\), the maximum is
\(a_\lambda(\varepsilon)\), which reaches \(2\) at its smaller positive root
\(\lambda_{\mathrm{e}}(\varepsilon)=2/(1+\varepsilon+\sqrt{1-\varepsilon^2})\).
The relation
\(\lambda_{\mathrm{e}}(\varepsilon)\leq\lambda_{\mathrm{v}}(\varepsilon)\)
is equivalent to \(\varepsilon\leq1/\sqrt2\). If instead
\(\varepsilon\geq1/\sqrt2\), then
\(\lambda_{\mathrm{v}}(\varepsilon)\leq\lambda_{\mathrm{e}}(\varepsilon)\), and for
\(\lambda\geq\lambda_{\mathrm{v}}(\varepsilon)\), one has
\(\max_{|s|\leq\varepsilon}a_\lambda(s)=a_\lambda(s_{\mathrm{v}})=(\lambda+2)^2/4\).
This value reaches \(2\) at
\(\lambda=2(\sqrt2-1)\). These two cases prove the exact range.
They also show \(\lambda\varepsilon<1\). Finally, assume \(0<\zeta\leq1/2\), and write \(r_\zeta(s)=1-\zeta(1+s)(2-\zeta s)\). 
Its derivative is \(r_\zeta'(s)=\zeta(2\zeta s+\zeta-2)<0\) on \([-\varepsilon,\varepsilon]\). 
Furthermore, \(r_\zeta(-\varepsilon)>0\) and
\(r_\zeta(-\varepsilon)+r_\zeta(\varepsilon)=2(1-2\zeta+\zeta^2\varepsilon^2)\geq0\).
Thus \(q_\varepsilon(\zeta)=r_\zeta(-\varepsilon)
=1-\zeta(1-\varepsilon)(2+\zeta\varepsilon)\) and
\(\kappa_\varepsilon(\zeta)=\zeta(1-\varepsilon)
(2+\zeta\varepsilon)\geq2(1-\varepsilon)\zeta\).
\end{proof}

\subsection{Proof of \texorpdfstring{\cref{thm:alpha_sublanding_complexity}}{Theorem~\ref{thm:alpha_sublanding_complexity}}}
\label{app:sublanding_master_descent_proof}

\begin{proof}
The safe-region subgradient bound in
\cref{safe_region_objective_bounds} gives
\begin{equation}\label{unified_tangential_bounds}
\|W_k\|_{\mathrm{F}}\leq C_W=2L_f^\varepsilon\sqrt{1+\varepsilon},
\quad
\|W_kX_k\|_{\mathrm{F}}\leq B_{\mathrm{T}}=2L_f^\varepsilon(1+\varepsilon).
\end{equation}
Set \(C_{\mathrm{mix}}=2B_{\mathrm{T}}\sqrt{1+\varepsilon}\).
The constants in the theorem may be chosen as
\begin{equation*}
c_{\mathrm{N}}
=
\begin{cases}
\kappa_\varepsilon(\lambda)/\lambda,&\alpha=0,\\
2(1-\varepsilon),&0<\alpha\leq1,
\end{cases}
\quad
\eta_{\mathrm{N}}(\alpha)
=
\begin{cases}
+\infty,&\alpha=0,\\
(2\lambda)^{-1/\alpha},&0<\alpha\leq1,
\end{cases}
\end{equation*}
and
\begin{equation*}
\eta^*
=
\min\left\{
\eta_{\mathrm{N}}(\alpha),
\frac{c_{\mathrm{N}}}{2C_{\mathrm{mix}}},
\left(
\frac{\varepsilon c_{\mathrm{N}}\lambda}{2B_{\mathrm{T}}^2}
\right)^{1/(2-\alpha)}
\right\}>0.
\end{equation*}
Thus, \(\eta_k\leq\eta_{\mathrm{N}}(\alpha)\) implies
\(\kappa_\varepsilon(\zeta_k)
\geq c_{\mathrm{N}}\lambda\eta_k^\alpha\). 
For the descent argument, let
\(\widetilde X_k=X_k-\eta_kW_kX_k\) and define
\begin{equation*}
\Omega=
\left\{X:\|X\|_2\leq
\sqrt{(1+\varepsilon)\bigl(1+(\eta^*)^2C_W^2\bigr)}\right\},
\quad
\widetilde L_f=\max_{X\in\Omega}\|\nabla h(X)\|_{\mathrm{F}}+L_r.
\end{equation*}
The safe-region and tangential bounds imply that \(X_k\),
\(\widetilde X_k\), and \(Q\widetilde X_k\), for every
\(Q\in\mathrm O(n)\), lie in \(\Omega\). Hence \(\widetilde L_f\) is a
common Lipschitz constant for \(f\) at the matrices used below. Finally, set
\begin{equation*}
\mu^*
=\frac{1}{2\{L_h(1+\varepsilon)
+L_f^\varepsilon\sqrt{1+\varepsilon}\}}>0,
\quad
\nu^*
=\frac{4\widetilde L_f\sqrt{1+\varepsilon}}{c_{\mathrm{N}}}>0.
\end{equation*}
Expansion of~\cref{sublanding_iteration}, using
\(X_k^\top W_k^{}X_k^{}+(X_k^\top W_k^{}X_k^{})^\top=0\), gives
\begin{equation}\label{unified_residual_bound}
\|E_{k+1}\|_{\mathrm{F}}
\leq
\left(1-\kappa_\varepsilon(\zeta_k)
+C_{\mathrm{mix}}\eta_k\zeta_k\right)\|E_k\|_{\mathrm{F}}
+B_{\mathrm{T}}^2\eta_k^2.
\end{equation}
The definitions of \(c_{\mathrm{N}}\), \(\eta_{\mathrm{N}}(\alpha)\), and
\(\eta^*\) imply that, whenever \(\|E_k\|_{\mathrm{F}}\leq\varepsilon\),
\[
\begin{aligned}
\kappa_\varepsilon(\zeta_k)
&\geq c_{\mathrm{N}}\zeta_k=c_{\mathrm{N}}\lambda\eta_k^\alpha,\quad
C_{\mathrm{mix}}\eta_k\zeta_k\leq\frac{\kappa_\varepsilon(\zeta_k)}2,\quad
B_{\mathrm{T}}^2\eta_k^2\leq\frac{\varepsilon\kappa_\varepsilon(\zeta_k)}2,\\
\|E_{k+1}\|_{\mathrm{F}}
&\leq
\left(1-\frac{\kappa_\varepsilon(\zeta_k)}2\right)\|E_k\|_{\mathrm{F}}
+B_{\mathrm{T}}^2\eta_k^2
\leq
\left(1-\frac{\kappa_\varepsilon(\zeta_k)}2\right)\varepsilon
+\frac{\varepsilon\kappa_\varepsilon(\zeta_k)}2
=\varepsilon.
\end{aligned}
\]
Induction from \(X_0\in\St(n,p)^\varepsilon\) proves the first conclusion.

For the descent estimate, suppress the iteration index and write
\(X=X_k\), \(F=F_k\), \(\eta=\eta_k\),
\(\zeta=\zeta_k\), \(W=W_k\), \(\widehat Q=\widehat Q_{X_k}\), and
\(\Delta=\widehat Q-I_n\).
Let \(\widetilde X=(I_n-\eta W)X,Q^{\mathrm{cmp}}=\widehat Q\exp(\eta W)\). 
Since \(W^\top=-W\), the matrix \(\exp(\eta W)\) is orthogonal, and hence
\(Q^{\mathrm{cmp}}\in\mathrm O(n)\).
Let \(R^{\exp}=\exp(\eta W)-I_n-\eta W\). Applying the scalar integral
remainder for \(e^{\mathrm i t}\) to the spectral decomposition of \(W\),
together with direct multiplication, gives
\begin{equation}\label{app_exponential_remainder}
\begin{aligned}
\|R^{\exp}\|_{\mathrm{F}}&\leq\frac12\eta^2\|W\|_{\mathrm{F}}^2,\\
Q^{\mathrm{cmp}}\widetilde X-\widehat QX
&=\widehat Q\bigl[\exp(\eta W)(I_n-\eta W)-I_n\bigr]X,\\
\exp(\eta W)(I_n-\eta W)-I_n
&=-\eta^2W^2+R^{\exp}-\eta R^{\exp}W.
\end{aligned}
\end{equation}
The bounds in~\cref{unified_tangential_bounds}, the restriction
\(\eta\leq\eta^*\), and~\cref{app_exponential_remainder} therefore give
\begin{equation}\label{app_comparison_point_remainder}
\|Q^{\mathrm{cmp}}\widetilde X-\widehat QX\|_{\mathrm{F}}
\leq C_2\eta^2,
\end{equation}
for some \(C_2>0\).
Also, \(Q^{\mathrm{cmp}}-I_n=\Delta+\eta\widehat QW+\widehat QR^{\exp}\). 
Since \(\|\Delta\|_{\mathrm{F}}\leq2\sqrt n\), expansion of the squared norm and
\cref{app_exponential_remainder} give
\begin{equation}\label{app_orthogonal_penalty_expansion}
\|Q^{\mathrm{cmp}}-I_n\|_{\mathrm{F}}^2
\leq
\|\Delta\|_{\mathrm{F}}^2
+2\eta\inner{\Delta}{\widehat QW}
+C_3\eta^2,
\end{equation}
for some \(C_3>0\).
Orthogonality of \(\widehat Q=I_n+\Delta\) and
\(W=FX^\top-XF^\top\) give
\[
\Delta+\Delta^\top=-\Delta\Delta^\top,\quad
\inner{\Delta}{\widehat QW}
=2\inner{\widehat QX-X}{F}
+\inner{\Delta\Delta^\top}{FX^\top}.
\]
Weak convexity, optimality of \(\widehat Q\) in
\cref{orthogonal_envelope}, and the safe-region bounds imply
\[
\begin{aligned}
\inner{\widehat QX-X}{F}
&\leq f(\widehat QX)-f(X)+\frac{L_h}{2}\|\widehat QX-X\|_{\mathrm{F}}^2\leq-\frac1{2\mu}\|\Delta\|_{\mathrm{F}}^2+\frac{L_h(1+\varepsilon)}2\|\Delta\|_{\mathrm{F}}^2,\\
\left|\inner{\Delta\Delta^\top}{FX^\top}\right|
&\leq L_f^\varepsilon\sqrt{1+\varepsilon}\|\Delta\|_{\mathrm{F}}^2.
\end{aligned}
\]
Thus, setting
\(A=L_h(1+\varepsilon)+L_f^\varepsilon\sqrt{1+\varepsilon}\), the preceding bounds give
\begin{equation}\label{app_critical_cross_bound}
\inner{\Delta}{\widehat QW}
\leq
-\left(\frac1\mu-A\right)\|\Delta\|_{\mathrm{F}}^2.
\end{equation}
Because \(Q^{\mathrm{cmp}}\in\mathrm O(n)\), it is a feasible choice in the
definition of \(f_\mu(\widetilde X)\). Hence
\[
f_\mu(\widetilde X)
\leq f(Q^{\mathrm{cmp}}\widetilde X)
+\frac{1}{2\mu}\|Q^{\mathrm{cmp}}-I_n\|_{\mathrm{F}}^2.
\]
By~\cref{app_comparison_point_remainder} and Lipschitz continuity of \(f\),
\(f(Q^{\mathrm{cmp}}\widetilde X)
\leq f(\widehat QX)+\widetilde L_fC_2\eta^2\).
Since
\(f_\mu(X)=f(\widehat QX)+\|\Delta\|_{\mathrm{F}}^2/(2\mu)\), combining these
inequalities with
\cref{app_orthogonal_penalty_expansion,app_critical_cross_bound} gives
\[
f_\mu(\widetilde X)
\leq
f_\mu(X)
-\eta\left(\frac1{\mu^2}-\frac A\mu\right)\|\Delta\|_{\mathrm{F}}^2
+C_4\eta^2.
\]
Here \(C_4=\widetilde L_fC_2+C_3/(2\mu)\).
The upper bound on \(\mu\) gives
\(\mu^{-2}-A/\mu\geq1/(2\mu^2)\). By~\cref{Theta_definition},
\(\|\Delta\|_{\mathrm{F}}^2/(2\mu^2)=\Theta_\mu^2(X)/[8(1+\varepsilon)]\). Therefore,
\begin{equation}\label{app_tangential_envelope_descent}
\frac{\eta}{8(1+\varepsilon)}\Theta_\mu^2(X)
\leq
f_\mu(X)-f_\mu(\widetilde X)+C_4\eta^2.
\end{equation}
The set \(\Omega\) is invariant under left multiplication by orthogonal
matrices. Moreover, \(\widetilde X\in\Omega\) by construction, and the
safe-region conclusion already proved gives
\(X_{k+1}\in\St(n,p)^\varepsilon\subseteq\Omega\). For arbitrary
\(U,V\in\Omega\), let \(Q_V\) be a minimizer in the definition of
\(f_\mu(V)\). Orthogonal invariance of \(\Omega\) and the
\(\widetilde L_f\)-Lipschitz continuity of \(f\) on \(\Omega\) give
\[
f_\mu(U)-f_\mu(V)
\leq f(Q_VU)-f(Q_VV)
\leq\widetilde L_f\|U-V\|_{\mathrm{F}}.
\]
Exchanging \(U\) and \(V\) shows that \(f_\mu\) is
\(\widetilde L_f\)-Lipschitz continuous on \(\Omega\).
The normal displacement is
\(X_{k+1}-\widetilde X=-\zeta XE\), so the envelope Lipschitz bound
gives
\begin{equation}\label{app_normal_envelope_change}
f_\mu(X_{k+1})
\leq
f_\mu(\widetilde X)
+\widetilde L_f\zeta\sqrt{1+\varepsilon}\|E\|_{\mathrm{F}}.
\end{equation}
The preceding argument also gives
\(\|E_{k+1}\|_{\mathrm{F}}\leq[1-\kappa_\varepsilon(\zeta)/2]\|E\|_{\mathrm{F}}
+B_{\mathrm{T}}^2\eta^2\).
Adding \(\nu\) times this bound to
\cref{app_tangential_envelope_descent,app_normal_envelope_change}
yields
\[
\mathcal L_{\mu,\nu}(X_{k+1})\leq\mathcal L_{\mu,\nu}(X)
-\frac{\eta}{8(1+\varepsilon)}\Theta_\mu^2(X)
-\left\{\frac\nu2\kappa_\varepsilon(\zeta)-\widetilde L_f\zeta\sqrt{1+\varepsilon}\right\}\|E\|_{\mathrm{F}}
+(C_4+\nu B_{\mathrm{T}}^2)\eta^2.
\]
Since \(\kappa_\varepsilon(\zeta)\geq c_{\mathrm{N}}\zeta\) and
\(\nu\geq\nu^*
=4\widetilde L_f\sqrt{1+\varepsilon}/c_{\mathrm{N}}\),
\[
\frac{\nu}{2}\kappa_\varepsilon(\zeta)
-\widetilde L_f\zeta\sqrt{1+\varepsilon}
\geq\frac{\nu}{4}\kappa_\varepsilon(\zeta)
\geq\frac{\nu c_{\mathrm{N}}\lambda}{4}\eta^\alpha.
\]
Consequently, with \(C_{\mathrm{S}}=C_4+\nu B_{\mathrm{T}}^2\), rearrangement gives
\begin{equation}\label{alpha_sublanding_master_descent}
\frac{\eta_k}{8(1+\varepsilon)}\Theta_\mu^2(X_k)
+\frac{\nu c_{\mathrm{N}}\lambda}{4}\eta_k^\alpha\|E_k\|_{\mathrm{F}}
\leq
\mathcal L_{\mu,\nu}(X_k)
-\mathcal L_{\mu,\nu}(X_{k+1})
+C_{\mathrm{S}}\eta_k^2.
\end{equation}
The quantity \(\Delta_{\mathrm{S},0}\) is finite because the safe region is compact
and \(\mathcal L_{\mu,\nu}\) is continuous. Summing
\cref{alpha_sublanding_master_descent} over \(k=0,\ldots,K-1\) telescopes the
merit-function terms and yields
\[
\sum_{k=0}^{K-1}
\left\{\frac{\eta_k}{8(1+\varepsilon)}\Theta_\mu^2(X_k)
+\frac{\nu c_{\mathrm{N}}\lambda}{4}\eta_k^\alpha\|E(X_k)\|_{\mathrm{F}}\right\}
\leq
\Delta_{\mathrm{S},0}+C_{\mathrm{S}}\sum_{k=0}^{K-1}\eta_k^2.
\]
Taking the minimum on the left proves
\cref{alpha_general_best_iterate}. At a minimizing index \(j_K\), both terms
are nonnegative, so discarding either one gives separate bounds on
\(\Theta_\mu(X_{j_K})\) and \(\|E(X_{j_K})\|_{\mathrm{F}}\). For part~(a),
\(\sum_{k=0}^{K-1}\eta_k^2=\eta_0^2\) and
\(\eta_{j_K}=\eta_0/\sqrt K\), which proves
\cref{alpha_stepsize_rates}. 
For part~(b), write the combined residual term as
\[
G_k
=\frac{\eta_k}{8(1+\varepsilon)}\Theta_\mu^2(X_k)
+\frac{\nu c_{\mathrm{N}}\lambda}{4}\eta_k^\alpha
\|E(X_k)\|_{\mathrm{F}},
\]
and set \(m=\lfloor K/2\rfloor\) for \(K\geq2\).
The safe region is compact and \(\mathcal L_{\mu,\nu}\) is continuous,
so its oscillation
\[
M_{\mathcal L}
=\max_{X\in\St(n,p)^\varepsilon}\mathcal L_{\mu,\nu}(X)
-\min_{X\in\St(n,p)^\varepsilon}\mathcal L_{\mu,\nu}(X)
\]
is finite. Summing~\cref{alpha_sublanding_master_descent} only over
\(k=m,\ldots,K-1\), and using \(K/m\leq3\), gives
\[
\sum_{k=m}^{K-1}G_k
\leq M_{\mathcal L}
+C_{\mathrm S}\eta_0^2\sum_{k=m}^{K-1}\frac1{k+1}
\leq M_{\mathcal L}+C_{\mathrm S}\eta_0^2\log(K/m)
\leq M_{\mathcal L}+C_{\mathrm S}\eta_0^2\log3
=:C_{\mathrm{tail}}.
\]
Because \(K-m\geq K/2\) and \(j_K\) minimizes \(G_k\) over all
\(0\leq k<K\),
\[
G_{j_K}\leq\min_{m\leq k<K}G_k
\leq\frac{C_{\mathrm{tail}}}{K-m}
\leq\frac{2C_{\mathrm{tail}}}{K}.
\]
Finally, \(\eta_{j_K}\geq\eta_0/\sqrt K\). Bounding the two
nonnegative terms in \(G_{j_K}\) separately yields
\cref{alpha_stepsize_rates} without logarithmic factors.
The stated complexity bounds follow by solving the resulting rate
inequalities for \(K\).
\end{proof}

\subsection{A Sufficient Condition for Local Weak Sharpness in SPCA}

In this subsection, \(\|U\|_\infty=\max_{i,j}|U_{ij}|\) denotes the
entrywise \(\ell_\infty\)-norm.
The following result gives concrete sufficient conditions for weak sharpness.
They are motivated by the restricted-injectivity and nondegenerate-source
characterization of \citet[Theorem~4.6]{fadili2023sharp} for linearly
constrained sparse recovery. That result does not apply directly to
the nonlinear Stiefel constraint, so the required estimate is proved directly
on the Stiefel manifold.

\begin{proposition}
\label[proposition]{prop:spca_local_weak_sharpness}
Consider \(f(X)=-\operatorname{tr}(X^\top LX)/2+\gamma\|X\|_1\),
where \(L\in\mathbb S^n\) and \(\gamma>0\). Let \(q\geq1\) be an integer and
\(\mathcal X_\star=\{X^{(1)},\ldots,X^{(q)}\}\subseteq\St(n,p)\)
be a finite set on which \(f\) is constant. For
\(Y\in\mathcal X_\star\), write
\(S_Y=\operatorname{supp}(Y)=\{(i,j):Y_{ij}\ne0\}\).
Assume that there is \(\tau_{\mathrm{sc}}>0\) such that every
\(Y\in\mathcal X_\star\) admits
\(G_Y\in\partial\|Y\|_1\) and
\(\Gamma_Y\in\mathbb S^p\) satisfying the KKT relation
\begin{equation}\label{eq:spca_weak_sharp_kkt}
-LY+\gamma G_Y-2Y\Gamma_Y=0,
\end{equation}
and the strict-complementarity condition
\(\|(G_Y)_{S_Y^c}\|_\infty\leq1-\tau_{\mathrm{sc}}\). Assume also the
restricted-injectivity condition
\begin{equation}\label{eq:spca_trivial_active_tangent}
\{H:Y^\top H+H^\top Y=0,\ H_{S_Y^c}=0\}=\{0\},
\end{equation}
for every \(Y\in\mathcal X_\star\). Then there are
\(\rho_{\mathrm{sh}},\sigma_{\mathrm{sh}}>0\) such that
\(f(X)-f(Y_0)\geq
\sigma_{\mathrm{sh}}\dist(X,\mathcal X_\star)\) for every
\(Y_0\in\mathcal X_\star\) and every feasible \(X\) satisfying
\(\dist(X,\mathcal X_\star)<\rho_{\mathrm{sh}}\).
\end{proposition}
The two assumptions have complementary roles. Strict complementarity puts
the selected \(\ell_1\)-subgradient strictly inside \([-1,1]\) at every
zero entry of \(Y\), so creating an off-support entry has a first-order cost.
Restricted injectivity rules out every nonzero feasible tangent direction
supported entirely on \(S_Y\). Therefore every nonzero nearby feasible
displacement has an off-support component to which the strict-complementarity
estimate applies.

The restricted-injectivity condition also requires
\(|S_Y|\leq p(p+1)/2\): the map \(H\mapsto Y^\top H+H^\top Y\)
must be injective from the \(|S_Y|\)-dimensional space of matrices
supported on \(S_Y\) into \(\mathbb S^p\).
Thus the fraction of zero entries must be at least
\(1-(p+1)/(2n)\), although this dimensional condition alone is not sufficient.

\begin{proof}
Choose a nearest \(Y\in\mathcal X_\star\), and set
\(D=X-Y\) and \(S=S_Y\). On \(S\), every entry of \(Y\) is nonzero, so
\((G_Y)_{ij}=\operatorname{sign}(Y_{ij})\). On \(S^c\), one has
\(Y_{ij}=0\) and \(|(G_Y)_{ij}|\leq1-\tau_{\mathrm{sc}}\). Applying the scalar
subgradient inequality on \(S\) and summing over all entries therefore gives
\[
\|X\|_1-\|Y\|_1-\inner{G_Y}{D}
\geq
\sum_{(i,j)\in S^c}
\bigl(|D_{ij}|-(G_Y)_{ij}D_{ij}\bigr)
\geq\tau_{\mathrm{sc}}\|D_{S^c}\|_1.
\]
The quadratic part of the objective satisfies
\[
-\frac12\operatorname{tr}(X^\top LX)
+\frac12\operatorname{tr}(Y^\top LY)
=-\inner{LY}{D}-\frac12\operatorname{tr}(D^\top LD).
\]
Moreover, the KKT relation~\cref{eq:spca_weak_sharp_kkt}, the symmetry of
\(\Gamma_Y\), and the feasibility identity
\[
Y^\top D+D^\top Y+D^\top D=0
\]
imply
\[
-\inner{LY}{D}+\gamma\inner{G_Y}{D}
=2\inner{Y\Gamma_Y}{D}
=\inner{\Gamma_Y}{Y^\top D+D^\top Y}
=-\inner{\Gamma_Y}{D^\top D}.
\]
Combining these three relations yields
\[
f(X)-f(Y)
\geq
\gamma\tau_{\mathrm{sc}}\|D_{S^c}\|_1
-\inner{\Gamma_Y}{D^\top D}
-\frac12\operatorname{tr}(D^\top LD)
\geq
\gamma\tau_{\mathrm{sc}}\|D_{S^c}\|_1
-C_\star\|D\|_{\mathrm{F}}^2,
\]
where
\(C_\star=\max_{Y\in\mathcal X_\star}\|\Gamma_Y\|_2+\|L\|_2/2<\infty\);
the last inequality follows from the standard spectral-norm bounds.
There are constants
\(c_{\mathrm{act}},\rho_0>0\), independent of the choice of
\(Y\in\mathcal X_\star\), such that
\(\|D_{S^c}\|_{\mathrm{F}}\geq c_{\mathrm{act}}\|D\|_{\mathrm{F}}\) whenever
\(0<\|D\|_{\mathrm{F}}<\rho_0\). Otherwise, there would be feasible matrices \(X_j\)
and nearest points \(Y_j\in\mathcal X_\star\) such that
\(D_j=X_j-Y_j\to0\) and
\(\|(D_j)_{S_{Y_j}^c}\|_{\mathrm{F}}/\|D_j\|_{\mathrm{F}}\to0\). Since
\(\mathcal X_\star\) is finite, pass to a subsequence on which \(Y_j=Y\).
A further subsequence of \(H_j=D_j/\|D_j\|_{\mathrm{F}}\) converges to a matrix
\(H\) with \(\|H\|_{\mathrm{F}}=1\). Dividing
\(Y^\top D_j+D_j^\top Y+D_j^\top D_j^{}=0\) by \(\|D_j\|_{\mathrm{F}}\) gives
\[
Y^\top H_j+H_j^\top Y+\|D_j\|_{\mathrm{F}} H_j^\top H_j^{}=0.
\]
Taking limits yields \(Y^\top H+H^\top Y=0\). The off-support ratio gives
\(H_{S_Y^c}=0\), whereas \(\|H\|_{\mathrm{F}}=1\). Thus \(H\) is a nonzero matrix
in the set on the left-hand side of
\cref{eq:spca_trivial_active_tangent}, contradicting that condition. Therefore,
using \(\|D_{S^c}\|_1\geq\|D_{S^c}\|_{\mathrm{F}}\),
\[
f(X)-f(Y)
\geq
(\gamma\tau_{\mathrm{sc}}c_{\mathrm{act}}-C_\star\|D\|_{\mathrm{F}})\|D\|_{\mathrm{F}}.
\]
Choose \(\rho_{\mathrm{sh}}\leq\rho_0\) such that
\(C_\star\rho_{\mathrm{sh}}\leq
\gamma\tau_{\mathrm{sc}}c_{\mathrm{act}}/2\), and set
\(\sigma_{\mathrm{sh}}=\gamma\tau_{\mathrm{sc}}c_{\mathrm{act}}/2\).
Since \(\|D\|_{\mathrm{F}}=\dist(X,\mathcal X_\star)\) and
\(f\) is constant on \(\mathcal X_\star\), the result follows.
\end{proof}

\subsection{Proof of \texorpdfstring{\cref{thm:sublanding_local_linear_convergence}}{Theorem~\ref{thm:sublanding_local_linear_convergence}}}
\label{app:sublanding_local_linear_proof}

For a differentiable map \(G\), \(\mathrm DG(X)\) denotes its Fr\'echet
derivative and \(\mathrm DG(X)[Z]\) its action on \(Z\).

\begin{proof}
For brevity, set
\(e_k=\|E_k\|_{\mathrm{F}}=\|X_k^\top X_k^{}-I_p\|_{\mathrm{F}}\) throughout this proof.
\proofstep{Step 1: comparison of the two distances.}
First, there are \(\rho_{\mathrm{align}}>0\) and
\(c_{\mathrm{align}}\geq1\) such that
\begin{equation}\label{eq:sublanding_orthogonal_distance_equivalence}
\dist(U,\mathcal X)
\leq
\omega(U)
\leq
c_{\mathrm{align}}\dist(U,\mathcal X),
\end{equation}
for every \(U\in\St(n,p)\) satisfying
\(\dist(U,\mathcal X)<\rho_{\mathrm{align}}\).
Fix \(A\in\mathcal X\). If \(p=n\), define
\(\mathcal C_A(U)=U\) for \(U\in\St(n,n)=\mathrm O(n)\). If \(p<n\), choose
an orthogonal completion \([A,A_\perp]\in\mathrm O(n)\). For
\(U\in\St(n,p)\) satisfying \(\|U-A\|_{\mathrm{F}}<1/2\), set
\(\widetilde U_\perp=(I_n-UU^\top)A_\perp\) and
\(U_\perp=\widetilde U_\perp
(\widetilde U_\perp^\top\widetilde U_\perp^{})^{-1/2}\). Since
\(A_\perp^\top A=0\),
\[
\|A_\perp^\top U\|_2
=\|A_\perp^\top(U-A)\|_2
\leq \|U-A\|_{\mathrm{F}}<\frac12.
\]
Moreover,
\[
\widetilde U_\perp^\top\widetilde U_\perp^{}
=I_{n-p}-(A_\perp^\top U)(A_\perp^\top U)^\top,
\quad
\lambda_{\min}(\widetilde U_\perp^\top\widetilde U_\perp^{})
\geq1-\|A_\perp^\top U\|_2^2>\frac34.
\]
Thus the inverse square root is well defined and
\(\mathcal C_A(U)=[U,U_\perp]\in\mathrm O(n)\) is a smooth orthogonal
completion. In either case, \(\mathcal C_A(U)\) is orthogonal and its first
\(p\) columns are \(U\). The defining expression extends smoothly to an
open neighborhood of \(A\) in \(\R^{n\times p}\). This is immediate
when \(p=n\), while for \(p<n\) it follows because the matrix inside the
inverse square root is positive definite at \(U=A\). Values of this
extension outside \(\St(n,p)\) need not be orthogonal. The extension is used
only to apply the Euclidean mean-value inequality. Since its domain is open,
choose \(\delta_A>0\)
such that \(\overline{B}(A,\delta_A)\) lies in this domain, and set
\(r_A=\min\{1/8,\delta_A/2\}\in(0,1/4)\). The derivative
is bounded on that compact ball, so the mean-value inequality yields a
constant \(c_A>0\) such that
\[
\|\mathcal C_A(U)-\mathcal C_A(Z)\|_{\mathrm{F}}
\leq c_A\|U-Z\|_{\mathrm{F}},
\]
for all \(U,Z\in\St(n,p)\) with
\(\|U-A\|_{\mathrm{F}}<2r_A\) and \(\|Z-A\|_{\mathrm{F}}<2r_A\). The sets
\(\{Z\in\mathcal X:\|Z-A\|_{\mathrm{F}}<r_A\}\) cover \(\mathcal X\). By compactness,
choose a finite subcover centered at \(A_1,\ldots,A_{N_{\mathrm{cov}}}\), then choose
\(0<\rho_{\mathrm{align}}<\min_{1\leq i\leq N_{\mathrm{cov}}}r_{A_i}\) and
\(c_{\mathrm{align}}=\max\{1,c_{A_1},\ldots,c_{A_{N_{\mathrm{cov}}}}\}\).
Now let \(U\) satisfy \(\dist(U,\mathcal X)<\rho_{\mathrm{align}}\). Since
\(\mathcal X\) is compact, choose \(Z\in\mathcal X\) such that
\(\|U-Z\|_{\mathrm{F}}=\dist(U,\mathcal X)\), and then choose \(i\) such that
\(\|Z-A_i\|_{\mathrm{F}}<r_{A_i}\). Then
\(\|U-A_i\|_{\mathrm{F}}\leq\|U-Z\|_{\mathrm{F}}+\|Z-A_i\|_{\mathrm{F}}<2r_{A_i}\), so \(U\) and
\(Z\) belong to the same completion chart. Set
\[
\begin{gathered}
Q=\mathcal C_{A_i}(Z)\mathcal C_{A_i}(U)^\top\in\mathrm O(n),
\quad QU=Z,\\
\|Q-I_n\|_{\mathrm{F}}
=\|\mathcal C_{A_i}(Z)-\mathcal C_{A_i}(U)\|_{\mathrm{F}}
\leq c_{\mathrm{align}}\|U-Z\|_{\mathrm{F}}.
\end{gathered}
\]
This proves the upper estimate in
\cref{eq:sublanding_orthogonal_distance_equivalence}. Conversely, every
\(Q\in\mathcal Q(U)\) satisfies
\(\|QU-U\|_{\mathrm{F}}\leq\|Q-I_n\|_{\mathrm{F}}\) because \(\|U\|_2=1\). Since
\(QU\in\mathcal X\), also
\(\dist(U,\mathcal X)\leq\|QU-U\|_{\mathrm{F}}\). Taking the minimum proves the lower
estimate.

\proofstep{Step 2: expansion of the next polar factor.}
Next, there are \(\eta_{\mathrm{pol}},b_{\mathrm{pol}}>0\) such that,
whenever \(X_k\in\St(n,p)^\varepsilon\) and
\(0<\eta_k\leq\eta_{\mathrm{pol}}\), the next polar factor is well defined and
\begin{equation}\label{eq:sublanding_polar_factor_recursion}
\overline X_{k+1}
=
\overline X_k-\eta_kW_k\overline X_k+H_k,
\quad
\|H_k\|_{\mathrm{F}}
\leq
b_{\mathrm{pol}}(\eta_ke_k+\eta_k^2).
\end{equation}
The polar decomposition of \(X_k\) and the normal-only update give
\[
\begin{gathered}
X_k=\overline X_k(I_p+E_k)^{1/2},\\
Y_k^{\mathrm{N}}=X_k+\lambda d_{\mathrm{N}}(X_k)
=X_k(I_p-\lambda E_k)=\overline X_kR_k^{\mathrm{N}},\\
R_k^{\mathrm{N}}=(I_p+E_k)^{1/2}(I_p-\lambda E_k).
\end{gathered}
\]
The two factors defining \(R_k^{\mathrm{N}}\) commute. Since
\(\lambda\varepsilon<1\), all eigenvalues
\(\sqrt{1+s}(1-\lambda s)\), \(|s|\leq\varepsilon\), are positive.
Thus \(R_k^{\mathrm{N}}\) is positive definite and
\begin{equation}\label{app_same_polar_factor}
\mathcal P(Y_k^{\mathrm{N}})=\overline X_k.
\end{equation}
The smallest singular value of \(Y_k^{\mathrm{N}}\) is uniformly bounded below by
\(m_\lambda=\min_{|s|\leq\varepsilon}\sqrt{1+s}(1-\lambda s)>0\).
Choose \(\eta_{\mathrm{pol}}>0\) so that
\(\eta_{\mathrm{pol}}B_{\mathrm{T}}\leq m_\lambda/2\). For every
\(\tau\in[0,1]\), Weyl's inequality gives
\[
\sigma_{\min}(Y_k^{\mathrm{N}}-\tau\eta_kW_kX_k)
\geq m_\lambda-\eta_k\|W_kX_k\|_2\geq m_\lambda/2.
\]
The radial segment
joining \(Y_k^{\mathrm{N}}=\overline X_kR_k^{\mathrm{N}}\) to \(\overline X_k\) is
\[
\overline X_k\bigl[(1-\tau)R_k^{\mathrm{N}}+\tau I_p\bigr],
\quad \tau\in[0,1].
\]
Every matrix on this segment has smallest singular value at least
\(\min\{m_\lambda,1\}>0\), and these segments are uniformly bounded as well.
Consequently, for some \(M_{\mathrm{pol}}<\infty\), both families of segments lie
in the compact set
\[
\mathcal K_{\mathrm{pol}}
=\left\{Y:\|Y\|_{\mathrm{F}}\leq M_{\mathrm{pol}},\ 
\sigma_{\min}(Y)\geq m_{\mathrm{pol}}\right\},
\quad
m_{\mathrm{pol}}=\min\{m_\lambda/2,1\}>0,
\]
which consists of full-column-rank matrices. The polar map is \(C^2\) on the
full-column-rank set, so its second derivative is uniformly bounded on
\(\mathcal K_{\mathrm{pol}}\), and its first derivative is uniformly Lipschitz
along the two families of segments.
At a Stiefel point \(U\), \(\mathrm D\mathcal P(U)[H]=H-U\sym(U^\top H)\). 
Since \(U^\top W_kU\) is skew-symmetric,
\begin{equation}\label{app_polar_tangent_derivative}
\mathrm D\mathcal P(U)[-W_kU]=-W_kU.
\end{equation}
Taylor expansion about \(Y_k^{\mathrm{N}}\), using
\cref{app_same_polar_factor}, gives \(
\overline X_{k+1}
=
\overline X_k
+\eta_k\mathrm D\mathcal P(Y_k^{\mathrm{N}})[-W_kX_k]
+R_k^{\mathrm{pol}},
\|R_k^{\mathrm{pol}}\|_{\mathrm{F}}\leq C_5\eta_k^2.
\) Let \(c_{\mathrm{pol}}=(1+\sqrt{1-\varepsilon})^{-1}\).
By~\cref{eq:sublanding_polar_distance},
\[
\|Y_k^{\mathrm{N}}-\overline X_k\|_{\mathrm{F}}
\leq\|X_k-\overline X_k\|_{\mathrm{F}}+\lambda\|X_kE_k\|_{\mathrm{F}}
\leq\bigl(c_{\mathrm{pol}}+\lambda\sqrt{1+\varepsilon}\bigr)e_k.
\]
Moreover, for some \(C_7>0\), \(
\|-W_kX_k+W_k\overline X_k\|_{\mathrm{F}}\leq C_7e_k
\). The Lipschitz continuity of \(\mathrm D\mathcal P\), together with
\cref{app_polar_tangent_derivative}, therefore yields
\(\|\mathrm D\mathcal P(Y_k^{\mathrm{N}})[-W_kX_k]+W_k\overline X_k\|_{\mathrm{F}}
\leq C_8e_k\) for some \(C_8>0\). Combining these estimates proves
\cref{eq:sublanding_polar_factor_recursion} with
\(b_{\mathrm{pol}}=\max\{C_5,C_8\}\).

\proofstep{Step 3: recursion for the orthogonal distance.}
It remains to establish the following recursion. Choose positive constants
\(\rho_{\mathrm{loc}}\), \(\eta_{\mathrm{loc}}\), \(L_\omega\), \(B_{\omega E}\), and
\(B_{\omega0}\), with
\(\rho_{\mathrm{loc}}<\min\{\rho_{\mathrm{sh}},\rho_{\mathrm{align}}\}\), such that
\begin{equation}\label{eq:sublanding_orthogonal_recursion}
\omega_{k+1}^2
\leq
(1+L_\omega\eta_k)\omega_k^2
-\mu_{\mathrm{sh}}\eta_k\omega_k
+B_{\omega E}\eta_ke_k
+B_{\omega0}\eta_k^2,
\quad
\mu_{\mathrm{sh}}=\frac{4\sigma_{\mathrm{sh}}}{c_{\mathrm{align}}},
\end{equation}
whenever \(e_k\leq\varepsilon\),
\(\omega_k\leq\rho_{\mathrm{loc}}\), and
\(0<\eta_k\leq\eta_{\mathrm{loc}}\).
Choose
\(0<\rho_{\mathrm{loc}}<\min\{\rho_{\mathrm{sh}},\rho_{\mathrm{align}}\}\), and choose a
minimizing orthogonal matrix \(Q_k\) in
\cref{eq:sublanding_orthogonal_distance}, and write \(\Delta_k=Q_k-I_n\) and
\(X_k^{\mathrm{ref}}=Q_k\overline X_k\in\mathcal X\).
Then \(\|\Delta_k\|_{\mathrm{F}}=\omega_k\) and
\(\Delta_k+\Delta_k^\top=-\Delta_k^{}\Delta_k^\top\). With
\(W_k=F_k^{}X_k^\top-X_k^{}F_k^\top\), a trace calculation gives
\begin{equation}\label{app_orthogonal_cross_identity}
\inner{\Delta_k}{Q_kW_k}
=
2\inner{Q_kX_k-X_k}{F_k}
+\inner{\Delta_k^{}\Delta_k^\top}{F_k^{}X_k^\top}.
\end{equation}
The points \(X_k\), \(Q_kX_k\), \(\overline X_k\), and \(X_k^{\mathrm{ref}}\) lie in the
safe region. By the definition of \(\omega_k\) and
\(\|\overline X_k\|_2=1\),
\(\dist(\overline X_k,\mathcal X)\leq\omega_k\leq\rho_{\mathrm{loc}}<\rho_{\mathrm{sh}}\),
so the weak-sharpness inequality applies. Since
\(X_k^{\mathrm{ref}}=Q_k\overline X_k\), use the decomposition
\[
f(Q_kX_k)-f(X_k)
=f(Q_kX_k)-f(Q_k\overline X_k)
+f(X_k^{\mathrm{ref}})-f(\overline X_k)
+f(\overline X_k)-f(X_k).
\]
The Lipschitz continuity of \(f\) and
\cref{eq:sublanding_polar_distance} give
\[
f(Q_kX_k)-f(X_k)
\leq
f(X_k^{\mathrm{ref}})-f(\overline X_k)
+2L_f^\varepsilon c_{\mathrm{pol}}e_k.
\]
Since \(X_k^{\mathrm{ref}}\in\mathcal X\), weak sharpness applied with
\(U=\overline X_k\) and \(Y=X_k^{\mathrm{ref}}\) yields
\[
f(X_k^{\mathrm{ref}})-f(\overline X_k)
\leq
-\sigma_{\mathrm{sh}}\dist(\overline X_k,\mathcal X).
\]
Finally,~\cref{eq:sublanding_orthogonal_distance_equivalence} gives
\[
\dist(\overline X_k,\mathcal X)
\geq\frac{\omega_k}{c_{\mathrm{align}}}.
\]
Consequently,
\[
f(Q_kX_k)-f(X_k)
\leq
-\frac{\sigma_{\mathrm{sh}}}{c_{\mathrm{align}}}\omega_k
+2L_f^\varepsilon c_{\mathrm{pol}}e_k.
\]
Apply the weak-convexity inequality to
\(\inner{Q_kX_k-X_k}{F_k}\) in
\cref{app_orthogonal_cross_identity}. Since \(
\|Q_kX_k-X_k\|_{\mathrm{F}}\leq\sqrt{1+\varepsilon}\,\omega_k
\) and
\(\left|\inner{\Delta_k^{}\Delta_k^\top}{F_k^{}X_k^\top}\right|
\leq L_f^\varepsilon\sqrt{1+\varepsilon}\,\omega_k^2\),
it follows that
\begin{equation}\label{app_sharp_cross_bound}
\inner{\Delta_k}{Q_kW_k}
\leq
-\frac{2\sigma_{\mathrm{sh}}}{c_{\mathrm{align}}}\omega_k
+a_{\mathrm{sh}}\omega_k^2+b_{\mathrm{sh}}e_k.
\end{equation}
Here one may take
\(a_{\mathrm{sh}}=L_h(1+\varepsilon)
+L_f^\varepsilon\sqrt{1+\varepsilon}\) and
\(b_{\mathrm{sh}}=4L_f^\varepsilon c_{\mathrm{pol}}\).
Set \(Q_k^{\mathrm{cmp}}=Q_k\exp(\eta_kW_k)\). 
Let \(H_k\) be the remainder in
\cref{eq:sublanding_polar_factor_recursion}. Since
\[
Q_k^{\mathrm{cmp}}\overline X_{k+1}-X_k^{\mathrm{ref}}
=Q_k\left(
[\exp(\eta_kW_k)(I_n-\eta_kW_k)-I_n]\overline X_k
+\exp(\eta_kW_k)H_k
\right),
\]
\cref{eq:sublanding_polar_factor_recursion,app_exponential_remainder} imply \(
\|Q_k^{\mathrm{cmp}}\overline X_{k+1}-X_k^{\mathrm{ref}}\|_{\mathrm{F}}
\leq
C_9(\eta_ke_k+\eta_k^2)
\) for some \(C_9>0\). Because
\(X_k^{\mathrm{ref}}\in\mathcal X\), the same quantity bounds
\(\dist(Q_k^{\mathrm{cmp}}\overline X_{k+1},\mathcal X)\). Set
\[
\eta_{\mathrm{loc}}
=\min\left\{\eta_{\mathrm{pol}},1,
\frac{\rho_{\mathrm{align}}}{2C_9(1+\varepsilon)}\right\}>0.
\]
Then \(e_k\leq\varepsilon\) and \(\eta_k\leq\eta_{\mathrm{loc}}\) give
\[
C_9(\eta_ke_k+\eta_k^2)
\leq C_9\eta_k(\varepsilon+1)
\leq\frac{\rho_{\mathrm{align}}}2<\rho_{\mathrm{align}}.
\]
By~\cref{eq:sublanding_orthogonal_distance_equivalence}, there is
\(Q_k^{\mathrm{cor}}\in\mathrm O(n)\) such that
\(Q_k^{\mathrm{cor}}Q_k^{\mathrm{cmp}}\overline X_{k+1}\in\mathcal X\) and
\(\|Q_k^{\mathrm{cor}}-I_n\|_{\mathrm{F}}
\leq c_{\mathrm{align}}C_9(\eta_ke_k+\eta_k^2)\).
Hence
\(Q_k^{\mathrm{cor}}Q_k^{\mathrm{cmp}}\in\mathcal Q(\overline X_{k+1})\).
Expanding \(Q_k^{\mathrm{cmp}}\) and using
\cref{app_exponential_remainder} gives
\[
\|Q_k^{\mathrm{cmp}}-I_n\|_{\mathrm{F}}^2
\leq
\omega_k^2
+2\eta_k\inner{\Delta_k}{Q_kW_k}
+C_{10}\eta_k^2.
\]
To account for the additional factor \(Q_k^{\mathrm{cor}}\), set
\(D_k^{\mathrm{cor}}=Q_k^{\mathrm{cor}}-I_n\). Then
\[
Q_k^{\mathrm{cor}}Q_k^{\mathrm{cmp}}-I_n
=(Q_k^{\mathrm{cmp}}-I_n)+D_k^{\mathrm{cor}}Q_k^{\mathrm{cmp}}.
\]
Writing \(A_k=Q_k^{\mathrm{cmp}}-I_n\) and using orthogonality of
\(Q_k^{\mathrm{cmp}}\), one has
\[
\left|\|A_k+D_k^{\mathrm{cor}}Q_k^{\mathrm{cmp}}\|_{\mathrm{F}}^2-\|A_k\|_{\mathrm{F}}^2\right|
\leq2\|A_k\|_{\mathrm{F}}\|D_k^{\mathrm{cor}}\|_{\mathrm{F}}+\|D_k^{\mathrm{cor}}\|_{\mathrm{F}}^2.
\]
The matrices \(A_k\) are uniformly bounded. The preceding estimate of
\(D_k^{\mathrm{cor}}\), together with the local stepsize and error bounds, therefore
bounds the right-hand side by \(C_{11}(\eta_ke_k+\eta_k^2)\). Combining these
estimates with~\cref{app_sharp_cross_bound} proves
\[
\omega_{k+1}^2
\leq
(1+2a_{\mathrm{sh}}\eta_k)\omega_k^2
-\frac{4\sigma_{\mathrm{sh}}}{c_{\mathrm{align}}}\eta_k\omega_k
+(2b_{\mathrm{sh}}+C_{11})\eta_ke_k
+(C_{10}+C_{11})\eta_k^2.
\]
This is~\cref{eq:sublanding_orthogonal_recursion} with
\(L_\omega=2a_{\mathrm{sh}}\),
\(B_{\omega E}=2b_{\mathrm{sh}}+C_{11}\), and
\(B_{\omega0}=C_{10}+C_{11}\).

\proofstep{Step 4: induction and convergence.}
Let \(
q_{\mathrm{N}}=q_\varepsilon(\lambda),
\delta_{\mathrm{N}}=1-q_{\mathrm{N}},
C_E=4B_{\mathrm{T}}^2/\delta_{\mathrm{N}}
\). By~\cref{prop:sublanding_exact_lambda_range}, \(q_{\mathrm{N}}<1\), so
\(\delta_{\mathrm{N}}>0\). Set
\(B_{\mathrm{loc}}=B_{\omega0}+B_{\omega E}C_E\) and
\(\delta_{\mathrm{lin}}=\min\{1/2,\mu_{\mathrm{sh}}^2/(8B_{\mathrm{loc}})\}>0\).
Choose \(\underline\vartheta\in(0,1)\) such that
\(\underline\vartheta^2\geq
\max\{(1+q_{\mathrm{N}})/2,1-\delta_{\mathrm{lin}}/2\}\).
Such a choice is possible because the quantity on the right is strictly
smaller than one. Fix \(\vartheta\in(\underline\vartheta,1)\), and choose
\(\overline\eta>0\) smaller than
\[
\min\left\{
\frac{\delta_{\mathrm{N}}}{4C_{\mathrm{mix}}\lambda},
\frac{\delta_{\mathrm{lin}}}{2L_\omega},
\sqrt{\frac{\varepsilon}{C_E}},
\frac{2(1-\vartheta^2)\rho_{\mathrm{loc}}}{\mu_{\mathrm{sh}}},
\eta_{\mathrm{loc}},
1
\right\}.
\]
Define \(a_{\mathrm{loc}}=1-\vartheta^2+L_\omega\overline\eta\) and
\(C_\omega=\mu_{\mathrm{sh}}/(2a_{\mathrm{loc}})\). The choices above give
\[
1-\vartheta^2<\frac{\delta_{\mathrm{lin}}}2,
\quad
L_\omega\overline\eta<\frac{\delta_{\mathrm{lin}}}2,
\]
and hence
\(0<a_{\mathrm{loc}}<\delta_{\mathrm{lin}}\leq1/2<\vartheta^2\). Moreover,
\(a_{\mathrm{loc}}\geq1-\vartheta^2\), so
\[
C_\omega\overline\eta
\leq\frac{\mu_{\mathrm{sh}}\overline\eta}
{2(1-\vartheta^2)}<\rho_{\mathrm{loc}}.
\]
Consider an initial point satisfying the conditions of the theorem. Since
\(\eta_0\leq\overline\eta\), the preceding choices give
\(\vartheta^2-q_{\mathrm{N}}>\delta_{\mathrm{N}}/2\) and
\(C_{\mathrm{mix}}\lambda\eta_0<\delta_{\mathrm{N}}/4\), and therefore
\(\vartheta^2-q_{\mathrm{N}}-C_{\mathrm{mix}}\lambda\eta_0>\delta_{\mathrm{N}}/4\).
They also give \(C_E\eta_0^2\leq\varepsilon\). The proof now uses
\cref{unified_residual_bound} directly. If
\(e_k\leq C_E\eta_k^2\), then \(\eta_k\leq\eta_0\) and
\[
e_{k+1}
\leq(q_{\mathrm{N}}+C_{\mathrm{mix}}\lambda\eta_0)C_E\eta_k^2
+B_{\mathrm{T}}^2\eta_k^2
=(q_{\mathrm{N}}+C_{\mathrm{mix}}\lambda\eta_0+\delta_{\mathrm{N}}/4)
C_E\eta_k^2
\leq C_E\vartheta^2\eta_k^2=C_E\eta_{k+1}^2.
\]
The initial condition and induction therefore yield
\begin{equation}\label{app:sublanding_local_normal_rate}
e_k\leq C_E\eta_k^2
=C_E\eta_0^2\vartheta^{2k},
\end{equation}
and \(X_k\in\St(n,p)^\varepsilon\) for every \(k\). It remains to prove the
orthogonal-alignment estimate. Whenever
\(\omega_k\leq C_\omega\eta_k\), the bound
\(C_\omega\overline\eta<\rho_{\mathrm{loc}}\) gives
\(\omega_k<\rho_{\mathrm{loc}}\). Thus
\cref{eq:sublanding_orthogonal_recursion} applies. Using
the bound in~\cref{app:sublanding_local_normal_rate} and \(\eta_k\leq1\) gives
\(B_{\omega E}\eta_ke_k\leq B_{\omega E}C_E\eta_k^3
\leq B_{\omega E}C_E\eta_k^2\).
It follows that
\begin{equation}\label{eq:sublanding_reduced_orthogonal_recursion}
\omega_{k+1}^2
\leq
(1+L_\omega\eta_k)\omega_k^2
-\mu_{\mathrm{sh}}\eta_k\omega_k+B_{\mathrm{loc}}\eta_k^2.
\end{equation}
For \(z\in[0,C_\omega]\), set
\(\phi_k(z)=(1+L_\omega\eta_k)z^2-\mu_{\mathrm{sh}}z+B_{\mathrm{loc}}\).
The bounds \(a_{\mathrm{loc}}\leq\delta_{\mathrm{lin}}\leq
\mu_{\mathrm{sh}}^2/(8B_{\mathrm{loc}})\) and
\(a_{\mathrm{loc}}\leq\vartheta^2\) imply
\[
\phi_k(0)=B_{\mathrm{loc}}
\leq\frac{\mu_{\mathrm{sh}}^2}{8a_{\mathrm{loc}}}
\leq\frac{\mu_{\mathrm{sh}}^2\vartheta^2}{4a_{\mathrm{loc}}^2}
=C_\omega^2\vartheta^2.
\]
At the other endpoint,
\(\phi_k(C_\omega)-C_\omega^2\vartheta^2
\leq a_{\mathrm{loc}}C_\omega^2-\mu_{\mathrm{sh}}C_\omega+B_{\mathrm{loc}}\).
Since \(C_\omega=\mu_{\mathrm{sh}}/(2a_{\mathrm{loc}})\), the right-hand side equals
\(B_{\mathrm{loc}}-\mu_{\mathrm{sh}}^2/(4a_{\mathrm{loc}})\leq0\).
Since \(\phi_k\) is convex, its value throughout
\([0,C_\omega]\) is bounded above by the larger endpoint value. The
initial condition gives \(\omega_0/\eta_0\leq C_\omega\). If
\(\omega_k/\eta_k\leq C_\omega\), then
\cref{eq:sublanding_reduced_orthogonal_recursion} yields
\(\omega_{k+1}^2\leq C_\omega^2\vartheta^2\eta_k^2
=C_\omega^2\eta_{k+1}^2\).
Induction therefore proves
\begin{equation}\label{app:sublanding_local_orthogonal_rate}
\omega_k\leq C_\omega\eta_k
=C_\omega\eta_0\vartheta^k.
\end{equation}
The same induction verifies the local-neighborhood condition required by
\cref{eq:sublanding_orthogonal_recursion} at every iteration. Using
\cref{eq:sublanding_polar_distance,app:sublanding_local_normal_rate,app:sublanding_local_orthogonal_rate},
it follows that \(\dist(\overline X_k,\mathcal X)\leq\omega_k\) and
\(\dist(X_k,\mathcal X)\leq c_{\mathrm{pol}}e_k+
\omega_k=\mathcal O(\vartheta^k)\).
Moreover,~\cref{unified_tangential_bounds} and
\cref{eq:local_sublanding_iteration} imply
\(\|X_{k+1}-X_k\|_{\mathrm{F}}\leq B_{\mathrm{T}}\eta_k
+\lambda\sqrt{1+\varepsilon}\,e_k\).
The right-hand side is summable, so \(\{X_k\}\) is a Cauchy sequence. Let
\(X_\infty\) denote its limit. Since
\(\dist(X_k,\mathcal X)\to0\) and \(\mathcal X\) is closed,
\(X_\infty\in\mathcal X\). Summing the preceding increment bound from
\(j=k\) to infinity gives \(\|X_k-X_\infty\|_{\mathrm{F}}=\mathcal O(\vartheta^k)\).

Finally,~\cref{eq:sublanding_weak_sharpness} makes every point of
\(\mathcal X\), including \(X_\infty\), a local minimizer of \(f\) on
\(\St(n,p)\). Since \(X_\infty\in\St(n,p)\),
\(E(X_\infty)=0\) and \(\St_{X_\infty^\top X_\infty^{}}=\St(n,p)\). As \(f\) is locally
Lipschitz, the nonsmooth Fermat rule
\citep[Chapter~10]{rockafellar1998variational} gives
\(0\in\partial_{\St_{X_\infty^\top X_\infty^{}}}^g f(X_\infty)\).
By~\cref{canonical_riemannian_subdifferential}, there is
\(\xi_\infty\in\partial r(X_\infty)\) such that, with
\(F_\infty=\nabla h(X_\infty)+\xi_\infty\),
\[
F_\infty-X_\infty^{} F_\infty^\top X_\infty^{}=0.
\]
Multiplying this identity by \(X_\infty^\top\) shows that
\(F_\infty^\top X_\infty^{}\) is symmetric. Hence
\(\Gamma_\infty=F_\infty^\top X_\infty^{}/2\) satisfies
\(F_\infty=2X_\infty\Gamma_\infty\), which is precisely
\cref{explicit_nonsmooth_KKT}.
\end{proof}
\subsection{Proof of \texorpdfstring{\cref{prop:prox_wellposedness}}{Proposition~\ref{prop:prox_wellposedness}}}
\label{app:prox_wellposedness_proof}

For this and the following proofs, write
\(\mathcal A_X:\R^{n\times p}\to\mathbb S^p\), \(\mathcal A_X(Z)=X^\top Z+Z^\top X\), so that
\(\mathrm T_X\St_{X^\top X}=\ker\mathcal A_X\).

\begin{proof}
The objective in~subproblem~\eqref{prox_subproblem} is coercive and
\(1/t\)-strongly convex on the closed linear space
\(\ker\mathcal A_X\), so it has a unique minimizer. The map
\(\mathcal A_X\) is surjective onto \(\mathbb S^p\): for
\(S\in\mathbb S^p\), the matrix
\(X(X^\top X)^{-1}S/2\) is mapped to \(S\). Convex KKT conditions
are therefore necessary and sufficient. Thus there is at least one pair
\((\Gamma_X,\xi_X)\), with \(\Gamma_X\in\mathbb S^p\) and
\(\xi_X\in\partial r(X+V(X))\), satisfying
\begin{equation}
\label{prox_kkt}
0=\nabla h(X)+\frac{1}{t}V(X)+\xi_X-2X\Gamma_X,
\quad
\mathcal A_X(V(X))=0.
\end{equation}
If
\(\Gamma=\Gamma_X\), the first relation in~\cref{prox_kkt} is the
proximal optimality condition associated with
\(V(X,\Gamma_X)\) in~\cref{prox_multiplier_direction}. Hence,
\(V(X,\Gamma_X)=V(X)\) and
\(R(X,\Gamma_X)=0\).
For arbitrary \(\Gamma\), the proximal optimality condition associated with
\cref{prox_multiplier_direction} is
\(0\in\nabla h(X)+V(X,\Gamma)/t
+\partial r(X+V(X,\Gamma))-2X\Gamma\). If
\(R(X,\Gamma)=0\), this is the KKT system for
 subproblem~\eqref{prox_subproblem}, and uniqueness gives
\(V(X,\Gamma)=V(X)\). The reverse implication follows from
\(\mathcal A_X(V(X))=0\), proving the stated equivalence.
It remains to characterize the zero direction. The KKT system gives
\begin{equation}
\label{prox_zero_equivalence}
V(X)=0
\quad\Longleftrightarrow\quad
\exists\,\xi\in\partial r(X),\ \Gamma\in\mathbb S^p:
\nabla h(X)+\xi-2X\Gamma=0.
\end{equation}
Indeed, the forward implication follows from~\cref{prox_kkt}. Conversely,
if the right-hand side holds, then \(V=0\) satisfies the KKT system, and
uniqueness of the minimizer gives \(V(X)=0\). Combining
\cref{prox_zero_equivalence} with feasibility proves the final equivalence.
\end{proof}

\subsection{Proof of \texorpdfstring{\cref{thm:prox_complexity}}{Theorem~\ref{thm:prox_complexity}}}
\label{app:proof-prox-complexity}

\begin{proof}
We first show that sufficiently small steps satisfy both line-search
conditions uniformly on the safe region. Throughout this proof, write
\(V_k=V(X_k)\). Set
\[
B_{\mathrm{P}}
=\left[t^2(L_f^\varepsilon)^2
+\lambda^2(1+\varepsilon)\varepsilon^2\right]^{1/2},
\quad
C_{\mathrm{P}}=\frac{L_h}{2}+\sigma_{\mathrm P},
\quad
b=2\lambda\sigma_{\mathrm P}(1-\varepsilon)
-\lambda L_f^\varepsilon\sqrt{1+\varepsilon},
\]
and define
\[
\eta_*
=\min\left\{
\frac{1}{2\lambda(1+\varepsilon)},
\frac{2\lambda\varepsilon(1-\varepsilon)}{B_{\mathrm{P}}^2},
1,
\frac{1-\rho}{tC_{\mathrm{P}}},
\frac{b-\rho}{C_{\mathrm{P}}\lambda^2(1+\varepsilon)\varepsilon}
\right\}.
\]
By~\cref{prox_sigma_condition},
\(b-\rho=(\vartheta_{\mathrm{P}}-1)
\bigl(\rho+\lambda L_f^\varepsilon\sqrt{1+\varepsilon}\bigr)>0\).
Hence every
entry in the minimum defining \(\eta_*\) is positive, so \(\eta_*>0\).
Let \(m_X(U)=\inner{\nabla h(X)}{U}+\|U\|_{\mathrm{F}}^2/(2t)+r(X+U)\).
Restricted to the linear space \(\ker\mathcal A_X\), this function is
\(1/t\)-strongly convex. Since \(V(X)\) minimizes \(m_X\) on this space
and \(0\in\ker\mathcal A_X\),
\(m_X(0)\geq m_X(V(X))+\|V(X)\|_{\mathrm{F}}^2/(2t)\).
For \(\xi\in\partial r(X)\), convexity of \(r\) then yields
\(\|V(X)\|_{\mathrm{F}}^2/t
\leq-\inner{\nabla h(X)+\xi}{V(X)}
\leq L_f^\varepsilon\|V(X)\|_{\mathrm{F}}\),
which proves
\begin{equation}
\label{prox_displacement_bound}
\|V(X)\|_{\mathrm{F}}\leq tL_f^\varepsilon.
\end{equation}
Since
\(V(X)\in\ker\mathcal A_X\),
\(\inner{V(X)}{d_{\mathrm{N}}(X)}=-\inner{V(X)}{XE(X)}=0\).
Together with~\cref{prox_displacement_bound} and the safe-region bounds,
orthogonality gives
\[
\|\Lambda(X)\|_{\mathrm{F}}^2
=\|V(X)\|_{\mathrm{F}}^2
+\lambda^2\|d_{\mathrm{N}}(X)\|_{\mathrm{F}}^2
\leq t^2(L_f^\varepsilon)^2
+\lambda^2(1+\varepsilon)\varepsilon^2
=B_{\mathrm{P}}^2.
\]
Write \(V=V(X)\), \(E=E(X)\), \(e=\|E\|_{\mathrm{F}}\), and
\(\Lambda=\Lambda(X)\). Direct
expansion gives \(E(X+\eta\Lambda)=E-2\eta\lambda(E+E^2)
+\eta^2\Lambda^\top\Lambda\).
When \(0<\eta\leq[2\lambda(1+\varepsilon)]^{-1}\), spectral calculus
and \(\|E\|_2\leq\varepsilon\) imply
\begin{equation}
\label{app:prox_linear_feasibility}
\|E(X+\eta\Lambda)\|_{\mathrm{F}}
\leq[1-2\eta\lambda(1-\varepsilon)]e
+\eta^2\|\Lambda\|_{\mathrm{F}}^2.
\end{equation}
The coefficient of \(e\) is nonnegative. Thus, the first two bounds in the
definition of \(\eta_*\) give
\(\|E(X+\eta\Lambda)\|_{\mathrm{F}}
\leq[1-2\eta\lambda(1-\varepsilon)]\varepsilon+
\eta^2B_{\mathrm{P}}^2\leq\varepsilon\). 
For \(0\leq\eta\leq1\), convexity of \(r\) gives
\(r(X+\eta V)-r(X)\leq\eta[r(X+V)-r(X)]\). Combining this inequality with
the preceding optimality comparison gives
\begin{equation}
\label{app:prox_model_decrease}
\eta\inner{\nabla h(X)}{V}
+r(X+\eta V)-r(X)
\leq-\frac{\eta}{t}\|V\|_{\mathrm{F}}^2.
\end{equation}
The descent lemma for \(h\), the global \(L_r\)-Lipschitz continuity of
\(r\), and~\cref{app:prox_model_decrease} yield the first estimate below.
Adding \(\sigma_{\mathrm P}\) times~\cref{app:prox_linear_feasibility} gives the second,
\[
\begin{aligned}
f(X+\eta\Lambda)-f(X)
&\leq-\frac{\eta}{t}\|V\|_{\mathrm{F}}^2
+\eta\lambda L_f^\varepsilon\sqrt{1+\varepsilon}\,e
+\frac{L_h}{2}\eta^2\|\Lambda\|_{\mathrm{F}}^2,\\
\Phi_{\sigma_{\mathrm P}}(X+\eta\Lambda)-\Phi_{\sigma_{\mathrm P}}(X)
&\leq-\frac{\eta}{t}\|V\|_{\mathrm{F}}^2-\eta b e
+C_{\mathrm{P}}\eta^2\|\Lambda\|_{\mathrm{F}}^2.
\end{aligned}
\]
Furthermore,
\(\|\Lambda\|_{\mathrm{F}}^2\leq\|V\|_{\mathrm{F}}^2
+\lambda^2(1+\varepsilon)\varepsilon e\).
The last two bounds in the definition of \(\eta_*\) now give
\begin{equation}
\label{prox_uniform_decrease}
X^+=X+\eta\Lambda(X)\in\St(n,p)^\varepsilon,
\quad
\Phi_{\sigma_{\mathrm P}}(X^+)
\leq\Phi_{\sigma_{\mathrm P}}(X)-\eta
\rho\left(\frac{\|V(X)\|_{\mathrm{F}}^2}{t}
+\|E(X)\|_{\mathrm{F}}\right).
\end{equation}
Every trial not exceeding \(\eta_*\) therefore passes both line-search
tests. Geometric backtracking terminates, and its first accepted step satisfies
\begin{equation}
\label{prox_accepted_lower_bound}
\eta_k\geq
\underline\eta=
\beta_{\mathrm{ls}}\min\{\eta_0,\eta_*\}>0.
\end{equation}

\Needspace{5\baselineskip}
\proofstep{Continuity.}
For \(X\in\R_*^{n\times p}\), set \(G=X^\top X\). Given
\(W\in\R^{n\times p}\), let \(S_X(W)\in\mathbb S^p\) be the unique
solution of \(GS_X(W)+S_X(W)G=X^\top W+W^\top X\), and set
\(\operatorname{Proj}_{X,\mathrm E}(W)=W-XS_X(W)\).
This is the Euclidean projection onto \(\ker\mathcal A_X\). On the
safe region, \(G\succeq(1-\varepsilon)I_p\), so the projection is
continuous in \((X,W)\).
Let \(X_j\to X\) in the safe region and set \(V_j=V(X_j)\). The bound in
\cref{prox_displacement_bound} implies that \(\{V_j\}\) is bounded. Consider
an arbitrary convergent subsequence and, without relabeling, write
\(V_j\to\overline V\).
Since \(\mathcal A_{X_j}(V_j)=0\) and
\((X,V)\mapsto\mathcal A_X(V)\) is continuous,
\(\mathcal A_X(\overline V)=0\), so
\(\overline V\in\ker\mathcal A_X\).
For arbitrary \(W\in\ker\mathcal A_X\), let
\(W_j=\operatorname{Proj}_{X_j,\mathrm E}(W)\). Then \(W_j\to W\).
Exact optimality gives
\[
\inner{\nabla h(X_j)}{V_j}
+\frac{1}{2t}\|V_j\|_{\mathrm{F}}^2+r(X_j+V_j)
\leq
\inner{\nabla h(X_j)}{W_j}
+\frac{1}{2t}\|W_j\|_{\mathrm{F}}^2+r(X_j+W_j).
\]
Passing to the limit shows that \(\overline V\) solves
 subproblem~\eqref{prox_subproblem} at \(X\). Uniqueness gives
\(\overline V=V(X)\). Every convergent subsequence has this limit, so
\(V(X_j)\to V(X)\).

To prove the joint continuity of \(V(X,\Gamma)\), let
\((X_j,\Gamma_j)\to(X,\Gamma)\) and set
\(Z_j=X_j-t\nabla h(X_j)+2tX_j\Gamma_j\). Then
\(Z_j\to Z=X-t\nabla h(X)+2tX\Gamma\). Nonexpansiveness of the fixed
proximal mapping \(\prox_{t r}\) gives
\(\prox_{t r}(Z_j)\to\prox_{t r}(Z)\), and hence
\(V(X_j,\Gamma_j)\to V(X,\Gamma)\).

Since \(X_0\in\St(n,p)^\varepsilon\), induction using
\cref{prox_uniform_decrease,prox_accepted_lower_bound} shows that every line
search terminates after finitely many backtracking steps,
\(X_k\in\St(n,p)^\varepsilon\) for all \(k\), and
\(\eta_k\geq\underline\eta\). Acceptance and telescoping give
\begin{equation}
\label{prox_global_sum}
\sum_{k=0}^{K-1}\eta_k\rho
\left(\frac{\|V_k\|_{\mathrm{F}}^2}{t}+\|E(X_k)\|_{\mathrm{F}}\right)
\leq\Delta_{\mathrm{P},0},
\end{equation}
because
\(\Phi_{\sigma_{\mathrm P}}(X_K)\geq f_*\). The lower bound
\cref{prox_accepted_lower_bound}, the minimum--average inequality, and
nonnegativity give~\cref{prox_separate_complexity_bounds}. Since
\(\eta_k\geq\underline\eta>0\), \cref{prox_global_sum} also implies
\[
\sum_{k=0}^\infty
\left(\frac{\|V_k\|_{\mathrm{F}}^2}{t}+\|E(X_k)\|_{\mathrm{F}}\right)<\infty.
\]
Therefore \(\|V_k\|_{\mathrm{F}}\to0\) and \(\|E(X_k)\|_{\mathrm{F}}\to0\). The safe region
is compact. If \(X_{k_j}\to\overline X\), the continuity established above gives
\(E(\overline X)=0\) and \(V(\overline X)=0\). By
\cref{prop:prox_wellposedness}, \(\overline X\) is first-order stationary.
\end{proof}

\subsection{Proof of \texorpdfstring{\cref{prop:iprox_residual_exactness}}{Proposition~\ref{prop:iprox_residual_exactness}}}
\label{app:iprox_residual_exactness_proof}

\begin{proof}
Suppose first that \(V(X,\Gamma)=0\) and \(E(X)=0\) for some
\(\Gamma\in\mathbb S^p\). The optimality condition defining the proximal
mapping gives \(\xi\in\partial r(X)\) such that
\(\nabla h(X)+\xi-2X\Gamma=0\), which is the first-order stationarity
system. Conversely, if \(X\) is first-order stationary, there are
\(\Gamma\in\mathbb S^p\) and \(\xi\in\partial r(X)\) satisfying this
system. Then \(V=0\) satisfies the optimality condition defining
\(V(X,\Gamma)\). Strong convexity of the proximal subproblem gives
\(V(X,\Gamma)=0\), completing the proof.
\end{proof}

\subsection{Tangency Corrections for IPGLanding}
\label{app:iprox_tangency_correction}

Fix \(X\in\St(n,p)^\varepsilon\) and \(\Gamma\in\mathbb S^p\), and abbreviate
\(V=V(X,\Gamma)\) and \(R=R(X,\Gamma)\). For \(R\ne0\), \(V\) need not be
tangent. Its Euclidean projection \(D_{\mathrm{E}}=V-XS_{\mathrm{E}}\) requires
\(S_{\mathrm{E}}\in\mathbb S^p\) solving the Sylvester equation
\[
(X^\top X)S_{\mathrm{E}}+S_{\mathrm{E}}(X^\top X)=R.
\]
If only tangency, rather than Euclidean orthogonality of the projection, is
required, the oblique projection
\[
D_{\mathrm{obl}}=V-\frac12X(X^\top X)^{-1}R
\]
maps \(V\) exactly into \(\mathrm T_X\St_{X^\top X}\) because
\(\mathcal A_X(X(X^\top X)^{-1}R/2)=R\). It avoids the Sylvester equation
but still requires a linear solve involving \(X^\top X\). 
IPGLanding instead replaces \((X^\top X)^{-1}\) by \(I_p\) and uses
\(D=V-XR/2=V-X\sym(X^\top V)\). When \(X^\top X=I_p\), all three
corrections coincide with the Euclidean tangent projection. At infeasible
points, the explicit correction controls the remaining  violation of the tangency condition
rather than enforcing exact tangency. Writing \(E=E(X)\), we have
\[
\mathcal A_X(D)
=R-\frac12\bigl((I_p+E)R+R(I_p+E)\bigr)
=-\frac12(ER+RE).
\]
Consequently, \(D\) is tangent whenever \(E=0\) or \(R=0\), and
\[
\|\mathcal A_X(D)\|_{\mathrm{F}}\leq\|E\|_{\mathrm{F}}\|R\|_{\mathrm{F}}.
\]

\subsection{Proof of \texorpdfstring{\cref{thm:iprox_complexity}}{Theorem~\ref{thm:iprox_complexity}}}
\label{app:iprox_complexity_proof}

\begin{proof}
For each completed outer iteration, abbreviate
\(V_k=V(X_k,\Gamma_k)=V_{k,J_k}\).
\proofstep{Reference-value invariant.}
Whenever an outer iteration \(k\) is completed, the acceptance test gives
\(\Phi_{\sigma_{\mathrm{IP}}}(X_{k+1})\leq\overline\Phi_k\). Since
\(a_{k+1}^{\mathrm{nm}}=\omega_{\mathrm{mem}}a_k^{\mathrm{nm}}+1\),
\cref{iprox_nonmonotone_reference} is a convex combination of
\(\overline\Phi_k\) and \(\Phi_{\sigma_{\mathrm{IP}}}(X_{k+1})\). Consequently,
\[
\Phi_{\sigma_{\mathrm{IP}}}(X_{k+1})
\leq\overline\Phi_{k+1}\leq\overline\Phi_k.
\]
The initialization \(\overline\Phi_0=\Phi_{\sigma_{\mathrm{IP}}}(X_0)\) thus implies
\(\Phi_{\sigma_{\mathrm{IP}}}(X_k)\leq\overline\Phi_k\) for every generated outer
iterate.

\proofstep{Multiplier iteration.}
Set
\[
c_{\mathrm{d}}=\frac{\theta_{\mathrm{d}}(2-\theta_{\mathrm{d}})}
{16t^2(1+\varepsilon)^2}>0.
\]
Fix \(X\in\St(n,p)^\varepsilon\), set
\(T_X(\Gamma)=\Gamma-\tau_{\mathrm{d}}R(X,\Gamma)\), and let
\(V^\star=V(X)\). By
\cref{prop:prox_wellposedness}, there are
\(\Gamma_X^\star\in\mathbb S^p\) and
\(\xi^\star\in\partial r(X+V^\star)\) such that
\(0=\nabla h(X)+V^\star/t+\xi^\star-2X\Gamma_X^\star\) and
\(\mathcal A_X(V^\star)=0\).
The first equality is also the optimality condition defining
\(V(X,\Gamma_X^\star)\). Strong convexity therefore gives
\(V(X,\Gamma_X^\star)=V^\star\), and hence
\(R(X,\Gamma_X^\star)=0\). The global \(L_r\)-Lipschitz continuity of \(r\)
gives \(\|\xi^\star\|_{\mathrm{F}}\leq L_r\). Together with the definition of
\(L_f^\varepsilon\), the argument of~\cref{prox_displacement_bound} gives
\(\|V(X)\|_{\mathrm{F}}\leq tL_f^\varepsilon\), and hence
\(2\|X\Gamma_X^\star\|_{\mathrm{F}}\leq\|\nabla h(X)\|_{\mathrm{F}}+L_r
+\|V^\star\|_{\mathrm{F}}/t\leq2L_f^\varepsilon\). Since
\(X^\top X\succeq(1-\varepsilon)I_p\), it also follows that
\(\sqrt{1-\varepsilon}\,\|\Gamma_X^\star\|_{\mathrm{F}}
\leq\|X\Gamma_X^\star\|_{\mathrm{F}}\leq L_f^\varepsilon\).
Thus \(\|\Gamma_X^\star\|_{\mathrm{F}}\leq R_{\mathrm{a}}\). For
\(\Gamma_1,\Gamma_2\in\mathbb S^p\), write
\(\Delta\Gamma=\Gamma_1-\Gamma_2\),
\(\Delta V=V(X,\Gamma_1)-V(X,\Gamma_2)\), and
\(\Delta R=R(X,\Gamma_1)-R(X,\Gamma_2)\). Firm nonexpansiveness of the
proximal mapping~\citep{bauschke2017convex} and symmetry of
\(\Delta\Gamma\) give
\[
\begin{aligned}
t\inner{\Delta R}{\Delta\Gamma}
&=t\inner{2\sym(X^\top\Delta V)}{\Delta\Gamma}
=\inner{\Delta V}{2tX\Delta\Gamma}\geq\|\Delta V\|_{\mathrm{F}}^2,\\
\inner{\Delta R}{\Delta\Gamma}
&\geq \frac{\|\Delta V\|_{\mathrm{F}}^2}{t},
\quad \|\Delta R\|_{\mathrm{F}}\leq2\|X\|_2\|\Delta V\|_{\mathrm{F}}.
\end{aligned}
\]
Consequently, \(\inner{\Delta R}{\Delta\Gamma}
\geq\|\Delta R\|_{\mathrm{F}}^2/(4t\|X\|_2^2)\). Expanding the squared distance
and using \(\|X\|_2^2\leq1+\varepsilon\) together with
\cref{iprox_safe_dual_step} gives
\[
\begin{aligned}
\|T_X(\Gamma_1)-T_X(\Gamma_2)\|_{\mathrm{F}}^2
&=\|\Delta\Gamma-\tau_{\mathrm{d}}\Delta R\|_{\mathrm{F}}^2\\
&\leq\|\Delta\Gamma\|_{\mathrm{F}}^2
-\tau_{\mathrm{d}}\left([2t\|X\|_2^2]^{-1}-\tau_{\mathrm{d}}\right)
\|\Delta R\|_{\mathrm{F}}^2\\
&\leq\|\Gamma_1-\Gamma_2\|_{\mathrm{F}}^2-c_{\mathrm{d}}
\|R(X,\Gamma_1)-R(X,\Gamma_2)\|_{\mathrm{F}}^2.
\end{aligned}
\]
Here the last inequality follows from
\(\tau_{\mathrm{d}}([2t(1+\varepsilon)]^{-1}-\tau_{\mathrm{d}})
=c_{\mathrm{d}}\). Therefore, \(T_X\) is
nonexpansive. Since \(R(X,\Gamma_X^\star)=0\), the point
\(\Gamma_X^\star\) is a fixed point of \(T_X\). Moreover, \(\Gamma_0,\Gamma_X^\star\in\mathbb B_{\mathrm{a}}\).
Since \(\Gamma_{j+1}=T_X(\Gamma_j)\) and \(T_X\) is nonexpansive,
\(\|\Gamma_j-\Gamma_X^\star\|_{\mathrm{F}}
\leq\|\Gamma_0-\Gamma_X^\star\|_{\mathrm{F}}\leq2R_{\mathrm{a}}\)
for every \(j\geq0\), and hence \(\|\Gamma_j\|_{\mathrm{F}}\leq3R_{\mathrm{a}}\).
The stronger estimate above also yields
\[
\|\Gamma_{j+1}-\Gamma_X^\star\|_{\mathrm{F}}^2
\leq\|\Gamma_j-\Gamma_X^\star\|_{\mathrm{F}}^2
-c_{\mathrm{d}}\|R(X,\Gamma_j)\|_{\mathrm{F}}^2.
\]
Summing this distance-decrease inequality gives
\[
c_{\mathrm{d}}\sum_{j=0}^{\infty}\|R(X,\Gamma_j)\|_{\mathrm{F}}^2
\leq\|\Gamma_0-\Gamma_X^\star\|_{\mathrm{F}}^2\leq4R_{\mathrm{a}}^2,
\]
and hence \(R(X,\Gamma_j)\to0\).
\proofstep{Uniform primal estimates.}
Set
\[
\begin{alignedat}{2}
B_V&=t\bigl(L_f^\varepsilon
+6\sqrt{1+\varepsilon}\,R_{\mathrm{a}}\bigr),\quad&
B_\Lambda&=(2+\varepsilon)B_V
+\lambda\sqrt{1+\varepsilon}\,\varepsilon,\\
c_\Lambda&=\frac{2}{\rho}\max\left\{
t(2+\varepsilon)^2,\lambda^2(1+\varepsilon)\varepsilon\right\},\quad&
C_{\mathrm{IP}}&=\left(\frac{L_h}{2}+\sigma_{\mathrm{IP}}\right)c_\Lambda,
\end{alignedat}
\]
and define
\[
\eta_*^{\mathrm{IP}}
=\min\left\{1,
\frac{1}{2\lambda(1+\varepsilon)},
\frac{\lambda(1-\varepsilon)\varepsilon}{B_\Lambda^2},
\frac{1-c}{C_{\mathrm{IP}}}\right\}.
\]
Set
\(\underline\eta_{\mathrm{IP}}
=\beta_{\mathrm{ls}}\min\{\eta_0,\eta_*^{\mathrm{IP}}\}>0\).
All these constants are positive.
The argument proceeds by induction over the accepted outer iterations. Initially,
\(X_0\in\St(n,p)^\varepsilon\) and
\(\Gamma_{0,0}\in\mathbb B_{\mathrm{a}}\). Every accepted primal trial remains in
the safe region by the acceptance test, while the projected warm start
returns the next initial multiplier to \(\mathbb B_{\mathrm{a}}\). The
multiplier-iteration argument above therefore gives
\(\|\Gamma_{k,j}\|_{\mathrm{F}}\leq3R_{\mathrm{a}}\) for all inner iterates.
Fix indices \((k,j)\) of an inner iterate, set
\((X,\Gamma)=(X_k,\Gamma_{k,j})\), and abbreviate
\(E=E(X)\), \(e=\|E\|_{\mathrm{F}}\), \(V=V_{k,j}\), \(R=R_{k,j}\),
\(D=D(X,\Gamma)\), \(\Lambda=\Lambda_{k,j}\), and \(M=M_{k,j}\). Choose
\(\xi\in\partial r(X)\). Since
\(X=\prox_{t r}(X+t\xi)\), nonexpansiveness of the
proximal mapping and~\cref{safe_region_objective_bounds} give
\(\|V\|_{\mathrm{F}}\leq t(L_f^\varepsilon
+6\sqrt{1+\varepsilon}\,R_{\mathrm{a}})\leq B_V\). Moreover,
\(\|D\|_{\mathrm{F}}\leq(2+\varepsilon)\|V\|_{\mathrm{F}}\) and
\(\|XE\|_{\mathrm{F}}\leq\sqrt{1+\varepsilon}\,e\).
Thus \(\|\Lambda\|_{\mathrm{F}}^2\leq2(2+\varepsilon)^2\|V\|_{\mathrm{F}}^2
+2\lambda^2(1+\varepsilon)e^2\).
Using \(e^2\leq\varepsilon e\) and the triangle inequality gives
\(\|\Lambda\|_{\mathrm{F}}\leq B_\Lambda\) and
\begin{equation}
\label{iprox_direction_bounds}
\|\Lambda(X,\Gamma)\|_{\mathrm{F}}^2\leq c_\Lambda M(X,\Gamma).
\end{equation}
When
\(0<\eta\leq[2\lambda(1+\varepsilon)]^{-1}\), using
\(\mathcal A_X(D)=-(ER+RE)/2\) to expand \(E(X+\eta\Lambda)\) and diagonalizing
\(E\) yield
\[
\begin{aligned}
E(X+\eta\Lambda)
&=E-2\eta\lambda(E+E^2)-\eta(ER+RE)/2+\eta^2\Lambda^\top\Lambda,\\
\|E-2\eta\lambda(E+E^2)\|_{\mathrm{F}}
&\leq[1-2\eta\lambda(1-\varepsilon)]e.
\end{aligned}
\]
Moreover, \(\|ER+RE\|_{\mathrm{F}}/2\leq e\|R\|_{\mathrm{F}}\) and
\(\|\Lambda^\top\Lambda\|_{\mathrm{F}}\leq\|\Lambda\|_{\mathrm{F}}^2\).
Combining these estimates gives
\begin{equation}
\label{iprox_feasibility_upper_model}
\|E(X+\eta\Lambda(X,\Gamma))\|_{\mathrm{F}}
\leq \|E(X)\|_{\mathrm{F}}
+\eta\|E(X)\|_{\mathrm{F}}
[-2\lambda(1-\varepsilon)+\|R(X,\Gamma)\|_{\mathrm{F}}]
+\eta^2\|\Lambda(X,\Gamma)\|_{\mathrm{F}}^2.
\end{equation}
For \(0<\eta\leq1\), the descent lemma for \(h\) and convexity of
\(r\) give
\[
f(X+\eta\Lambda)-f(X)
\leq\eta\inner{\nabla h(X)}{\Lambda}
+\eta[r(X+\Lambda)-r(X)]
+\frac{L_h}{2}\eta^2\|\Lambda\|_{\mathrm{F}}^2.
\]
Adding \(\sigma_{\mathrm{IP}}\) times
\cref{iprox_feasibility_upper_model} to this objective bound and then using
\cref{iprox_direction_bounds} gives
\(\Phi_{\sigma_{\mathrm{IP}}}(X+\eta\Lambda)-\Phi_{\sigma_{\mathrm{IP}}}(X)
\leq\eta\ell(X,\Gamma)+C_{\mathrm{IP}}\eta^2M\).
The sufficient-descent condition gives \(\ell(X,\Gamma)\leq-M\). If
\(\eta\leq(1-c)/C_{\mathrm{IP}}\), then
\(\Phi_{\sigma_{\mathrm{IP}}}(X+\eta\Lambda)\leq\Phi_{\sigma_{\mathrm{IP}}}(X)-c\eta M
\leq\overline\Phi_k-c\eta M\), which is the merit test in
\cref{iprox_acceptance_test}. Under the tangency-residual condition
\(\|R\|_{\mathrm{F}}\leq\lambda(1-\varepsilon)\) in
\cref{iprox_primal_conditions},~\cref{iprox_feasibility_upper_model} reduces
to the following bound. Since
\(0\leq1-\eta\lambda(1-\varepsilon)\) and
\(\eta B_\Lambda^2\leq\lambda(1-\varepsilon)\varepsilon\),
\[
\|E(X+\eta\Lambda)\|_{\mathrm{F}}
\leq[1-\eta\lambda(1-\varepsilon)]e+\eta^2B_\Lambda^2
\leq[1-\eta\lambda(1-\varepsilon)]\varepsilon
+\eta\lambda(1-\varepsilon)\varepsilon
=\varepsilon.
\]
Thus every trial with
\(0<\eta\leq\eta_*^{\mathrm{IP}}\) is accepted, and geometric backtracking
returns \(\eta_k\geq\underline\eta_{\mathrm{IP}}\). If the initial trial with
\(\eta=\eta_0\) is accepted before
\cref{iprox_primal_conditions} is checked, then
\(\eta_k=\eta_0\geq\underline\eta_{\mathrm{IP}}\) directly from the definition
of \(\underline\eta_{\mathrm{IP}}\).

\proofstep{Sufficient descent.}
We next show that reducing the tangency residual eventually ensures
the conditions in~\cref{iprox_primal_conditions} at any nonstationary
iterate. Let
\[
a_{\mathrm{IP}}=\min\left\{1-\rho,
(\vartheta_{\mathrm{IP}}-1)
\bigl(\rho+\lambda L_f^\varepsilon\sqrt{1+\varepsilon}\bigr)\right\}>0,
\quad
b_{\mathrm{I}}=3R_{\mathrm{a}}
+\frac12L_f^\varepsilon\sqrt{1+\varepsilon}
+\sigma_{\mathrm{IP}}\varepsilon.
\]
Fix indices \((k,j)\) of an arbitrary inner iterate, set
\((X,\Gamma)=(X_k,\Gamma_{k,j})\), and abbreviate
\(E=E(X)\), \(e=\|E\|_{\mathrm{F}}\), \(V=V_{k,j}\), \(R=R_{k,j}\),
\(D=D(X,\Gamma)\), \(\Lambda=\Lambda_{k,j}\), \(M=M_{k,j}\), and
\(\ell=\ell_{k,j}\). By~\cref{iprox_penalty_condition},
\(a_{\mathrm{IP}}>0\), and the definition of \(b_{\mathrm{I}}\) gives
\(b_{\mathrm{I}}>0\).
The optimality condition of the proximal mapping gives a subgradient
\(\xi\in\partial r(X+V)\). It satisfies
\(0=\nabla h(X)-2X\Gamma+V/t+\xi\).
Convexity of \(r\) yields
\(\inner{\nabla h(X)-2X\Gamma}{V}+r(X+V)-r(X)
\leq-\|V\|_{\mathrm{F}}^2/t\).
Since \(D-V=-XR/2\), the gradient bound at \(X\) in the safe region and
the global Lipschitz continuity of \(r\) give
\(\inner{\nabla h(X)}{D-V}+r(X+D)-r(X+V)
\leq L_f^\varepsilon\sqrt{1+\varepsilon}\|R\|_{\mathrm{F}}/2\).
Also,
\(2\inner{X\Gamma}{V}=\inner{\Gamma}{R}\). Therefore,
\[
\inner{\nabla h(X)}{D}+r(X+D)-r(X)
\leq-\frac{1}{t}\|V\|_{\mathrm{F}}^2
+\left(3R_{\mathrm{a}}
+\frac12L_f^\varepsilon\sqrt{1+\varepsilon}\right)\|R\|_{\mathrm{F}}.
\]
Since \(\Lambda-D=\lambda d_{\mathrm{N}}(X)=-\lambda XE\), the gradient bound and the global
\(L_r\)-Lipschitz continuity of \(r\) give
\[
\inner{\nabla h(X)}{\Lambda-D}
+r(X+\Lambda)-r(X+D)
\leq\lambda\bigl(\|\nabla h(X)\|_{\mathrm{F}}+L_r\bigr)\|XE\|_{\mathrm{F}}
\leq\lambda L_f^\varepsilon\sqrt{1+\varepsilon}\,e.
\]
Combining the two direction estimates with the definition of \(\ell\),
and using \(e\leq\varepsilon\) in its \(e\|R\|_{\mathrm{F}}\) term, gives
\[
\ell(X,\Gamma)
\leq-\frac{1}{t}\|V\|_{\mathrm{F}}^2
-[2\sigma_{\mathrm{IP}}\lambda(1-\varepsilon)
-\lambda L_f^\varepsilon\sqrt{1+\varepsilon}]e
+b_{\mathrm{I}}\|R\|_{\mathrm{F}}.
\]
Rearranging this bound for \(\ell\) gives the following lower bound on
\(-M-\ell\):
\[
-M-\ell\geq(1-\rho)\frac{\|V\|_{\mathrm{F}}^2}{t}
+[2\sigma_{\mathrm{IP}}\lambda(1-\varepsilon)-\lambda L_f^\varepsilon\sqrt{1+\varepsilon}-\rho]e
-b_{\mathrm{I}}\|R\|_{\mathrm{F}}.
\]
By~\cref{iprox_penalty_condition}, the coefficient of \(e\) equals
\((\vartheta_{\mathrm{IP}}-1)
(\rho+\lambda L_f^\varepsilon\sqrt{1+\varepsilon})\). Hence the definition
of \(a_{\mathrm{IP}}\) proves
\begin{equation}
\label{iprox_transfer_lower_bound}
-M(X,\Gamma)-\ell(X,\Gamma)
\geq\frac{a_{\mathrm{IP}}}{\rho}M(X,\Gamma)
-b_{\mathrm{I}}\|R(X,\Gamma)\|_{\mathrm{F}}.
\end{equation}

\proofstep{Infinite inner loops.}
Suppose that no outer step is accepted at \(X_k\), even after
infinitely many multiplier updates. The multiplier estimates
above give \(R_{k,j}\to0\) and show that
\(\{\Gamma_{k,j}\}\) is bounded. Every accumulation point
\(\overline\Gamma\) of this sequence therefore satisfies
\(R(X_k,\overline\Gamma)=0\). Continuity of the proximal
direction and \cref{prop:prox_wellposedness} then give
\(V_{k,j}\to V(X_k)\).

If either \(E(X_k)\ne0\) or \(V(X_k)\ne0\), then
\(M_{k,j}\) is bounded away from zero for all sufficiently
large \(j\). Together with \(R_{k,j}\to0\),
\cref{iprox_transfer_lower_bound} then implies both
conditions in~\cref{iprox_primal_conditions}.
The uniform primal estimates above would make the line
search accept a primal step, contradicting the assumption.

Hence \(E(X_k)=0\) and \(V(X_k)=0\), so
\cref{prop:prox_wellposedness} gives first-order
stationarity. Moreover, \(V_{k,j}\to0\).
Consequently, for any prescribed tolerances
\(\epsilon_{\mathrm{stat}},\epsilon_{\mathrm{feas}}>0\),
either an outer step is accepted after finitely many
multiplier updates, or
\[
\frac{\|V(X_k,\Gamma_{k,j})\|_{\mathrm F}}{t}
\leq\epsilon_{\mathrm{stat}},
\qquad
\|X_k^\top X_k-I_p\|_{\mathrm F}
\leq\epsilon_{\mathrm{feas}}
\]
hold after finitely many multiplier updates. Thus, these conditions provide a finite stopping
criterion for practical implementations.

\proofstep{Global decrease and complexity.}
The acceptance test and~\cref{iprox_nonmonotone_reference} give
\(\overline\Phi_{k+1}\leq\overline\Phi_k
-c\eta_kM_k/a_{k+1}^{\mathrm{nm}}\). Since
\(a_{k+1}^{\mathrm{nm}}\leq(1-\omega_{\mathrm{mem}})^{-1}\) and
\(\eta_k\geq\underline\eta_{\mathrm{IP}}\),
\[
\overline\Phi_{k+1}
\leq
\overline\Phi_k-c(1-\omega_{\mathrm{mem}})\underline\eta_{\mathrm{IP}}M_k.
\]
Summing this inequality and using the lower bound
\(\overline\Phi_K\geq
\min_{X\in\St(n,p)^\varepsilon}\Phi_{\sigma_{\mathrm{IP}}}(X)\) gives
\begin{equation}
\label{iprox_global_decrease_sum}
\sum_{k=0}^{K-1}M_k\leq\Delta_{\mathrm{IP}}.
\end{equation}
Thus, \(M_k\to0\) whenever infinitely many outer iterations are completed.
The definition of \(M_k\) and the minimum--average inequality give
\[
\min_{0\leq k<K}
\left\{
\frac{\rho}{t}\|V_k\|_{\mathrm{F}}^2
+\rho\|E(X_k)\|_{\mathrm{F}}
\right\}
\leq
\frac{\Delta_{\mathrm{IP}}}{K}.
\]
Nonnegativity of its two terms gives
\cref{iprox_primal_best_iterate_rates}.

Before the feasibility tolerance is first attained, every completed outer
iteration starts from an \(X_k\) satisfying
\(\|E(X_k)\|_{\mathrm{F}}>\epsilon\), and hence
\(M_k\geq\rho\epsilon\). Therefore~\cref{iprox_global_decrease_sum} permits at
most \(\mathcal O(\epsilon^{-1})\) completed outer iterations before this
tolerance is attained. Likewise, before the stationarity tolerance is
attained,
\(\|V_k\|_{\mathrm{F}}/t>\epsilon\), so
\(M_k>\rho t\epsilon^2\), which permits at most
\(\mathcal O(\epsilon^{-2})\) completed outer iterations.

Finally, suppose that infinitely many outer iterates are generated. The
summability established above gives \(M_k\to0\), hence
\(E(X_k)\to0\) and
\(\|V_k\|_{\mathrm{F}}^2\leq tM_k/\rho\to0\).
If \(X_{k_i}\to\overline X\), boundedness of the accepted multipliers permits a
further subsequence on which
\(\Gamma_{k_i}\to\overline\Gamma\). The joint continuity established in
\cref{app:proof-prox-complexity} gives
\(E(\overline X)=0\) and \(V(\overline X,\overline\Gamma)=0\).
Therefore~\cref{prop:iprox_residual_exactness} shows that \(\overline X\) is
first-order stationary.
\end{proof}

\subsection{Proof of \texorpdfstring{\cref{thm:iprox_kl_convergence}}{Theorem~\ref{thm:iprox_kl_convergence}}}
\label{app:iprox_kl_convergence_proof}

\begin{proof}
Throughout this proof, write \(V_k=V(X_k,\Gamma_k)=V_{k,J_k}\).
Write \(\phi_k=\Phi_{\sigma_{\mathrm{IP}}}(X_k)\),
\(a_k=a_k^{\mathrm{nm}}\), and \(\omega=\omega_{\mathrm{mem}}\). From
\cref{iprox_nonmonotone_reference},
\begin{equation}
\label{iprox_zh_identities}
\delta_k^{\mathrm{nm}}
=\frac{\overline\Phi_k-\phi_{k+1}}{a_{k+1}},
\quad
\overline\Phi_k-\phi_k
=(a_k-1)\delta_{k-1}^{\mathrm{nm}},
\quad k\geq1.
\end{equation}
Indeed, the first identity follows from
\(a_{k+1}-1=\omega a_k\), and the second follows by applying the first one
at iteration \(k-1\). The recurrence for \(a_k\) also gives
\(1\leq a_k\leq(1-\omega)^{-1}\). Consequently, the acceptance test and
the uniform stepsize bound in~\cref{thm:iprox_complexity} yield
\begin{equation}
\label{iprox_delta_controls_M}
\delta_k^{\mathrm{nm}}
\geq c(1-\omega)\underline\eta_{\mathrm{IP}}M_k
=\gamma_{\mathrm{nm}}M_k,
\quad
\gamma_{\mathrm{nm}}=c(1-\omega)\underline\eta_{\mathrm{IP}}>0.
\end{equation}
Thus \(\{\overline\Phi_k\}\) is nonincreasing and, by
\cref{iprox_global_decrease_sum}, bounded below. Let
\(\overline\Phi_k\to\overline\Phi_\infty\). Then
\(\delta_k^{\mathrm{nm}}\to0\), \(M_k\to0\), and
\cref{iprox_zh_identities} gives \(\phi_k\to\overline\Phi_\infty\).
We average consecutive reference values so that the decrease involves
the same two terms as the Bregman-error bound.
For \(k\geq1\), define
\[
\Psi_k=\frac{\overline\Phi_{k-1}+\overline\Phi_k}{2},
\quad
s_k=\sqrt{\delta_{k-1}^{\mathrm{nm}}+\delta_k^{\mathrm{nm}}}.
\]
Then \(\Psi_k\downarrow\overline\Phi_\infty\) and
\begin{equation}
\label{iprox_averaged_energy_decrease}
\Psi_k-\Psi_{k+1}=\frac{s_k^2}{2}.
\end{equation}
Set
\[
\mathcal J_k=h(X_k)+r(Z_k)+\inner{\xi_k}{X_k-Z_k}
-\inner{\Gamma_k}{E_k}.
\]
By~\cref{iprox_bregman_error},
\begin{equation}
\label{iprox_lift_base_identity}
\mathcal J_k=f(X_k)-B_k-\inner{\Gamma_k}{E_k}.
\end{equation}
Let
\[
g_k^{\mathrm{lift}}=\Psi_k-\mathcal J_k,
\quad
u_k=\operatorname{sgn}(g_k^{\mathrm{lift}})
\sqrt{\lvert g_k^{\mathrm{lift}}\rvert},
\quad
\mathcal W_k=(X_k,Z_k,\Gamma_k,\xi_k,u_k),
\]
where \(\operatorname{sgn}(0)=0\). Since
\(u_k\lvert u_k\rvert=g_k^{\mathrm{lift}}\),
\begin{equation}
\label{iprox_lift_energy_identity}
\mathcal H_{\mathrm{IP}}(\mathcal W_k)=\Psi_k.
\end{equation}
Moreover,~\cref{iprox_zh_identities,iprox_lift_base_identity} give the exact
relation
\begin{equation}
\label{iprox_lift_gap_identity}
g_k^{\mathrm{lift}}=\left(a_k-\frac12\right)\delta_{k-1}^{\mathrm{nm}}
+\sigma_{\mathrm{IP}}\|E_k\|_{\mathrm{F}}+B_k
+\inner{\Gamma_k}{E_k}.
\end{equation}
Since \(\|E_k\|_{\mathrm{F}}\leq M_k/\rho
\leq\delta_k^{\mathrm{nm}}/(\rho\gamma_{\mathrm{nm}})\)
and \(\|\Gamma_k\|_{\mathrm{F}}\leq3R_{\mathrm{a}}\),
\cref{ass:iprox_bregman,iprox_lift_gap_identity} imply that, for all
sufficiently large \(k\),
\begin{equation}
\label{iprox_lift_gap_bound}
\lvert g_k^{\mathrm{lift}}\rvert\leq C_G s_k^2,
\quad
\lvert u_k\rvert\leq\sqrt{C_G}\,s_k,
\end{equation}
for some constant \(C_G>0\). 
The scalar function \(u\mapsto u\lvert u\rvert\) is continuously differentiable
on \(\R\), with derivative \(2\lvert u\rvert\), including at \(u=0\).
In the variable order \((X,Z,\Gamma,\xi,u)\), direct subdifferential
calculus gives
\[
\partial\mathcal H_{\mathrm{IP}}(X,Z,\Gamma,\xi,u)
=\left\{
\bigl(
\nabla h(X)+\xi-2X\Gamma,\,
\zeta-\xi,\,
-E(X),\,
X-Z,\,
2\lvert u\rvert
\bigr)
:\zeta\in\partial r(Z)
\right\}.
\]
Using~\cref{iprox_selected_subgradient} and choosing \(\zeta=\xi_k\) gives
\begin{equation}
\label{iprox_lift_subgradient}
\mathcal G_k=
\bigl(
-V_k/t,\,
0,\,
-E_k,\,
-V_k,\,
2\lvert u_k\rvert
\bigr)
\in\partial\mathcal H_{\mathrm{IP}}(\mathcal W_k).
\end{equation}
By~\cref{iprox_delta_controls_M}, \(\|V_k\|_{\mathrm{F}}^2
\leq tM_k/\rho
\leq t\delta_k^{\mathrm{nm}}/(\rho\gamma_{\mathrm{nm}})\).
The safe-region bound \(\|E_k\|_{\mathrm{F}}\leq\varepsilon\) also gives
\(\|E_k\|_{\mathrm{F}}^2
\leq\varepsilon\|E_k\|_{\mathrm{F}}
\leq\varepsilon\delta_k^{\mathrm{nm}}/(\rho\gamma_{\mathrm{nm}})\).
Together with~\cref{iprox_lift_gap_bound,iprox_lift_subgradient}, these
estimates yield
\begin{equation}
\label{iprox_lift_relative_error}
\dist\bigl(0,\partial\mathcal H_{\mathrm{IP}}(\mathcal W_k)\bigr)
\leq\|\mathcal G_k\|
\leq C_R s_k,
\end{equation}
for some constant \(C_R>0\) and all sufficiently large \(k\), where
\(\|\mathcal G_k\|\) is the Euclidean product norm.

The sequence \(\{\mathcal W_k\}\) is bounded. Indeed, \(X_k\) lies in the
compact safe region, \(V_k\to0\), and hence
\(Z_k=X_k+V_k\) is bounded. The
multipliers are bounded by the argument
in~\cref{app:iprox_complexity_proof}, while
\(\|\xi_k\|_{\mathrm{F}}\leq L_r\) follows from
\cref{nonsmooth_subgradient_bound}. Moreover, \cref{iprox_lift_gap_bound} bounds
\(u_k\). Let \(\Omega_{\mathrm{KL}}\) be the nonempty compact set of accumulation
points of \(\{\mathcal W_k\}\). Then
\(\dist(\mathcal W_k,\Omega_{\mathrm{KL}})\to0\). The lifted function is
continuous, and~\cref{iprox_lift_energy_identity} gives
\(\mathcal H_{\mathrm{IP}}(\mathcal W_k)=\Psi_k\to\overline\Phi_\infty\). Therefore,
\cref{iprox_lift_relative_error} and closedness of the limiting
subdifferential show that
\[
\mathcal H_{\mathrm{IP}}(W)=\overline\Phi_\infty,
\quad
0\in\partial\mathcal H_{\mathrm{IP}}(W)
\quad\text{for every }W\in\Omega_{\mathrm{KL}}.
\]
The uniformized KL inequality
\citep[Lemma~6, p.~478]{bolte2014proximal}, applied on \(\Omega_{\mathrm{KL}}\),
provides a concave desingularizing
function \(\varphi\) such that, for all sufficiently large \(k\) with
\(H_k=\Psi_k-\overline\Phi_\infty>0\),
\begin{equation}
\label{iprox_uniform_kl}
\varphi'(H_k)\dist\bigl(0,\partial\mathcal H_{\mathrm{IP}}(\mathcal W_k)\bigr)
\geq1.
\end{equation}
If \(H_k=0\) at some index,~\cref{iprox_averaged_energy_decrease} implies
\(s_k=0\) thereafter. Hence~\cref{iprox_delta_controls_M} gives
\(M_k=0\), and~\cref{iprox_direction_bounds} gives
\(X_{k+1}=X_k\) for all subsequent indices. Otherwise,
concavity of \(\varphi\),~\cref{iprox_averaged_energy_decrease}, and
\cref{iprox_lift_relative_error,iprox_uniform_kl} give
\[
\varphi(H_k)-\varphi(H_{k+1})
\geq\varphi'(H_k)(H_k-H_{k+1})
=\frac12\varphi'(H_k)s_k^2
\geq\frac{s_k}{2C_R}.
\]
Summation yields \(\sum_{k=1}^{\infty}s_k<\infty\). Since
\(M_k\leq\delta_k^{\mathrm{nm}}/\gamma_{\mathrm{nm}}\leq
s_k^2/\gamma_{\mathrm{nm}}\),
\begin{equation}
\label{iprox_sqrt_M_summability}
\sum_{k=0}^{\infty}\sqrt{M_k}<\infty.
\end{equation}
Finally,~\cref{iprox_direction_bounds} and \(\eta_k\leq\eta_0\) imply
\[
\|X_{k+1}-X_k\|_{\mathrm{F}}
=\eta_k\|\Lambda_{k,J_k}\|_{\mathrm{F}}
\leq\eta_0\sqrt{c_\Lambda M_k}.
\]
Thus~\cref{iprox_sqrt_M_summability} proves the finite-length assertion in
\cref{iprox_kl_finite_length}, and \(\{X_k\}\) converges to some \(X_\infty\).
Choose a subsequence \(\{k_i\}\) such that
\(\Gamma_{k_i}\to\Gamma_\infty\). Since the primal sequence
converges and \(M_k\to0\), continuity
(\cref{app:proof-prox-complexity}) gives \(E(X_\infty)=0\) and
\(V(X_\infty,\Gamma_\infty)=0\).
Thus \cref{prop:iprox_residual_exactness} yields the KKT conclusion.
\end{proof}


\begin{thebibliography}{39}
\providecommand{\natexlab}[1]{#1}
\providecommand{\url}[1]{\texttt{#1}}
\expandafter\ifx\csname urlstyle\endcsname\relax
  \providecommand{\doi}[1]{doi: #1}\else
  \providecommand{\doi}{doi: \begingroup \urlstyle{rm}\Url}\fi

\bibitem[Ablin and Peyr{\'e}(2022)]{ablin2022fast}
Pierre Ablin and Gabriel Peyr{\'e}.
\newblock Fast and accurate optimization on the orthogonal manifold without
  retraction.
\newblock In Gustau Camps-Valls, Francisco J.~R. Ruiz, and Isabel Valera,
  editors, \emph{Proceedings of the Twenty-Fifth International Conference on
  Artificial Intelligence and Statistics}, volume 151 of \emph{Proceedings of
  Machine Learning Research}, pages 5636--5657. PMLR, 2022.
\newblock URL \url{https://proceedings.mlr.press/v151/ablin22a.html}.

\bibitem[Ablin et~al.(2024)Ablin, Vary, Gao, and Absil]{ablin2024infeasible}
Pierre Ablin, Simon Vary, Bin Gao, and Pierre-Antoine Absil.
\newblock Infeasible deterministic, stochastic, and variance-reduction
  algorithms for optimization under orthogonality constraints.
\newblock \emph{Journal of Machine Learning Research}, 25\penalty0
  (389):\penalty0 1--38, 2024.
\newblock URL \url{https://www.jmlr.org/papers/v25/23-0451.html}.

\bibitem[Alon et~al.(1999)Alon, Barkai, Notterman, Gish, Ybarra, Mack, and
  Levine]{alon1999colon}
Uri Alon, Naama Barkai, Daniel~A. Notterman, Kurt Gish, Suzanne Ybarra,
  D.~Mack, and Arnold~J. Levine.
\newblock Broad patterns of gene expression revealed by clustering analysis of
  tumor and normal colon tissues probed by oligonucleotide arrays.
\newblock \emph{Proceedings of the National Academy of Sciences}, 96\penalty0
  (12):\penalty0 6745--6750, 1999.
\newblock \doi{10.1073/pnas.96.12.6745}.

\bibitem[Attouch et~al.(2013)Attouch, Bolte, and
  Svaiter]{attouch2013convergence}
Hedy Attouch, J{\'e}r{\^o}me Bolte, and Benar~Fux Svaiter.
\newblock Convergence of descent methods for semi-algebraic and tame problems:
  Proximal algorithms, forward--backward splitting, and regularized
  {Gauss--Seidel} methods.
\newblock \emph{Mathematical Programming}, 137\penalty0 (1--2):\penalty0
  91--129, 2013.
\newblock \doi{10.1007/s10107-011-0484-9}.

\bibitem[Barzilai and Borwein(1988)]{barzilai1988two}
Jonathan Barzilai and Jonathan~M. Borwein.
\newblock Two-point step size gradient methods.
\newblock \emph{IMA Journal of Numerical Analysis}, 8\penalty0 (1):\penalty0
  141--148, 1988.
\newblock \doi{10.1093/imanum/8.1.141}.

\bibitem[Bauschke and Combettes(2017)]{bauschke2017convex}
Heinz~H. Bauschke and Patrick~L. Combettes.
\newblock \emph{Convex Analysis and Monotone Operator Theory in Hilbert
  Spaces}.
\newblock {CMS} Books in Mathematics. Springer, Cham, 2nd edition, 2017.
\newblock \doi{10.1007/978-3-319-48311-5}.

\bibitem[Bolte et~al.(2014)Bolte, Sabach, and Teboulle]{bolte2014proximal}
J{\'e}r{\^o}me Bolte, Shoham Sabach, and Marc Teboulle.
\newblock Proximal alternating linearized minimization for nonconvex and
  nonsmooth problems.
\newblock \emph{Mathematical Programming}, 146\penalty0 (1--2):\penalty0
  459--494, 2014.
\newblock \doi{10.1007/s10107-013-0701-9}.

\bibitem[Burke and Ferris(1993)]{burke1993weak}
James~V. Burke and Michael~C. Ferris.
\newblock Weak sharp minima in mathematical programming.
\newblock \emph{SIAM Journal on Control and Optimization}, 31\penalty0 (5):\penalty0
  1340--1359, 1993.
\newblock \doi{10.1137/0331063}.

\bibitem[Chen et~al.(2020)Chen, Ma, So, and Zhang]{chen2020proximal}
Shixiang Chen, Shiqian Ma, Anthony Man-Cho So, and Tong Zhang.
\newblock Proximal gradient method for nonsmooth optimization over the
  {Stiefel} manifold.
\newblock \emph{SIAM Journal on Optimization}, 30\penalty0 (1):\penalty0
  210--239, 2020.
\newblock \doi{10.1137/18M122457X}.

\bibitem[Chen et~al.(2016)Chen, Ji, and You]{chen2016augmented}
Weiqiang Chen, Hui Ji, and Yanfei You.
\newblock An augmented {Lagrangian} method for {$\ell_1$}-regularized
  optimization problems with orthogonality constraints.
\newblock \emph{SIAM Journal on Scientific Computing}, 38\penalty0
  (4):\penalty0 B570--B592, 2016.
\newblock \doi{10.1137/140988875}.

\bibitem[Deng and Peng(2023)]{deng2023manifold}
Kangkang Deng and Zheng Peng.
\newblock A manifold inexact augmented {Lagrangian} method for nonsmooth
  optimization on {Riemannian} submanifolds in {Euclidean} space.
\newblock \emph{IMA Journal of Numerical Analysis}, 43\penalty0 (3):\penalty0
  1653--1684, 2023.
\newblock \doi{10.1093/imanum/drac018}.

\bibitem[Deng et~al.(2025)Deng, Hu, Wu, and Wen]{deng2025oracle}
Kangkang Deng, Jiang Hu, Jiayuan Wu, and Zaiwen Wen.
\newblock Oracle complexities of augmented {Lagrangian} methods for nonsmooth
  composite optimization on a compact submanifold.
\newblock \emph{Mathematics of Operations Research}, 2025.
\newblock \doi{10.1287/moor.2024.0498}.

\bibitem[Fadili et~al.(2023)Fadili, Nghia, and Tran]{fadili2023sharp}
Jalal Fadili, Tran T.~A. Nghia, and Trinh T.~T. Tran.
\newblock Sharp, strong and unique minimizers for low complexity robust
  recovery.
\newblock \emph{Information and Inference: A Journal of the IMA}, 12\penalty0
  (3):\penalty0 1461--1513, 2023.
\newblock \doi{10.1093/imaiai/iaad005}.

\bibitem[Gao et~al.(2019)Gao, Liu, and Yuan]{gao2019parallelizable}
Bin Gao, Xin Liu, and Ya-xiang Yuan.
\newblock Parallelizable algorithms for optimization problems with orthogonality constraints.
\newblock \emph{SIAM Journal on Scientific Computing}, 41(3):A1949--A1983, 2019.
\newblock \doi{10.1137/18M1221679}.

\bibitem[Gao et~al.(2022)Gao, Vary, Ablin, and Absil]{gao2022optimization}
Bin Gao, Simon Vary, Pierre Ablin, and Pierre-Antoine Absil.
\newblock Optimization flows landing on the {Stiefel} manifold.
\newblock \emph{IFAC-PapersOnLine}, 55\penalty0 (30):\penalty0 25--30, 2022.
\newblock \doi{10.1016/j.ifacol.2022.11.023}.
\newblock 25th IFAC Symposium on Mathematical Theory of Networks and Systems.

\bibitem[Goyens et~al.(2024)Goyens, Eftekhari, and Boumal]{goyens2024layered}
Florentin Goyens, Armin Eftekhari, and Nicolas Boumal.
\newblock Computing second-order points under equality constraints: revisiting Fletcher's augmented Lagrangian.
\newblock \emph{Journal of Optimization Theory and Applications}, 201(3):1198--1228, 2024.
\newblock \doi{10.1007/s10957-024-02421-6}.

\bibitem[Hu et~al.(2024)Hu, Xiao, Liu, and Toh]{hu2024constraint}
Xiaoyin Hu, Nachuan Xiao, Xin Liu, and Kim-Chuan Toh.
\newblock A constraint dissolving approach for nonsmooth optimization over the
  {Stiefel} manifold.
\newblock \emph{IMA Journal of Numerical Analysis}, 44\penalty0 (6):\penalty0
  3717--3748, 2024.
\newblock \doi{10.1093/imanum/drad098}.

\bibitem[Huang and Wei(2022)]{huang2022riemannian}
Wen Huang and Ke~Wei.
\newblock {Riemannian} proximal gradient methods.
\newblock \emph{Mathematical Programming}, 194\penalty0 (1--2):\penalty0
  371--413, 2022.
\newblock \doi{10.1007/s10107-021-01632-3}.

\bibitem[Huang and Wei(2023)]{huang2023inexact}
Wen Huang and Ke~Wei.
\newblock An inexact {Riemannian} proximal gradient method.
\newblock \emph{Computational Optimization and Applications}, 85\penalty0
  (1):\penalty0 1--32, 2023.
\newblock \doi{10.1007/s10589-023-00451-w}.

\bibitem[Jolliffe et~al.(2003)Jolliffe, Trendafilov, and
  Uddin]{jolliffe2003modified}
Ian~T. Jolliffe, Nickolay~T. Trendafilov, and Mudassir Uddin.
\newblock A modified principal component technique based on the {LASSO}.
\newblock \emph{Journal of Computational and Graphical Statistics}, 12\penalty0
  (3):\penalty0 531--547, 2003.
\newblock \doi{10.1198/1061860032148}.

\bibitem[Karkhaneei and Mahdavi-Amiri(2019)]{karkhaneei2019nonconvex}
Mohammad~Mahdi Karkhaneei and Nezam Mahdavi-Amiri.
\newblock Nonconvex weak sharp minima on {Riemannian} manifolds.
\newblock \emph{Journal of Optimization Theory and Applications}, 183\penalty0 (1):\penalty0
  85--104, 2019.
\newblock \doi{10.1007/s10957-019-01539-2}.

\bibitem[Kovnatsky et~al.(2016)Kovnatsky, Glashoff, and
  Bronstein]{kovnatsky2016madmm}
Artiom Kovnatsky, Klaus Glashoff, and Michael~M. Bronstein.
\newblock {MADMM}: A generic algorithm for non-smooth optimization on manifolds.
\newblock In Bastian Leibe, Jiri Matas, Nicu Sebe, and Max Welling, editors,
  \emph{Computer Vision -- ECCV 2016}, volume 9909 of \emph{Lecture Notes in
  Computer Science}, pages 680--696, Cham, 2016. Springer.
\newblock \doi{10.1007/978-3-319-46454-1_41}.

\bibitem[Lai and Osher(2014)]{lai2014splitting}
Rongjie Lai and Stanley Osher.
\newblock A splitting method for orthogonality constrained problems.
\newblock \emph{Journal of Scientific Computing}, 58\penalty0 (2):\penalty0
  431--449, 2014.
\newblock \doi{10.1007/s10915-013-9740-x}.

\bibitem[Lerman and Maunu(2018)]{lerman2018overview}
Gilad Lerman and Tyler Maunu.
\newblock An overview of robust subspace recovery.
\newblock \emph{Proceedings of the IEEE}, 106\penalty0 (8):\penalty0
  1380--1410, 2018.
\newblock \doi{10.1109/JPROC.2018.2853141}.

\bibitem[Li et~al.(2011)Li, Mordukhovich, Wang, and Yao]{li2011weak}
Chong Li, Boris~S. Mordukhovich, Jinhua Wang, and Jen-Chih Yao.
\newblock Weak sharp minima on {Riemannian} manifolds.
\newblock \emph{SIAM Journal on Optimization}, 21\penalty0 (4):\penalty0 1523--1560, 2011.
\newblock \doi{10.1137/09075367X}.

\bibitem[Li et~al.(2025)Li, Ma, and Srivastava]{li2025riemannian}
Jiaxiang Li, Shiqian Ma, and Tejes Srivastava.
\newblock A {Riemannian} alternating direction method of multipliers.
\newblock \emph{Mathematics of Operations Research}, 50\penalty0 (4):\penalty0
  3222--3242, 2025.
\newblock \doi{10.1287/moor.2023.0068}.

\bibitem[Li et~al.(2021)Li, Chen, Deng, Qu, Zhu, and So]{li2021weakly}
Xiao Li, Shixiang Chen, Zengde Deng, Qing Qu, Zhihui Zhu, and Anthony Man-Cho
  So.
\newblock Weakly convex optimization over {Stiefel} manifold using
  {Riemannian} subgradient-type methods.
\newblock \emph{SIAM Journal on Optimization}, 31\penalty0 (3):\penalty0
  1605--1634, 2021.
\newblock \doi{10.1137/20M1321000}.

\bibitem[Liu et~al.(2024)Liu, Xiao, and Yuan]{liu2024penaltyfree}
Xin Liu, Nachuan Xiao, and Ya-xiang Yuan.
\newblock A penalty-free infeasible approach for a class of nonsmooth
  optimization problems over the {Stiefel} manifold.
\newblock \emph{Journal of Scientific Computing}, 99\penalty0 (2):\penalty0 30,
  2024.
\newblock \doi{10.1007/s10915-024-02495-4}.

\bibitem[Ozoli{\c{n}}{\v{s}} et~al.(2013)Ozoli{\c{n}}{\v{s}}, Lai, Caflisch,
  and Osher]{ozolicnvs2013compressed}
Vidvuds Ozoli{\c{n}}{\v{s}}, Rongjie Lai, Russel Caflisch, and Stanley Osher.
\newblock Compressed modes for variational problems in mathematics and physics.
\newblock \emph{Proceedings of the National Academy of Sciences}, 110\penalty0
  (46):\penalty0 18368--18373, 2013.
\newblock \doi{10.1073/pnas.1318679110}.

\bibitem[Qi and Sun(1993)]{qi1993nonsmooth}
Liqun Qi and Jie Sun.
\newblock A nonsmooth version of {Newton}'s method.
\newblock \emph{Mathematical Programming}, 58:\penalty0 353--367, 1993.
\newblock \doi{10.1007/BF01581275}.

\bibitem[Rockafellar and Wets(1998)]{rockafellar1998variational}
Ralph~Tyrrell Rockafellar and Roger J.-B. Wets.
\newblock \emph{Variational Analysis}, volume 317 of \emph{Grundlehren der
  Mathematischen Wissenschaften}.
\newblock Springer, Berlin, 1998.
\newblock \doi{10.1007/978-3-642-02431-3}.

\bibitem[Schechtman et~al.(2023)Schechtman, Tiapkin, Muehlebach, and Moulines]{schechtman2023orthogonal}
Sholom Schechtman, Daniil Tiapkin, Michael Muehlebach, and {\'E}ric Moulines.
\newblock Orthogonal directions constrained gradient method: from non-linear
  equality constraints to {Stiefel} manifold.
\newblock In Gergely Neu and Lorenzo Rosasco, editors,
  \emph{Proceedings of Thirty Sixth Conference on Learning Theory},
  volume 195 of \emph{Proceedings of Machine Learning Research},
  pages 1228--1258. PMLR, 2023.
\newblock URL \url{https://proceedings.mlr.press/v195/schechtman23a.html}.

\bibitem[Thiao et~al.(2010)Thiao, Pham~Dinh, and Le~Thi]{thiao2010dc}
Mamadou Thiao, Tao Pham~Dinh, and Hoai~An Le~Thi.
\newblock A {DC} programming approach for sparse eigenvalue problem.
\newblock In \emph{Proceedings of the 27th International Conference on Machine
  Learning}, pages 1063--1070, Haifa, Israel, 2010. Omnipress.
\newblock URL \url{https://icml.cc/Conferences/2010/papers/359.pdf}.

\bibitem[Xiao et~al.(2021)Xiao, Liu, and Yuan]{xiao2021exact}
Nachuan Xiao, Xin Liu, and Ya-xiang Yuan.
\newblock Exact penalty function for {$\ell_{2,1}$}-norm minimization over the
  {Stiefel} manifold.
\newblock \emph{SIAM Journal on Optimization}, 31\penalty0 (4):\penalty0
  3097--3126, 2021.
\newblock \doi{10.1137/20M1354313}.

\bibitem[Xiao et~al.(2018)Xiao, Li, Wen, and Zhang]{xiao2018regularized}
Xiantao Xiao, Yongfeng Li, Zaiwen Wen, and Liwei Zhang.
\newblock A regularized semi-smooth {Newton} method with projection steps for
  composite convex programs.
\newblock \emph{Journal of Scientific Computing}, 76\penalty0 (1):\penalty0
  364--389, 2018.
\newblock \doi{10.1007/s10915-017-0624-3}.

\bibitem[Xie and Li(2026)]{xie2026fsipl}
Xiyuan Xie and Qia Li.
\newblock An infeasible method with feasibility safeguard for nonsmooth
  composite optimization over manifolds, 2026.
\newblock arXiv:2605.24954v2.

\bibitem[Xu et~al.(2025)Xu, Jiang, Liu, and So]{xu2025oracle}
Meng Xu, Bo~Jiang, Ya-Feng Liu, and Anthony Man-Cho So.
\newblock On the oracle complexity of a {Riemannian} inexact augmented
  {Lagrangian} method for nonsmooth composite problems over {Riemannian}
  submanifolds.
\newblock \emph{Optimization Letters}, 2025.
\newblock \doi{10.1007/s11590-025-02210-8}.

\bibitem[Yang et~al.(2011)Yang, Shen, Ma, Huang, and Zhou]{yang2011}
Yi~Yang, Heng~Tao Shen, Zhigang Ma, Zi~Huang, and Xiaofang Zhou.
\newblock $\ell_{2,1}$-norm regularized discriminative feature selection for
  unsupervised learning.
\newblock In \emph{Proceedings of the Twenty-Second International Joint
  Conference on Artificial Intelligence}, pages 1589--1594, 2011.
\newblock \doi{10.5591/978-1-57735-516-8/IJCAI11-267}.

\bibitem[Zhang and Hager(2004)]{zhang2004nonmonotone}
Hongchao Zhang and William~W. Hager.
\newblock A nonmonotone line search technique and its application to
  unconstrained optimization.
\newblock \emph{SIAM Journal on Optimization}, 14\penalty0 (4):\penalty0
  1043--1056, 2004.
\newblock \doi{10.1137/S1052623403428208}.

\bibitem[Zhou et~al.(2023)Zhou, Bao, Ding, and Zhu]{zhou2021semismooth}
Yuhao Zhou, Chenglong Bao, Chao Ding, and Jun Zhu.
\newblock A semismooth {Newton} based augmented {Lagrangian} method for
  nonsmooth optimization on matrix manifolds.
\newblock \emph{Mathematical Programming}, 201\penalty0 (1--2):\penalty0 1--61,
  2023.
\newblock \doi{10.1007/s10107-022-01898-1}.

\bibitem[Zhu et~al.(2026)Zhu, Ding, Liu, and So]{zhu2026primal}
Linglingzhi Zhu, Wentao Ding, Shangyuan Liu, and Anthony Man-Cho So.
\newblock Primal-dual methods for nonsmooth nonconvex optimization with
  orthogonality constraints, 2026.
\newblock arXiv:2604.04130.

\end{thebibliography}
\end{document}